\documentclass[12pt]{article}

\usepackage[utf8]{inputenc}
\usepackage[T1]{fontenc}
\usepackage[english]{babel}
\usepackage{lmodern}
\usepackage{geometry}

\usepackage[colorlinks,allcolors=blue]{hyperref}
\usepackage{multirow}
\usepackage{xspace}
\usepackage{color}
\usepackage{graphicx}
\usepackage{booktabs}
\usepackage{comment}
\usepackage{algorithm}
\usepackage{algpseudocode}

\input{my.sty}

\usepackage[bottom]{footmisc}
\usepackage{amsthm}
\usepackage{amssymb}
\usepackage{datetime}
\usepackage{tikz}
\usetikzlibrary{patterns}

\newdateformat{mydate}{\THEDAY \ \monthname[\THEMONTH] \THEYEAR}

\definecolor{myred}{HTML}{FD6360}
\definecolor{myblue}{HTML}{46A5FF}

\theoremstyle{definition}
\newtheorem{theorem}{Theorem}
\newtheorem{lemma}{Lemma}[section]

\newtheorem{remark}{Remark}
\newtheorem{example}{Example}
\newtheorem{corollary}{Corollary}
\newtheorem{assumption}{Assumption}
\newtheorem{definition}{Definition}[section]
\newtheorem{proposition}{Proposition}

\numberwithin{equation}{section}

\newcommand{\given}{\,\vert\,}

\newenvironment{keywords}
 {\textbf{Keywords:}}
 {}

\usepackage{authblk}
\renewcommand*{\Affilfont}{\normalsize\normalfont}
\title{Consistent line clustering using geometric hypergraphs\footnote{Supported by the Vilho, Yrj\"{o} and Kalle V\"{a}is\"{a}l\"{a} Foundation of the Finnish Academy of Science and Letters, and the French government, through the France 2030 investment plan managed by the Agence Nationale de la Recherche, as part of the Universit\'e C\^{o}te d'Azur's Initiative of Excellence, reference ANR-15-IDEX-0001.}}
\author[1]{Kalle Alaluusua}
\author[2]{Konstantin Avrachenkov}
\author[3]{B. R. Vinay Kumar}
\author[1]{Lasse Leskel\"{a}}
\affil[1]{ Aalto University, Espoo, Finland}
\affil[ ]{\Affilfont \{\href{mailto:kalle.alaluusua@aalto.fi}{kalle.alaluusua}, \href{mailto:lasse.leskela@aalto.fi}{lasse.leskela}\}@aalto.fi}
\affil[2]{ INRIA, Sophia Antipolis, Valbonne, France}
\affil[ ]{\Affilfont \href{mailto:k.avrachenkov@inria.fr}{k.avrachenkov@inria.fr}}
\affil[3]{Indian Institute of Technology Bombay (IITB), Mumbai, India}
\affil[ ]{\Affilfont \href{mailto:vinaykumar.br@iitb.ac.in}{vinaykumar.br@iitb.ac.in}}
\renewcommand\Affilfont{\itshape\small}

\begin{document}

\mydate

\maketitle
\thispagestyle{plain}
\pagestyle{plain}

%\begin{center}
%\rnote{Current: Sec~\ref{sec:InterCommunityHyperedgeProbability}, Lemma \ref{lem:prob_between}}
%\end{center}

\begin{abstract}

Subspace clustering becomes inherently difficult near intersections, where points from different subspaces are barely separated. Most existing theoretical results address this issue by imposing separation or sampling assumptions that limit the statistical effect of points near the intersection. We study a minimal setting of two intersecting lines in which the latent sampling law places polynomially large mass in small neighborhoods of the intersection. We derive information-theoretic lower bounds for exact and almost exact recovery under Gaussian noise.  In particular, we show that the exact-recovery threshold is determined by the rate at which the latent law concentrates near the intersection.  Since any two points are collinear, pairwise information alone does not reveal whether they are sampled from the same latent line. We therefore construct a hypergraph in which nearly collinear triples form hyperedges, and study the resulting hypergraph similarity matrix. Under a simple regularity condition on the latent distribution, we introduce a spectral algorithm that achieves the information-theoretic bounds up to polylogarithmic factors.
\end{abstract}

\tableofcontents

\section{Introduction}

Subspace clustering seeks to partition observations according to the low-dimensional subspaces from which they were generated \cite{vidal2011survey}. It becomes intrinsically difficult when the subspaces intersect, since observations near the intersection carry little information about their latent membership. We study the simplest nontrivial instance of this phenomenon: points sampled from two intersecting lines in the plane with additive Gaussian noise.

Because the two lines intersect, the data form a single connected high-density region, so standard density-based clustering methods designed for well-separated, roughly convex clusters are not well suited to this setting. Moreover, pairwise geometry is uninformative, since any two distinct points determine a line \cite{agarwal2005beyond}. By contrast, triples sampled from the same latent line are more likely to be approximately collinear than mixed triples. This motivates a hypergraph representation of the data, in which approximately collinear triples form hyperedges. Subspace clustering can then be viewed as a community-recovery problem on the similarity matrix induced by this geometric hypergraph.

\subsection{Statistical model}
We formalize the line-clustering problem as follows. Fix a community membership vector \(\bz=(z_1,\dots,z_N)\in\{1,2\}^N\), an angle
\(\alpha\in(0,\pi)\), a probability measure \(\nu\) on \(\R\) that is compactly supported and symmetric
about \(0\), and a noise parameter \(\sigma>0\).
Let \(L_1,L_2\subset\R^2\) be lines through the origin that intersect at angle \(\alpha\), and let
\(\bv_1,\bv_2\in\R^2\) be unit vectors spanning \(L_1\) and \(L_2\), respectively.
We observe points
\begin{equation}
 \label{eq:model}
 \bX_i = V_i \bv_{z_i} + \sigma \bY_i,
 \qquad i=1,\dots,N,
\end{equation}
where \(V_i \eqd \nu\), \(\bY_i \eqd N(0,I_2)\), and all random variables appearing on the
right-hand side of \eqref{eq:model} are mutually independent.
The conditional distribution of \(\bX=(\bX_1,\dots,\bX_N)\) given the label vector \(\bz\) is denoted by
\[
 \GLMM_N(\alpha,\nu,\sigma\mid \bz),
\]
which we refer to as the Gaussian line mixture model. To study the average-case performance of an estimator, we further assume that the community memberships
\(\bZ=(Z_1,\dots,Z_N)\) are independent random variables uniformly sampled from \(\{1,2\}\), independently
of \((V_i,\bY_i)_{i=1}^N\), and we then write
\[
 (\bZ,\bX)\eqd \GLMM_N(\alpha,\nu,\sigma).
\]
Given \(\bz\in\{1,2\}^N\) and observations \(\bX\sim \GLMM_N(\alpha,\nu,\sigma\mid \bz)\), the goal is to estimate the community membership vector \(\bz\) from the data \(\bX\).

Intersections are the main source of statistical difficulty in this model. Points generated near the origin have a small geometric margin separating the two lines, so even moderate noise can render their labels ambiguous. To describe this effect, we impose lower and upper mass-distribution conditions on the common one-dimensional sampling law \(\nu\).

\begin{assumption}[Concentration of mass at the intersection]
\label{ass:nu_lower_smallball}
There exists \(\rho\in(0,1]\) such that
\[
\liminf_{u\downarrow 0}\frac{\nu([-u,u])}{u^\rho}>0.
\]
\end{assumption}

Assumption~\ref{ass:nu_lower_smallball} says that the sampling law places non-negligible mass near the origin. This is the information-theoretically interesting part of the model: if a non-negligible fraction of the latent points lie near the intersection, then points from different communities can still form nearly collinear triples, which creates genuinely ambiguous configurations and leads to lower bounds for \emph{any} recovery procedure.

The exponent \(\rho\) quantifies the strength of concentration near the intersection. Smaller values of \(\rho\) allow stronger concentration near zero, while \(\rho=1\) corresponds to linear scaling, as would occur for a law whose density is locally bounded below near the origin.

\begin{assumption}[Upper \(\rho\)-regularity]
\label{ass:nu_upper_smallball}
There exists \(\rho\in(0,1]\) such that
\[
\sup_{I\subset\R,\ |I|>0}\frac{\nu(I)}{|I|^\rho}<\infty,
\]
where the supremum is taken over all bounded intervals \(I\subset\R\).
\end{assumption}

Assumption~\ref{ass:nu_upper_smallball} is a uniform anti-concentration condition. It prevents the latent law from placing excessive mass on arbitrarily short intervals. In the analysis, it is used to control both the probability that a latent point falls very close to the intersection and the probability that two independent latent coordinates fall unusually close to one another. For general \(\rho\in(0,1]\), Assumption~\ref{ass:nu_upper_smallball} means that the cumulative distribution
function of \(\nu\) is H\"older continuous of order \(\rho\). In particular, when \(0<\rho<1\), this is strictly weaker than a bounded-density assumption and includes absolutely continuous laws with integrable singularities; see Example~\ref{ex:power-law}.

In the special case \(\rho=1\), this reduces to Lipschitz continuity. If \(\nu\) is absolutely continuous, this is equivalent to \(\nu\) admitting a density that is bounded above. Together with Assumption~\ref{ass:nu_lower_smallball}, the case \(\rho=1\) therefore corresponds to a law whose density is bounded above globally and bounded below near the origin.

\begin{example}    
\label{ex:power-law}
A natural example is the power-law density
\[
\nu(du)=\frac{\rho}{2}|u|^{\rho-1}\1\{|u|\le 1\}\,du,
\qquad 0<\rho\le 1.
\]
For this law,
\(
\nu([-u,u])=u^\rho
\)
for all \(0\le u\le 1,\)
so Assumptions~\ref{ass:nu_lower_smallball} and \ref{ass:nu_upper_smallball} hold. When \(\rho=1\), this reduces to the uniform distribution on \([-1,1]\). When \(0<\rho<1\), the density has an integrable singularity at the origin, so the sampling law places more mass near the intersection.
\end{example}

Taken together, Assumptions~\ref{ass:nu_lower_smallball} and \ref{ass:nu_upper_smallball} describe the local regularity of the latent sampling law near the intersection. The lower condition governs the intrinsic statistical difficulty of the problem, whereas the upper condition provides the regularity needed for the concentration and separation estimates used in the recovery analysis.
\subsection{Related work}

While the subspace clustering literature is broad, a prevalent approach is to first construct an affinity graph from the data and then apply spectral clustering to recover the partition \cite{qu2023survey}. In contemporary formulations, one typically builds this affinity either from a learned representation matrix or from local or higher-order geometric relations \cite{miao2025survey}. The line of work closest to our construction comes from hybrid linear modeling \cite{chen2009multiway,chen2009scc}. These works emphasize that when local pairwise geometry is uninformative, one should build affinities from higher-order tuples. In the one-dimensional case, the \((d+2)\)-tuple affinities of \cite{chen2009multiway} reduce naturally to triples, and an experiment of \cite{agarwal2005beyond} illustrates the necessity of three-way affinities for line clustering.

From a theoretical perspective, intersections are usually handled indirectly. Standard surveys already note that subspace clustering becomes harder when many observations lie near an intersection \cite{vidal2011survey}. Much of the pointwise recovery literature, therefore, considers regimes where the ambiguous intersection region is either excluded or statistically negligible. Concretely, some guarantees impose disjointness \cite{elhamifar2010disjoint}, while others allow intersecting subspaces but study the similarity graph induced by a learned affinity matrix. A common proof strategy is then to distinguish three questions: whether this graph is subspace-preserving, meaning that it contains no edges between different underlying subspaces; whether the portion of the graph within each true subspace is connected; and whether these properties imply correct pointwise clustering, possibly after a post-processing step \cite{nasihatkon2011connectivity,soltanolkotabi2012geometric,wang2016graph}. In noisy settings, the same template is combined with quantitative conditions such as small inter-subspace affinity (for example, separation in principal angles), sufficiently well-spread samples within each subspace (for example, i.i.d. sampling and large \emph{inradius} conditions), and noise levels below thresholds set by these geometric quantities \cite{heckel2015tsc,soltanolkotabi2014robust,wang2016noisyssc}. Thus, while some works obtain pointwise guarantees even when subspaces intersect \cite{soltanolkotabi2012geometric,wang2016graph}, much of the theory emphasizes representation quality \cite{soltanolkotabi2014robust,wang2016noisyssc}, graph connectivity \cite{nasihatkon2011connectivity}, or subspace estimation \cite{liu2013lrr}, rather than information-theoretic thresholds for exact pointwise clustering. Such exact pointwise clustering can also be established for simpler pairwise affinities, such as thresholded dot product graphs, under appropriate separation and noise conditions \cite{heckel2015tsc}. In hybrid linear modeling, \cite{chen2009multiway,chen2009scc} consider tuple-based affinities rather than pairwise ones, but existing recovery results are still not formulated in terms of a sampling law whose local mass near the intersection determines exact-recovery thresholds.

Information-theoretic ideas have appeared in several distinct forms in the subspace-clustering literature. In the missing-data setting, \cite{pimentel2016information} studies identifiability and sampling requirements for subspace clustering with missing entries, deriving deterministic sampling conditions and showing that the underlying subspaces can be recovered from incomplete observations under suitable generic-position and random-observation assumptions. The closest work to ours in spirit is \cite{ahn2017information}, which studies exact subspace clustering from noisy sampled similarity measurements across data points, reformulates the problem as community recovery in hypergraphs, and derives a sharp threshold on the number of observed similarities required for exact recovery. By contrast, \cite{he2015information} uses information-theoretic objective functions, based on R\'enyi-type entropy and correntropy, to design robust low-rank representation methods for outlier-corrupted data, while \cite{hubig2017information} uses entropy, variation of information, and minimum-description-length ideas to control redundancy in coordinate-subspace clustering; these are algorithmic uses of information theory rather than information-theoretic limit results. Likewise, \cite{wang2018theoretical} analyzes noisy sparse subspace clustering after dimensionality reduction through geometric quantities such as subspace incoherence and inradius. Its main results are perturbation-based success conditions under deterministic and random models, while its impossibility result appears only in the privacy-preserving extension. None of these works addresses our regime of direct Euclidean observations from intersecting subspaces in which the latent law may place polynomially large mass near the intersection itself. In our setting, the dominant statistical obstruction is not missing data, similarity-sampling complexity, or dimensionality reduction, but the local small-ball behavior of the latent sampling law near the intersection, which directly determines the exact-recovery threshold.

Other subspace clustering methods such as generalizations of \(K\)-means algorithms \cite{bradley2000kplane,lipor2021ekss,wang2022kss} provide useful reference points but are less closely aligned with the mechanism of our procedure. Local geometric methods estimate local tangent or flat structure and then aggregate these estimates into a global partition, as in local best-fit flats and spectral methods based on local linear approximations or local principal component analysis \cite{ariascastro2011locallinear,ariascastro2017localpca, zhang2012lbf}. A geometric \(\ell_p\) minimization is analyzed primarily as a method for recovering the underlying subspaces themselves rather than exact pointwise labels \cite{lerman2011lp}.

On the hypergraph side, community recovery in planted hypergraph models is well understood across a range of regimes  \cite{ahn2019community,Chien_Lin_Wang_2019,dumitriu2025partial,zhang2022exact}. Of particular relevance are algorithms based on the similarity matrix \(\bW\), where \(W_{ij}\) counts the number of hyperedges containing both \(i\) and \(j\); this is exactly the matrix used by our spectral procedure \cite{alaluusua2023multilayer,dumitriu2026optimal,gaudio2023community}. The works \cite{ahn_lee_suh_2018,brusa2024model,ghoshdastidar_dukkipati_2015_AAAI,leordeanuEfficientHypergraphClustering2012,kaminski2019clustering} already point toward subspace clustering applications, while \cite{wang2023information} gives sharp exact recovery guarantees from \(\bW\) in the hypergraph SBM. Our setting differs in two essential ways. First, the hypergraph is not induced by a latent block model but instead obtained from noisy Euclidean geometry by thresholding total least squares (TLS) residuals of triples. Second, the resulting hyperedges are strongly dependent whenever they intersect. This contrasts in particular with block models, where edges are conditionally independent given the latent labels \cite{brusa2024model}. Our contribution is to show that similarity-matrix spectral methods remain near-optimal even in this dependent, intersection-dominated geometric setting.

\subsection{Contributions}
Our contributions are as follows:
\begin{enumerate}
\item We identify the statistical effect of the subspace intersection, when recovering \(\bZ\) from \(\bX\) with \((\bZ,\bX)\eqd \GLMM_N(\alpha,\nu,\sigma_N)\). If the latent coordinate law \(\nu\) places non-negligible mass near zero with exponent \(\rho\in(0,1]\), then exact recovery can occur only if \(\sigma_N\ll N^{-1/\rho}\), while almost exact recovery can occur only if \(\sigma_N\ll 1\). Thus, the exact recovery threshold is governed by the polynomial concentration of \(\nu\) near the intersection.

\item We propose a polynomial-time spectral method based on a geometric hypergraph whose hyperedges are
defined by thresholding total least squares (TLS) residuals over triples. When the common coordinate law
is sufficiently regular, this method achieves exact recovery when
\[
 \sigma_N \ll N^{-1/\rho}\polylog(N)^{-1},
\]
thereby matching the information-theoretic lower bound up to a polylogarithmic factor.
We also provide a data-driven quantile rule for selecting the threshold \(t\) based on the empirical distribution of the triple-wise TLS residuals. This procedure requires no prior knowledge of \(\sigma_N\) and achieves almost exact recovery whenever \(\sigma_N\ll 1\).
\end{enumerate}
In addition, we show that the coordinatewise MAP estimator, which is optimal for the usual Hamming loss, remains Bayes optimal
under permutation-invariant Hamming loss.

We also present a simulation study on synthetic two-line data. The first experiment isolates the effect of the hyperedge construction, comparing a TLS-based spectral method with the hypergraph clustering approach of \cite{brusa2024model}. The second isolates the effect of thresholding by comparing our data-driven TLS hypergraph with the clique-averaging procedure of \cite{agarwal2005beyond}. In these experiments, the TLS-based construction performs at least as well as the the alternatives, suggesting that the choice of geometric hyperedge construction is at least as important as the downstream clustering routine.

Taken together, these results place the Gaussian line mixture model at the intersection of two literatures: subspace clustering near intersections and hypergraph community recovery. They show that a simple geometric hypergraph representation can convert an intersecting-subspace problem into a dependent hypergraph recovery problem without losing statistical efficiency, up to polylogarithmic factors. Extending this picture to higher-dimensional or more general intersection-dominated subspace models remains an interesting open problem.

\paragraph{Notation.}
For positive number sequences $(a_N)$ and $(b_N)$, we write
$a_N \ll b_N$ when $\lim_{N\to\infty} a_N/b_N = 0$,
$a_N \lesssim b_N$ when $\limsup_{N\to\infty} a_N/b_N < \infty$,
and
$a_N \asymp b_N$ when $a_N \lesssim b_N$ and $a_N \gtrsim b_N$.
We also write
$a_N = o(b_N)$ and $b_N = \omega(a_N)$ as synonyms for $a_N \ll b_N$,
$a_N = O(b_N)$ and $b_N = \Omega(a_N)$ as synonyms for $a_N \lesssim b_N$,
and
$a_N = \Theta(b_N)$ as a synonym for $a_N \asymp b_N$.
For numbers $a,b$, we write $a \land b = \min\{a,b\}$ and
$a \lor b = \max\{a,b\}$.
Boldface letters denote vectors and matrices.
We denote equality in distribution by $X \eqd Y$, and let $X \eqd \mu$ mean that
the random variable $X$ has law $\mu$. 
Finally, we say that an event holds \emph{with high probability} if its probability converges to one as $N \to \infty$.
%Finally, ''with high probability'' means ''with probability tending to one as $N \to \infty$''.

The remainder of the paper is organized as follows. Section~\ref{sec:information-theoretic_bounds} states our information-theoretic limits for the Gaussian line mixture model (Theorem~\ref{thm:information_theory}). Section~\ref{sec:algorithms} introduces the geometric hypergraph-based spectral clustering algorithms and their near-optimal exact and almost exact recovery guarantees (Theorems \ref{thm:almost_main} and \ref{thm:main_nonparametric}). Section~\ref{sec:simulation} presents a simulation study on synthetic two-line data, comparing our TLS-based spectral method with existing line clustering approaches. Section~\ref{sec:proofs} contains the proofs, with additional details deferred to the appendices.

\section{Information-theoretic bounds}
\label{sec:information-theoretic_bounds}

We seek to recover the community membership vector $\bZ$ from the observed data $\bX$
when the pair $(\bZ,\bX)$ is sampled from the Gaussian line mixture model $\GLMM_N(\alpha,\nu,\sigma)$.
In this setting, a \emph{community membership estimator} is a measurable function
$\bt \colon (\R^2)^N \mapsto \{1,2\}^N$
that maps a list of points $\bX$ into a community membership vector $\hat \bZ = \bt(\bX)$. 
To assess the accuracy of such estimators, a natural choice is to use the Hamming distance
$
 \Ham(\bZ, \hat \bZ) = \sum_{i=1}^N 1(\hat Z_i \ne Z_i).
$
When viewing community membership vectors as partitions of $[N]$, we use the permutation-invariant Hamming distance
\begin{equation}
 \label{def:ham_star}
 \Ham^*(\bZ,\hat\bZ)
 \weq \min\left\{\Ham(\bZ,\hat\bZ), \ N - \Ham(\bZ,\hat\bZ) \right\},
 %\weq \min_{\pi \in S_2} \sum_{i=1}^N \1\{\hat{z}_i \ne \pi(z_i)\},
\end{equation}
which is a non-normalized version of the standard partition metric
known as classification error or misclassification rate \cite{Abbe_Fan_Wang_Zhong_2020,Meila_2007}.
% Meila_2005, Meila_2007: "classification error"
% Lei_Rinaldo_2015: "overall relative error"
% Zhang_Zhou_2016: "mis-match ratio"
% Gao_Ma_Zhang_Zhou_2017: "misclassification proportion"
% Abbe_2008: "1-agreement"
% Abbe_Fan_Wang_Zhong_2020 "misclassification rate"
We say that a community membership estimator $\bt$ achieves
\emph{exact recovery} if
\begin{equation}
 \label{eq:ExactRecovery}
 \lim_{N\to\infty} \E \left[\Ham^*(\bt(\bX),\bZ)\right] \weq 0,
\end{equation}
and \emph{almost exact recovery} if
\begin{equation}
 \label{eq:AlmostExactRecovery}
 \lim_{N \to\infty} N^{-1} \E \left[\Ham^*(\bt(\bX),\bZ)\right] \weq 0.
\end{equation}

The following theorem considers an oracle setting where the two latent lines are known.
%where the equations for the two line segments are known.
Community recovery then reduces to a sequence of Gaussian hypothesis tests, and we can compute a signal-to-noise ratio at which detection becomes impossible. %Theorem \ref{thm:almost_main} in Section \ref{sec:algorithms} implies that this bound is critical. %Since this same ratio forms a lower bound for the unknown-line problem, any algorithm that matches it in the oracle case also achieves the information-theoretic threshold when the lines must be estimated.

\begin{theorem}
\label{thm:information_theory}
Let \((\bZ,\bX)\) be sampled from \(\GLMM_N(\alpha,\nu,\sigma_N)\), and assume that
\(\nu\) satisfies Assumption~\ref{ass:nu_lower_smallball} with exponent
\(\rho\in(0,1]\). Then there exists a constant \(c>0\), depending only on
\(\alpha\) and \(\nu\), such that
\[
 \min_{\bt \in \cT} \E \left[ \Ham^*(\bt(\bX),\bZ) \right]
 \wge c N \left( \sigma_N\wedge 1 \right)^\rho
\]
for all sufficiently large \(N\). Consequently, exact recovery is impossible if \(\sigma_N \not\ll N^{-1/\rho}\), and almost exact
recovery is impossible if \(\sigma_N \not\ll 1\).
\end{theorem}

\section{Algorithms}
\label{sec:algorithms}

In this section, we introduce polynomial-time algorithms achieving exact recovery up to a polylogarithmic factor of the information-theoretic threshold and almost exact recovery up to the threshold.

\subsection{Spectral method}

Given three points $\bX_1, \bX_2, \bX_3 \in \R^2$ we seek to fit a line $\cL \subset \R^2$
by minimizing the sum of squared perpendicular distances $d(\bX_i, \cL)^2$ between $ \bX_i $ and $ \cL $.
This is referred to as the total least squares (TLS) problem \cite{van2013total}.
Define the TLS residual as
\begin{equation}
 \sigma^2_\tls(\bX_1, \bX_2, \bX_3)
 \weq \min_{\cL \in \text{lines}(\R^2)} \sum_{i=1}^3 d(\bX_i, \cL)^2.
\label{eq:TLS_error}
\end{equation}

We represent the data as a geometric hypergraph by treating sets of three points as hyperedges whenever they do not deviate much from a best-fit line (Figure \ref{fig:TLS-hypergrap}). This is the geometric preprocessing step that converts the intersecting-subspace problem into a hypergraph community recovery problem. 
The TLS hypergraph of $\bX = (\bX_1, \dots, \bX_N)$ with threshold parameter $t > 0$
is defined as $H = (V(H),E(H))$ where $V(H) = [N]$ and
\begin{equation}
 E(H)
 \weq \left\{ \{i,j,k\} \in \binom{[N]}{3} \colon \sigma_\tls(\bX_i,\bX_j,\bX_k) \leq t \right\}.
\label{eq:hyperedge_construction}
\end{equation}

\begin{figure}[ht]
\centering
\includegraphics[width=0.3\textwidth, trim={3cm 5cm 3cm 4.5cm},clip]{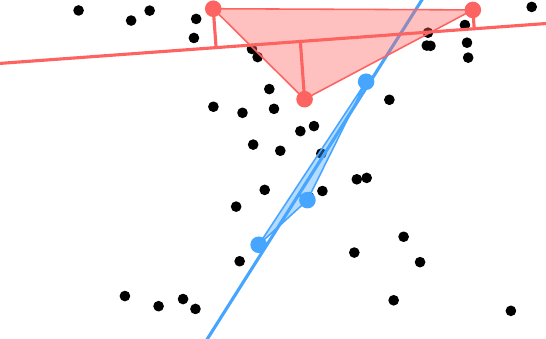}
\caption{Candidate hyperedges in the TLS hypergraph: the blue triple has a small TLS residual and becomes a hyperedge, while the red triple does not.}
\label{fig:TLS-hypergrap}
\end{figure}%

For a given threshold $t$,
Algorithm~\ref{alg:line_clustering_spectral} constructs a similarity matrix that,
for each node pair, counts the number of common hyperedges in $H$.
The community memberships are estimated using $K$-means on the leading eigenvectors of the similarity matrix. 
Theorem~\ref{thm:almost_main} shows that, with a suitable choice of threshold parameter,
Algorithm~\ref{alg:line_clustering_spectral} is information-theoretically optimal up to a polylogarithmic factor.

\begin{algorithm}[H]
\small % Set the font size to footnotesize

\algrenewcommand\algorithmicrequire{\textbf{Input:}}
\algrenewcommand\algorithmicensure{\textbf{Output:}}
\caption{Spectral method}
\begin{algorithmic}[1]
\Require Points $\bX_1,\bX_2,\cdots, \bX_N \in \R^2$, threshold $t > 0$, parameter $\varepsilon > 0$
\Ensure Membership vector $\hat\bZ =(\hat Z_1,\dots, \hat Z_N) \in \{1,2\}^N$\\
Construct the TLS hypergraph $H$ of $\bX$ with threshold parameter $t > 0$.\\
Construct a similarity matrix $\bW$ where
$
W_{ij}
\weq \sum_{e: e \ni i,j} \1\{e \in E(H)\}
$
counts the number of hyperedges incident to $i$ and $j$.\\

Let $ \bU \in \mathbb{R}^{N \times 2} $ be the matrix with columns $ \bv_1, \bv_2 $, eigenvectors of $\bW$ corresponding to eigenvalues $\lambda_1$ and $\lambda_2$ with $\lambda_1 \geq \dots \geq \lambda_N$
\\
Assign community labels $\hat Z_i \in \{1,2\}$ based on a $(1+\varepsilon)$-approximate $K$-means clustering with $K = 2$ on the rows of $\bU$.
\end{algorithmic}
\label{alg:line_clustering_spectral}
\end{algorithm}	
In the last step of Algorithm~\ref{alg:line_clustering_spectral}, a
$(1+\varepsilon)$-approximate $K$-means clustering with $K=2$ refers to any
$K$-means procedure whose objective value
is at most $(1+\varepsilon)$ times the global minimum
(see \eqref{eq:k-means-objective}).

\begin{theorem}
\label{thm:almost_main}
Let \((\bZ,\bX)\) be sampled from \(\GLMM_N(\alpha,\nu,\sigma_N)\), and assume that \(\nu\)
satisfies Assumption~\ref{ass:nu_upper_smallball} with exponent \(\rho\in(0,1]\).
Fix an approximation parameter \(\varepsilon>0\), and let \(t_N\) denote the threshold parameter used in
Algorithm~\ref{alg:line_clustering_spectral}. Then the following hold.
\begin{enumerate}[label=(\roman*)]
\item If
\begin{equation}
 \sigma_N\sqrt{\log N}\ll t_N\ll \frac{1}{(N\log N)^{1/\rho}},
 \label{eq:exac}
\end{equation}
then Algorithm~\ref{alg:line_clustering_spectral}, with input parameters \(t_N\) and \(\varepsilon\), achieves exact recovery.
\item If
\begin{equation}
 \sigma_N\ll t_N\ll 1,
 \label{eq:almost_exact}
\end{equation}
then Algorithm~\ref{alg:line_clustering_spectral}, with input parameters \(t_N\) and \(\varepsilon\), achieves almost exact recovery.
\end{enumerate}
\end{theorem}

When the noise level \(\sigma_N\) is known, Theorem~\ref{thm:almost_main} gives an admissible window for the threshold \(t_N\).
In the exact recovery regime
\[
 \sigma_N \ll \frac{1}{N^{1/\rho}(\log N)^{1/\rho+1/2}},
\]
one convenient choice is the geometric mean of the lower and upper bounds in \eqref{eq:exac},
\[
 t_N
 \weq
 \Bigl(\sigma_N \sqrt{\log N}\cdot (N\log N)^{-1/\rho}\Bigr)^{1/2}.
\]
This lies in the admissible range and therefore yields exact recovery.

In the almost-exact recovery regime \(\sigma_N \ll 1\), from \eqref{eq:almost_exact}, one may similarly take
\(
 t_N = \sqrt{\sigma_N},
\)
which satisfies
\(
 \sigma_N \ll t_N \ll 1
\)
and yields almost exact recovery.

\begin{remark}[A simple thresholding method]
\label{re:thresholding_alternative}
Besides the spectral method of Algorithm~\ref{alg:line_clustering_spectral}, one may also consider a simpler \emph{local thresholding} rule based on the similarity matrix \(\bW\). Informally, one fixes an anchor vertex \(i_0\), uses the entry \(W_{i_0j}\) as a score for whether vertex \(j\) belongs to the same community as \(i_0\), and assigns \(j\) according to whether this score exceeds a suitable threshold.

Under the conditions of Theorem~\ref{thm:almost_main}, the entries of a fixed row of \(\bW\) exhibit a clear two-group structure: intra- and inter-community pairs have conditional means separated by a gap of order \(N(p-q)\), where \(p\) and \(q\) denote the intra- and inter-community hyperedge probabilities of the TLS hypergraph, while the corresponding variances are of order \(N^2((1-p)+q)\). A direct Chebyshev argument therefore shows that thresholding a single row at any level between these two mean values yields the same guarantees as in Theorem~\ref{thm:almost_main}: exact recovery in the very-low-noise regime and almost exact recovery in the low-noise regime.

We do not present this rule as the main algorithm because it depends on a threshold defined through unknown mean similarity levels, and its output depends on the choice of anchor vertex. By contrast, the spectral method uses the full similarity matrix and requires no such threshold. For this reason, we present and analyze the spectral method in detail, and use the thresholding rule only as a conceptual benchmark.
\end{remark}

\subsection{Spectral method with threshold selection}
When the noise level is unknown, we next describe a data-driven method for selecting $t_N$. In Algorithm \ref{alg:line_clustering_spectral_nonparameteric}, the threshold selection is based on counting observed triples in the TLS hypergraph. Theorem \ref{thm:main_nonparametric} shows that the algorithm achieves the information-theoretic threshold for almost exact recovery when we exploit the gap between the intra- and inter-community hyperedge probabilities. In particular, the threshold should be chosen such that the observed fraction of all triples aligns well with the expected fraction of intra-community triples
$
\P(Z_1=Z_2=Z_3)
= 2 \cdot (\frac12)^3
= \frac14.
$

\begin{algorithm}[H]
\small

\algrenewcommand\algorithmicrequire{\textbf{Input:}}
\algrenewcommand\algorithmicensure{\textbf{Output:}}
\caption{Spectral method with threshold selection}
\begin{algorithmic}[1]
\Require Points $\bX_1,\dots,\bX_N \in \mathbb{R}^2$, integer $M \ge 1$, target fraction $\theta \in [0,1]$, parameter $\varepsilon > 0$
\Ensure Membership vector $\hat\bZ =(\hat Z_1, \dots, \hat Z_N) \in \{1,2\}^N$

\State Sample unordered triples $\tau_1,\dots,\tau_M \subset [N]$ independently and uniformly at random.

\State Let
$\cM = \cup_{m=1}^M \tau_m$
%\cup_{i=1}^M \tau_i
and $\cM^c = [N] \setminus \cM$.

\State Compute the values $s_i=\sigma_{\rm TLS}(\tau_i)$ and sort the list $(s_i)_{i=1}^M$ to obtain
$s_{(1)}\le s_{(2)}\le\cdots\le s_{(M)}$.

\State 
 Let $k=\newround{\theta M}$ where ''$\newround{\cdot}$'' rounds to the nearest integer and set 
 $$
  t^{*}= 
  \begin{cases}
 s_{(k)},& \text{ if } 1\le k\le M,\\
 s_{(1)},& \text{ if } k = 0.  \end{cases}
 $$

\State Obtain a community label $\hat Z_i$ for each node $i \in \cM^c$ from Algorithm \ref{alg:line_clustering_spectral} on $\cM^c$ with $t^*$ and $\varepsilon$ as inputs.

\State Assign labels on \(\cM\) arbitrarily, for example uniformly at random.
\end{algorithmic}
\label{alg:line_clustering_spectral_nonparameteric}
\end{algorithm}
Since each sampled triple contributes at most three vertices, \(|\cM| \le 3M\). Under the regime of Theorem~\ref{thm:main_nonparametric}, we have \(M_N=o(N)\), so arbitrary labels on \(\cM\) do not affect almost exact recovery.

\begin{theorem}
\label{thm:main_nonparametric}
Let \((\bZ,\bX)\) be sampled from \(\GLMM_N(\alpha,\nu,\sigma_N)\), and assume that \(\nu\)
satisfies Assumptions~\ref{ass:nu_lower_smallball} and \ref{ass:nu_upper_smallball} with exponent
\(\rho\in(0,1]\). Fix an approximation parameter \(\varepsilon>0\), and let \(M_N\) denote the
integer input parameter used in Algorithm~\ref{alg:line_clustering_spectral_nonparameteric}. If
\[
 \sigma_N\ll 1
 \qquad\text{and}\qquad
 1\ll M_N\ll N^{1/2},
\]
then Algorithm~\ref{alg:line_clustering_spectral_nonparameteric}, with input parameters
\(M_N\), \(\theta=1/4\), and \(\varepsilon\), achieves almost exact recovery.
\end{theorem}

\begin{remark}[Computational complexity]
Algorithm \ref{alg:line_clustering_spectral_nonparameteric} has two main stages: threshold selection and clustering. First, selecting a threshold $t^*$ involves computing TLS residuals of $M_N$ sampled triples, each computation requiring constant time. Sorting these residuals takes $O(M_N \log M_N)$ operations, so the overall cost of threshold selection is $O(M_N \log M_N)$. Given $M_N \ll N^{1/2}$, this cost is negligible compared to the second stage. 

The clustering step involves computing TLS residuals of $\binom{N}{3} = O(N^3)$ triples and incrementing the corresponding entries of the similarity matrix, followed by computing its two leading eigenvectors in at most $O(N^3)$ time. The final two-means clustering of the $N$ rows is carried out by a $(1+\varepsilon)$-approximate $K$-means algorithm (see \eqref{eq:k-means-objective}); for example, the $(1+\varepsilon)$-approximation algorithm of \cite{matouvsek2000approximate} runs in time linear in $N$ up to a polylogarithmic factor. Hence, the total time complexity of the algorithm is $O(N^3)$.
\end{remark}

\begin{remark}[Recovering the line directions]
Once the community memberships have been estimated, one may estimate the two latent line directions by applying principal component analysis separately within each recovered cluster. Concretely, for \(k=1,2\), let
\[
\widehat C_k\coloneqq \{i:\hat Z_i=k\},
\qquad
n_k\coloneqq |\widehat C_k|,
\]
and define the empirical mean and covariance matrix
\[
 \widebar \bX_k=\frac1{n_k}\sum_{i\in\widehat C_k}\bX_i,
 \qquad
 \widehat \bS_k
 =\frac1{n_k}\sum_{i\in\widehat C_k}
 \left( \bX_i-\widebar \bX_k \right)\left( \bX_i-\widebar \bX_k \right)^{\top}.
\]
Any leading eigenvector \(\hat\bv_k\) of \(\widehat \bS_k\) may then be used as an estimator of the direction of the \(k\)th line. This kind of post-processing is standard in subspace clustering; see, for example, \cite[Section~2.2]{soltanolkotabi2014robust}. Since the main inferential target of the present paper is pointwise labeling rather than subspace estimation, we do not analyze this step further.
\end{remark}

\begin{remark}[Refinement by reassignment]
As a heuristic improvement, one may further refine the initial community assignments after estimating the two line directions. A simple option is to reassign each point to the cluster whose estimated line is closest in perpendicular distance. We do not establish theoretical guarantees for this refinement, but it is natural to expect that it helps resolve the instability of assignments near the intersection, where the hypergraph construction is most sensitive to the threshold. We conjecture that, in the exact recovery regime \(\sigma_N\ll N^{-1/\rho}\), combining Algorithm~\ref{alg:line_clustering_spectral_nonparameteric} with this refinement yields exact recovery.
\end{remark}

\section{Simulation study}
\label{sec:simulation}

We report two synthetic experiments designed to clarify the algorithmic picture rather than explore finite-sample phase transitions under the generative model. In both experiments, the latent geometry consists of two intersecting lines, but the sampling scheme is partly chosen to match earlier empirical work and therefore is not identical to the Gaussian line mixture model \(\GLMM_N(\alpha,\nu,\sigma_N)\). The first experiment isolates the effect of the hyperedge construction, while the second isolates the effect of thresholding versus using all triple-wise TLS residuals. 

To evaluate clustering accuracy given the latent communities, we use the Hubert--Arabie adjusted Rand index (ARI) \cite{Gosgens_Tikhonov_Prokhorenkova_2021,hubert_arabie_1985}, which measures similarity between two community assignments. The index equals one for identical assignments and has expected value zero under independent random assignments.

\subsection{Fixed thresholds for Gaussian and TLS hypergraphs}

We compare Algorithm~\ref{alg:line_clustering_spectral} with the \emph{HyperSBM} method of
\cite{brusa2024model} in the line clustering setting of \cite[Section~4.2]{brusa2024model}.
Let \(L_1,L_2 \subset \R^2\) be two line segments
given by
\[
\{(x,y)\in \R^2 : x \in [-0.5,0.5],\, y = m_k x\}
\] 
with slopes \(m_1 = -1\) and \(m_2 = 0.8\).
We take \(N = 60\) points, \(30\) in each community.
For each index \(i\) with label \(Z_i = k\) we draw
\(x_i \eqd \Unif[-0.5,0.5]\) and set
\[
  \bX_i = (x_i,y_i),
  \qquad
  y_i = m_k x_i + \varepsilon_i,
  \qquad
  \varepsilon_i \eqd \cN(0,\sigma_N^2),
\]
with \(\sigma_N = 0.01\), so that the noise perturbs only the \(y\)-coordinate. In contrast to \cite{brusa2024model}, we do not add outliers, so every node belongs to one of the two communities. This sampling scheme differs from the line mixture model \(\GLMM_N(\alpha,\nu,\sigma_N)\) used in our theory, where Gaussian noise is added in both coordinates; here we restrict the perturbation to the vertical direction to match the setup of \cite{brusa2024model}.

We construct two 3-uniform hypergraphs on this point set.
For the \emph{Gaussian hypergraph}, we follow the public replication code provided by
the authors of \cite{brusa2024model}.
For each triple \(\{i,j,k\}\) we fit a least-squares line to
\(\bX_i,\bX_j,\bX_k\), compute the sum of squared vertical residuals
\(d(\bX_i,\bX_j,\bX_k)^2\), and define the Gaussian similarity
\[
  s(\bX_i,\bX_j,\bX_k)
  = \exp\left(-d(\bX_i,\bX_j,\bX_k)^2 / \tau^2\right),
\]
with \(\tau^2 = 4\times 10^{-4}\) and threshold \(\varepsilon = 0.999\).
A 3-uniform hyperedge is added whenever \(s(\bX_i,\bX_j,\bX_k) > \varepsilon\).
This choice matches the replication code, although the article text reports \(\tau^2 = 0.04\).

For the \emph{TLS hypergraph}, we use the same latent data but replace the least-squares vertical residual by the TLS residual \(\sigma^2_\tls(\bX_i,\bX_j,\bX_k)\) from~\eqref{eq:TLS_error}, and define
hyperedges as in~\eqref{eq:hyperedge_construction}. Guided by Theorem~\ref{thm:almost_main}, we fix, for \(\sigma_N = 0.01\),
\[
  t_N \;=\; \sqrt{\sigma_N}.
\]

Following \cite{brusa2024model}, we thin the hypergraph by independently subsampling intra- and inter-community triples so that, in expectation, there are about twice as many intra- as inter-community retained hyperedges. Concretely, if $T_{\mathrm{in}}$ and $T_{\mathrm{out}}$ denote the numbers of intra- and inter-community hyperedges before thinning, we retain each intra-community triple with probability $p_{\mathrm{in}} = 0.1$ and each inter-community triple with probability 
\[ 
 p_{\mathrm{out}}
 = \min \left\{1,\, \frac{\mu p_{\mathrm{in}} T_{\mathrm{in}}}
 {(1-\mu) T_{\mathrm{out}}} \right\},
 \qquad
 \mu = \tfrac{1}{3}, 
\] 
so that the expected fraction of inter-community triples among the retained hyperedges is \(\mu\), and the expected ratio of intra- to inter-community hyperedges is \((1-\mu):\mu = 2:1\). For the TLS hypergraph, this thinning step is used only to make the retained edge densities comparable to those of the Gaussian construction.

Table~\ref{tab:edges_two_lines} summarizes,
for each construction, the size and composition of the candidate hyperedge set (before thinning) and the average number of hyperedges retained after thinning. For the Gaussian hypergraph, the Gaussian kernel with threshold \(\varepsilon\)
produces roughly \(1.1\times 10^4\) candidate hyperedges, of which about three quarters are
intra-community.
The TLS threshold \(t_N\) yields a denser candidate set, but after thinning, both constructions
produce hypergraphs with about \(1.2\times 10^3\) edges and a similar fraction of
inter-community edges.

\begin{table}[t]
\centering

\caption{
Edge statistics for the two-line clustering experiment with $N=60$.
''Candidate'' denotes the number of candidate hyperedges before thinning;
''Intra'' and ''Inter'' are within- and between-community edges;
''Sampled'' is the mean number of hyperedges retained after thinning over 100 datasets.}
%\caption{Edge statistics for the two-line clustering experiment with $N=60$. ''Candidate'' is the total number of pre-thinning hyperedges; ''Intra'' and ''Inter'' split this total into within- and between-community edges, and ''Mean sampled edges'' is the average number of hyperedges retained after thinning over 100 datasets.}

\label{tab:edges_two_lines}
\begin{tabular}{lrrrr}
\toprule
Hypergraph & Candidate & Intra & Inter & Sampled \\
\midrule
Gauss  & 11\,375 & 8\,098 & 3\,277  & 1\,208 \\
TLS    & 18\,403 & 8\,120 & 10\,283 & 1\,214 \\
\bottomrule
\end{tabular}
\end{table}

On this common data-generation scheme, we compare three procedures, all with the number
of communities fixed at two:
\begin{itemize}
  \item \emph{HSBM--Gauss}: HyperSBM applied to the Gaussian hypergraph;
  \item \emph{Spectral--Gauss}: Algorithm~\ref{alg:line_clustering_spectral} applied to the
        Gaussian hypergraph;
  \item \emph{Spectral--TLS}: Algorithm~\ref{alg:line_clustering_spectral} applied to the TLS
        hypergraph.
\end{itemize}
We perform \(100\) independent repetitions of the data generation, hypergraph construction, 
and clustering steps, and compute the ARI between the estimated labels and the latent partition.
Figure~\ref{fig:ARI} presents boxplots which display both the typical accuracy (medians) and the variability across the 100 replications, providing a direct comparison of clustering accuracy under the two hypergraph constructions and the two clustering approaches.
\begin{figure}[ht]
\centering
\includegraphics[width=.6\textwidth,clip]{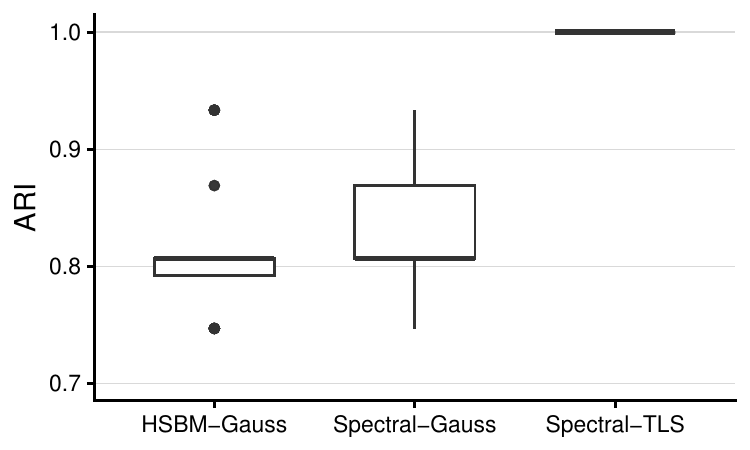}
\caption{Boxplots of the ARI values obtained by the three clustering methods on the two-line clustering problem.}
\label{fig:ARI}
\end{figure}%

The results in Figure~\ref{fig:ARI} suggest two main conclusions. First, on the Gaussian hypergraph, \emph{HSBM--Gauss} and \emph{Spectral--Gauss} exhibit comparable performance, with mean ARIs 0.818 and 0.820, respectively, and relatively tight dispersion over the 100 repetitions. In this experiment, once the hypergraph is fixed, the choice between HyperSBM and similarity-matrix spectral clustering has only a limited effect. Second, replacing the Gaussian hyperedge construction by the TLS construction markedly improves accuracy: \emph{Spectral--TLS} achieves exact recovery in every run (ARI \(=1\)). Thus, in this fixed-\(K\), matched-density comparison, the main empirical gain comes from the geometric hyperedge construction rather than from fitting a richer probabilistic model on a given hypergraph.

This complements the comparison in \cite{brusa2024model}, where HyperSBM clearly outperformed modularity-based methods. We do not claim that the TLS-based spectral method also outperforms those approaches, as we fix the number of communities in advance and do not repeat the model-selection procedure used in their study.

\subsection{Data-driven threshold selection against clique averaging}

This experiment isolates the effect of thresholding versus using the full weighted collection of TLS residuals. We keep the two-line model of the previous subsection, but increase the noise level to \(\sigma_N = 0.02\); all other aspects of the data-generating mechanism remain unchanged.
For each triple \(\{i,j,k\}\) we first compute the TLS residual
\[
  d_{ijk}
  \;\coloneqq\;
  \sigma_{\mathrm{TLS}}^2(\bX_i,\bX_j,\bX_k),
\]
and collect all such residuals in a vector
\(\bd \in [0,\infty)^{\binom{N}{3}}\), indexed by unordered triples \(\{i,j,k\}\).

The first method follows the \emph{Clique averaging} scheme of \cite{agarwal2005beyond}.
We treat \(d_{ijk}\) as a dissimilarity and convert it to an affinity
\[
  h_{ijk}
  \;=\;
  \exp\left(-d_{ijk} / \kappa\right),
\]
where \(\kappa > 0\) is a single global scale parameter, chosen as the median of
\(\{d_{ijk}\}_{1 \le i<j<k \le N}\).
Let \(\bh \in [0,1]^{\binom{N}{3}}\) collect all \(h_{ijk}\) in the same ordering
of triples as \(\bd\), and let \(\bDelta\) denote the incidence matrix of the complete
3-uniform hypergraph on \([N]\), with one row per triple and one column per unordered
node pair \(\{u,v\}\):
\[
  \bDelta \in \{0,1\}^{\binom{N}{3} \times \binom{N}{2}},
  \qquad
  \bDelta_{\{i,j,k\},\{u,v\}} = \1\left(\{u,v\} \subset \{i,j,k\}\right).
\]
We then solve the bounded least-squares problem
\[
  \min_{\bd^{(2)} \in [0,1]^{\binom{N}{2}}}
  \left\| \lambda_3 \bDelta \bd^{(2)} - \bh \right\|_2^2,
  \qquad
  \lambda_3 = \binom{3}{2}^{-1},
\]
where \(\bd^{(2)}\) is a vector of edge weights indexed by unordered node pairs \(\{u,v\}\).
These weights define a symmetric similarity matrix
\(\bW^{\mathrm{A}} \in \R^{N \times N}\) by
\[
  \bW^{\mathrm{A}}_{uv} = \bW^{\mathrm{A}}_{vu}
  = d^{(2)}_{\{u,v\}},
  \qquad 1 \le u < v \le N.
\]
We then apply to \(\bW^{\mathrm{A}}\) the same spectral clustering step as in
Algorithm~\ref{alg:line_clustering_spectral}.

The second method uses our TLS-based spectral procedure with data-driven threshold
selection, as in Algorithm~\ref{alg:line_clustering_spectral_nonparameteric}, but
with two simplifications tailored to this small-\(N\) setting.
First, instead of sampling \(M_N\) triples, we set
\(M_N = \binom{N}{3}\) and use \emph{all} triples when forming the empirical distribution
of \(\{d_{ijk}\}_{1 \le i<j<k \le N}\), so that the threshold \(t^*\) is the
empirical \(\theta\)-quantile of the full collection of TLS residuals, with \(\theta = 1/4\) as in Theorem~\ref{thm:main_nonparametric}.
Second, after computing \(t^*\) we do not split the vertex set into \(\cM\) and
\(\cM^c\), but instead run Algorithm~\ref{alg:line_clustering_spectral} on the
full vertex set \([N]\) with threshold \(t_N\) replaced by \(t^*\):
we form the TLS hypergraph by including all triples satisfying
\(\sigma_{\mathrm{TLS}}(\bX_i,\bX_j,\bX_k) \le t^*\), construct the similarity matrix
\(\bW^{\mathrm{TLS}} \in \R^{N \times N}\) by clique expansion as in the previous
subsection, and apply the same spectral \(k\)-means step with two clusters.

Thus, the downstream clustering step is identical; only the mapping from the weighted hypergraph to a similarity matrix differs. Clique averaging introduces a fitted scale parameter \(\kappa\), whereas TLS thresholding uses only the fixed quantile level \(\theta = 1/4\) and has no other tuning parameters. We repeat the full experiment (data generation, similarity construction, and
spectral clustering) for 100 independent replications and measure performance with
ARI values between the estimated labels and the latent partition. The resulting ARI values are summarized in Figure \ref{fig:ARI_2}.

\begin{figure}[htbp!]
\centering
\includegraphics[width=.6\textwidth,clip]{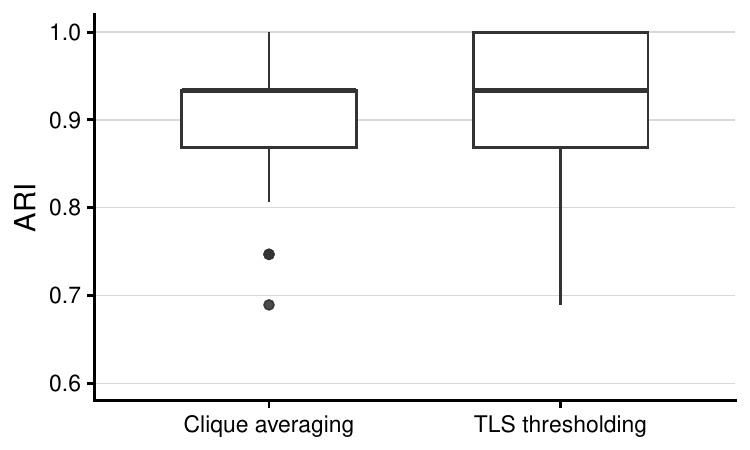}
\caption{Boxplots of ARI values for \emph{Clique averaging} and \emph{TLS thresholding} on the two-line clustering problem.}
\label{fig:ARI_2}
\end{figure}%

The mean ARI was approximately 0.90 for \emph{Clique averaging} and 0.92 for \emph{TLS thresholding}. The boxplots in Figure~\ref{fig:ARI_2} indicate that both methods typically achieve high clustering accuracy, with most ARI values above 0.8. The thresholding method shows a slightly higher central tendency and a broadly comparable spread to clique averaging, indicating a small but consistent performance gain in this two-line setting. At the same time, the overall performance of the two methods remains similar.

Across both experiments, the empirical picture is consistent: in this simple two-line setting, the geometric hyperedge construction appears to matter at least as much as the downstream clustering routine, and a thresholded TLS hypergraph performs comparably to approaches using the full collection of TLS residuals.

\section{Proofs}
\label{sec:proofs}

Throughout the proofs, we fix an auxiliary parameter \(\ell>0\) such that
\[
\supp(\nu)\subset[-\ell/2,\ell/2].
\]
With this choice, all latent locations lie on the truncated portions
\[
\widetilde L_k \coloneqq \{u\bv_k:\ |u|\le \ell/2\},
\qquad k=1,2,
\]
of the two latent lines. Thus, whenever a proof refers to line segments of length \(\ell\), it is only using these auxiliary truncations; the model itself remains defined in terms of the full lines.

For convenience, we also record the concrete forms of Assumptions~\ref{ass:nu_lower_smallball}
and \ref{ass:nu_upper_smallball} that will be used below. With the above choice of \(\ell\),
Assumption~\ref{ass:nu_lower_smallball} is equivalent to the existence of constants \(B>0\) and
\(u_0>0\) such that
\[
\nu([-u,u])\ge B\left(\frac{u}{\ell}\right)^\rho
\qquad\text{for all }0\le u\le u_0,
\]
and Assumption~\ref{ass:nu_upper_smallball} is equivalent to the existence of a finite constant
\(A_0>0\) such that
\[
\nu(I)\le A\left(\frac{|I|}{\ell}\right)^\rho
\]
for every interval \(I\subset[-\ell/2,\ell/2]\). Since the same bound then also holds with any larger
constant, we may replace \(A_0\) by \(A\coloneqq \max\{A_0,1\}\) and therefore assume \(A\ge 1\)
without loss of generality. In the proofs below, we invoke Assumptions~\ref{ass:nu_lower_smallball}
and \ref{ass:nu_upper_smallball} through these equivalent scale-normalized formulations without further comment.

\subsection{Information-theoretic analysis}
\label{sec:information-theoretic_analysis}

In this section, we study the optimal performance of any estimator under permutation-invariant Hamming loss. The main theoretical contribution of the section is a general decision-theoretic result: the coordinatewise MAP rule is Bayes optimal not only for standard Hamming loss, but also for the permutation-invariant Hamming loss. To prove this, we first establish a coordinatewise improvement lemma, which shows that, outside a null set, one can always reduce the risk of a suboptimal rule by flipping a single coordinate. We then specialize this result to our line mixture model and derive the information-theoretic recovery thresholds.

\subsubsection{Lower bound for Bayes risk}
\label{sec:lower_bound_for_bayes_risk}
In this subsection, we work in a general binary Bayesian setting,
with observations taking values in a general measurable space $\cX$.
Let $(Z,X)$ take values in $\{1,2\}\times\cX$ with prior
$Z\eqd\Unif\{1,2\}$ and class-conditional laws
$X\given\{Z=k\}\eqd F_k$, $k=1,2$.
Assume that each $F_k$ admits a density $f_k$ with respect to a
$\sigma$-finite measure $\mu$ on $\cX$, and that the marginal law of $X$
has density $(f_1+f_2)/2$ with respect to $\mu$.
By \cite[Theorem~1.31]{Schervish_1995}, there exists a regular conditional
distribution $\P(Z\in\cdot\given X=x)$, and on the set
$\{x\in\cX : f_1(x)+f_2(x)>0\}$ it satisfies
\begin{equation}
  \P(Z=k\given X=x)
  = \frac{f_k(x)}{f_1(x)+f_2(x)},\qquad k=1,2.
\label{eq:posterior_def_generic}
\end{equation}

The single-coordinate maximum a posteriori (MAP) estimator $\tilde t:\cX\to\{1,2\}$ is then
\begin{equation}
\label{eq:bayes_rule_generic}
  \tilde t(x)\in\argmax_{k\in\{1,2\}} f_k(x),
\end{equation}
with an arbitrary measurable tie-breaking.
For $N$ independent copies $(Z_i,X_i)$ of $(Z,X)$, write
\[
  \cZ=\{1,2\},\qquad
  \bX=(X_1,\dots,X_N)\in\cX^N,\qquad
  \bZ=(Z_1,\dots,Z_N)\in\cZ^N,
\]
let $\P$ be the joint law of $(\bZ,\bX)$ and $\P_2$ the marginal law of $\bX$,
and define the coordinatewise MAP estimator
\begin{equation}
\label{def:coordinatewise_map}  
  \tbt(\bx)
  \coloneqq
  \left(\tilde t(x_1),\dots,\tilde t(x_N)\right),\qquad \bx\in\cX^N.
\end{equation}

The following proposition shows that the coordinatewise MAP estimator $\tilde{\bt}$ from
\eqref{def:coordinatewise_map}, which is Bayes optimal for the usual Hamming loss, remains Bayes optimal for the permutation-invariant Hamming loss. We first define a regularity set relevant to the lemma.
\begin{definition}
\label{def:good_set_bayes}
For $\bx=(x_1,\dots,x_N)\in\cX^N$ define
\[
 \cR
 \;\coloneqq\;
 \left\{
 \bx\in\cX^N:
 \min\{f_1(x_i),f_2(x_i)\}>0
 \ \text{and}\ 
 f_1(x_i)\neq f_2(x_i)
 \ \text{for all } i \in [N]
 \right\}.
\]
\end{definition}

\begin{proposition}\label{prop:pi-hamming-cwmap}
Let $(\bZ,\bX)$ be distributed according to $\P$ as above, and assume that
$\P_2(\cR)=1$.
Let $\cT$ be the class of measurable decision rules
$\bt:\cX^N\to\cZ^N$, and let $\tbt\in\cT$ denote the coordinatewise MAP
estimator defined in \eqref{def:coordinatewise_map}.
Then
\[
 \E \left[ \Ham^*(\tbt(\bX),\bZ) \right]
 \weq \min_{\bt \in \cT} \E \left[ \Ham^*(\bt(\bX),\bZ) \right].
\]
\end{proposition}
The proof of the above proposition is deferred to Appendix \ref{ap:proof_of_prop:pi-hamming-cwmap} since it relies on auxiliary decision-theoretic results collected in Appendix~\ref{ap:decision_theory}.

\subsubsection{Application to the Gaussian line mixture model}

We now specialize the general Bayes optimality results to the clustering problem at hand. Let
\((\bZ,\bX)\eqd \GLMM_N(\alpha,\nu,\sigma)\), and fix \(\ell>0\) such that
\(\supp(\nu)\subset[-\ell/2,\ell/2]\). Recalling the model definition \eqref{eq:model}, Tonelli's theorem implies that the conditional law of \(\bX_i\) given \(Z_i=k\) admits a Lebesgue density
\(f_k\) on \(\R^2\),
\begin{align*}
 f_k(\bx)
 &= \frac{1}{\sigma^2}
 \int_{\R^2}
 \phi\left(\frac{x_1-u_1}{\sigma}\right)
 \phi\left(\frac{x_2-u_2}{\sigma}\right)
 \,G_k(du_1,du_2),
\end{align*}
where \(\phi(t)=(2\pi)^{-1/2}e^{-t^2/2}\) is the standard normal density and
\(G_k\) is the image of \(\nu\) under the map \(u\mapsto u\bv_k\). Equivalently,
\[
 f_k(x_1,x_2)
 =
 \frac{1}{\sigma^2}
 \int_{\R}
 \phi\left(\frac{x_1-u\,v_{k1}}{\sigma}\right)
 \phi\left(\frac{x_2-u\,v_{k2}}{\sigma}\right)
 \,\nu(du),
\]
where \(\bv_k=(v_{k1},v_{k2})\).
The unconditional law of \(\bX_i\) is then absolutely continuous with respect to Lebesgue measure \(\lambda\) on \(\R^2\) with density \((f_1+f_2)/2\), so the model assumptions of Section~\ref{sec:lower_bound_for_bayes_risk} hold with \(\mu=\lambda\) and \(f_1,f_2\) as above. Moreover, \(\min\{f_1(\bx),f_2(\bx)\}\) is strictly positive for every \(\bx\in\R^2\), and the one-observation tie set
\[
 T \coloneqq \{\bx\in\R^2:f_1(\bx)=f_2(\bx)\}
\]
is \(\lambda\)-null (see the proof of Lemma~\ref{lem:impossibility}). Since the marginal law of
each \(\bX_i\) has density \((f_1+f_2)/2\) with respect to \(\lambda\), it follows that
\[
 \P_2\bigl(\{\bx\in(\R^2)^N:\bx_i\in T\}\bigr)=0
 \qquad\text{for each } i\in[N].
\]
Since
\[
 \cR
 =
 \left\{
 \bx\in(\R^2)^N: \min\{f_1(\bx_i),f_2(\bx_i)\}>0 \text{ and }
 \bx_i\notin T \text{ for all } i\in[N]
 \right\},
\]
we conclude that \(\P_2(\cR)=1\). Hence Proposition~\ref{prop:pi-hamming-cwmap}
applies to the Gaussian line mixture model.

Let \(\tilde t\) be the single-coordinate MAP estimator from
\eqref{eq:bayes_rule_generic}, and define
\begin{equation} 
 \perr
 \;\coloneqq\;
 \P\left(Z_i\neq\tilde t(\bX_i)\right)
\label{eq:perr_def}
\end{equation}
for the corresponding Bayes error probability. The next lemma gives an
explicit lower bound for \(\perr\).

\begin{lemma}
\label{lem:impossibility}
Let \((\bZ,\bX)\eqd \GLMM_N(\alpha,\nu,\sigma)\), and fix \(\ell>0\) such that
\(\supp(\nu)\subset[-\ell/2,\ell/2]\). Assume that there exist \(\rho\in(0,1]\), \(B>0\), and
\(u_0>0\) such that
\[
 \nu([-u,u])
 \wge
 B\left(\frac{u}{\ell}\right)^\rho,
 \qquad 0\le u\le u_0.
\]
Then there exists a constant \(c>0\), depending only on \(\alpha,\rho,B,u_0/\ell\), such that
\[
 \perr
 \wge
 c\left(\frac{\sigma}{\ell}\wedge 1\right)^\rho
\]
for every \(\sigma>0\).
\end{lemma}
The proof is postponed to Appendix~\ref{ap:gaussian}, since it consists of lengthy but straightforward evaluations of Gaussian integrals.

\subsubsection{Proof of Theorem \ref{thm:information_theory}}

Let us fix a coordinatewise MAP estimator \(\tbt\), let
\((\bZ,\bX)\eqd \GLMM_N(\alpha,\nu,\sigma_N)\), and fix \(\ell>0\) such that
\[
\supp(\nu)\subset[-\ell/2,\ell/2].
\]
Let
\[
 D \weq \Ham(\tbt(\bX),\bZ),
 \qquad
 D^* \weq \Ham^*(\tbt(\bX),\bZ).
\]
By Proposition~\ref{prop:pi-hamming-cwmap}, the minimum Bayes risk under the permutation-adjusted
Hamming loss is achieved by \(\tbt\) and equals \(\E D^*\).

Let
\[
 p_N \weq \P\left(Z_i\neq \tbt_i(\bX_i)\right)
\]
denote the single-coordinate Bayes error. Since \(\tbt\) acts coordinatewise,
\(
 D \eqd \operatorname{Bin}(N,p_N),
\)
and therefore
\(
 \E D = N p_N,
\)
 and
\(
 \Var(D)=N p_N(1-p_N).
\)

We first relate \(\E D^*\) to \(N p_N\). Since
\[
 D^* = D\wedge(N-D) = D-(2D-N)1(D\ge N/2),
\]
we have
\[
 D^* \wge D - N1(D\ge N/2).
\]
Taking expectations gives
\[
 \E D^* \wge \E D - N\P(D\ge N/2).
\]
Since \(p_N\le 1/2\), we have \(\E D\le N/2\), and Chebyshev's inequality yields
\[
 \P(D\ge N/2)
 \weq \P(D-\E D \ge N/2-\E D)
 \wle \frac{\Var(D)}{(N/2-\E D)^2}
 \wle \frac{p_N}{N(1/2-p_N)^2}.
\]
Hence
\begin{equation}
\label{eq:BayesRiskLowerBound1_new}
 \E D^*
 \wge N p_N - \frac{p_N}{(1/2-p_N)^2}.
\end{equation}

We also have
\[
 D^* = \frac{N}{2}-\abs{D-\frac{N}{2}}.
\]
Since \(\E D\le N/2\),
\[
 \abs{D-\frac{N}{2}}
 \wle \abs{D-\E D} + \abs{\E D-\frac{N}{2}},
\]
and therefore
\[
 D^*
 \wge \E D - \abs{D-\E D}.
\]
Taking expectations and applying Cauchy--Schwarz, we obtain
\begin{equation}
\label{eq:BayesRiskLowerBound2_new}
 \E D^*
 \wge \E D - \sqrt{\Var(D)}
 \wge N p_N - \sqrt{N p_N}.
\end{equation}

If \(p_N\le 1/4\), then \eqref{eq:BayesRiskLowerBound1_new} gives
\[
 \E D^*
 \wge N p_N - 16p_N
 \weq (1-16N^{-1})Np_N.
\]
If \(1/4\le p_N\le 1/2\), then \(N p_N\ge N/4\), and \eqref{eq:BayesRiskLowerBound2_new} gives
\[
 \E D^*
 \wge N p_N - \sqrt{N p_N}
 \wge \left(1-(Np_N)^{-1/2}\right)Np_N
 \wge \left(1-2N^{-1/2}\right)Np_N.
\]
Therefore, for all sufficiently large \(N\),
\begin{equation}
\label{eq:BayesRiskLowerBound3_new}
 \E D^* \wge \frac12\,N p_N.
\end{equation}

By Lemma~\ref{lem:impossibility}, there exists a constant \(c_0>0\), depending only on
\(\alpha,\rho,B,u_0/\ell\), such that
\[
 p_N \wge c_0\left(\frac{\sigma_N}{\ell}\wedge 1\right)^\rho.
\]
Since \(\ell\) is fixed,
\[
 \left(\frac{\sigma_N}{\ell}\wedge 1\right)^\rho
 \wge (\ell\lor 1)^{-\rho}(\sigma_N\wedge 1)^\rho.
\]
Hence
\[
 p_N \wge c_1(\sigma_N\wedge 1)^\rho,
\]
where \(c_1>0\) depends only on \(\alpha\) and \(\nu\). Combining this with
\eqref{eq:BayesRiskLowerBound3_new}, we obtain
\[
 \min_{\bt \in \cT} \E \left[ \Ham^*(\bt(\bX),\bZ) \right]
 \weq \E D^*
 \wge \frac12 N p_N
 \wge \frac{c_1}{2}\,N \left(\sigma_N\wedge 1\right)^\rho
\]
for all sufficiently large \(N\). This proves the claimed lower bound.

The recovery consequences are immediate. Exact recovery implies
\[
 \min_{\bt\in\cT}\E\left[\Ham^*(\bt(\bX),\bZ)\right]\to 0,
\]
so the lower bound forces
\(
 N(\sigma_N\wedge 1)^\rho \to 0.
\)
In particular, \(\sigma_N\to 0\), and hence
\(
 N\sigma_N^\rho \to 0,
\)
that is, \(\sigma_N\ll N^{-1/\rho}\).

Likewise, almost exact recovery implies
\[
 \frac{1}{N}\min_{\bt\in\cT}\E\left[\Ham^*(\bt(\bX),\bZ)\right]\to 0,
\]
so
\(
 (\sigma_N\wedge 1)^\rho \to 0.
\)
Hence \(\sigma_N\ll 1\). This concludes the proof.
\qed

\subsection{Bounds on hyperedge probabilities}
\label{sec:HyperedgeProbabilities}

In this section, we analyze the hyperedge probabilities associated with the TLS hypergraph $H$
defined in \eqref{eq:hyperedge_construction}.
For $(\bZ,\bX)$ sampled from $\GLMM_N(\alpha,\nu,\sigma)$,
we denote the intra- and inter-community hyperedge probabilities by
\begin{align*}
 p &\ \coloneqq \ \P\left(\sigma_\tls(\bX_i,\bX_j,\bX_k) \le t \given Z_i=Z_j=Z_k \right), \\
 q &\ \coloneqq \ \P\left(\sigma_\tls(\bX_i,\bX_j,\bX_k) \le t \given Z_i \neq Z_j \text{ or } Z_i \neq Z_k \right).
\end{align*}
We derive lower bounds for $p$ and upper bounds for $q$ that ensure sufficient separation between the clusters in the hypergraph.

\subsubsection{Intra-community hyperedge probability}

We analyze the probability of a triple \(\bX_1,\bX_2,\bX_3\) forming a hyperedge in the TLS hypergraph when all three points are drawn from the same community. By symmetry, it is enough to consider community \(1\).

\begin{lemma}
\label{lem:_within_probability}
Let \((\bZ,\bX)\eqd \GLMM_N(\alpha,\nu,\sigma)\). For \(t>\sqrt{3}\,\sigma\), the intra-community hyperedge probability \(p\) satisfies
\[
 1 - p
 \wle \exp\left( - \frac{3}{2}
 \left( \dfrac{t^2}{3 \sigma^2} - 1 - \log \left( \dfrac{t^2}{3 \sigma^2} \right) \right)
 \right).
\]
\end{lemma}

\begin{proof}
By symmetry, when analyzing the joint law of \((\bX_1,\bX_2,\bX_3)\) given \(Z_1=Z_2=Z_3\), we may assume without loss of generality that \(Z_1=Z_2=Z_3=1\). Denote by \(\cL_1 \subset\R^2\) the line corresponding to community \(1\). Then the TLS residual defined by \eqref{eq:TLS_error} is bounded by
\[
 \sigma^2_\tls(\bX_1,\bX_2,\bX_3)
 \wle \sum_{i=1}^3 d(\bX_i,\cL_1)^2.
\]

Now recall that \(\bX_i = V_i \bv_1 + \sigma \bY_i\), where \(V_i\eqd \nu\) and \(\bY_i\eqd N(0,I_2)\). Since \(V_i \bv_1\in \cL_1\), we have
\[
 d(\bX_i,\cL_1)=\sigma \|\bY_i^\perp\|_2,
\]
where \(\bY_i^\perp\) denotes the component of \(\bY_i\) perpendicular to \(\cL_1\). By rotation invariance of the standard Gaussian distribution on \(\R^2\), we may assume without loss of generality that \(\cL_1\) is the \(x\)-axis, so that \(\|\bY_i^\perp\|_2=|Y_{i2}|\), where \(\bY_i=(Y_{i1},Y_{i2})\). Hence
\[
 \sum_{i=1}^3 d(\bX_i,\cL_1)^2
 \weq \sigma^2 \sum_{i=1}^3 Y_{i2}^2,
\]
and therefore
\[
 1-p
 \weq \P\bigl(\sigma^2_\tls(\bX_1,\bX_2,\bX_3)>t^2\bigr)
 \wle \P\left(\sum_{i=1}^3 Y_{i2}^2>t^2/\sigma^2\right).
\]
Since \(Y_{12},Y_{22},Y_{32}\) are independent standard Gaussian random variables, their sum of squares has the chi-square distribution with \(3\) degrees of freedom. Applying the tail bound \eqref{eq:ChiSquare} with \(s=t^2/(3\sigma^2)\) concludes the proof.
\end{proof}

\subsubsection{Inter-community hyperedge probability}
\label{sec:InterCommunityHyperedgeProbability}

We analyze the probability of a triple $\bX_1,\bX_2,\bX_3 \in \R^2$ forming a hyperedge in the TLS hypergraph, when not all points are chosen from the same community. This occurs when exactly one point is in the community 1, and the other two are in the community 2, or vice versa. By symmetry, the two configurations are identically distributed.

\begin{lemma}
\label{lem:prob_between}
Let \((\bZ,\bX)\eqd \GLMM_N(\alpha,\nu,\sigma)\) with
\(\supp(\nu)\subset[-\ell/2,\ell/2]\), \(\ell>0\). Choose coordinates so that
\[
L_1=\{(u,0):u\in\R\}\quad \text{and}\quad L_2=\{(u\cos\alpha,u\sin\alpha):u\in\R\},
\] where
\(\alpha\in(0,\pi)\) is fixed. For the inter-community configuration
\((Z_1,Z_2,Z_3)=(1,1,2)\), write
\(\bU_1=(U_{11},0)\), \(\bU_2=(U_{21},0)\), and
\(\bU_3=(V_3\cos\alpha,V_3\sin\alpha)\), and define \(S_{12}\coloneqq |U_{11}-U_{21}|\). Assume that there exist \(\rho\in(0,1]\) and constants \(C_0,C_1,C_2\ge 1\) such that, for all
\(0\le u\le \ell\) and \(0\le a\le a+u\le \ell\),
\begin{equation}
\label{eq:smallball_V}
 \P(|V_3| \le u) \wle C_0 \left(\frac{u}{\ell}\right)^\rho,
\end{equation}
\begin{equation}
\label{eq:smallball_S_cdf}
 \P(S_{12} \le u) \wle C_1 \left(\frac{u}{\ell}\right)^\rho,
\end{equation}
and
\begin{equation}
\label{eq:smallball_S_interval}
 \P\left(S_{12} \in [a,a+u]\right) \wle C_2 \left(\frac{u}{\ell}\right)^\rho.
\end{equation}
Let \(\delta\coloneqq (t+\sigma)/\ell\). If \(\delta\le 1/(2e)\), then the inter-community hyperedge probability \(q\) of the TLS hypergraph \(H=H_t\) satisfies
\[
 q \wle c \delta^\rho \log(1/\delta),
\]
where \(C\) depends only on \(\alpha,\rho,C_0,C_1,C_2\).
\end{lemma}

\begin{proof}
Let \(\bY_1,\bY_2,\bY_3\) be independent standard normal random variables in \(\R^2\),
independent of \(\bU_1,\bU_2,\bU_3\), and define
$
 \bX_i = \bU_i + \sigma \bY_i.
$
Then
\[
 q \weq \P\left( \sigma_\tls(\bX_1,\bX_2,\bX_3)\le t \right).
\]

Denote by \(B_{\bX_i}(t)\) the closed disk with center \(\bX_i\) and radius \(t\).
If \(\sigma_\tls(\bX_1,\bX_2,\bX_3) \le t\), then there exists a line \(\cL\)
such that \(d(\bX_i,\cL) \le t\) for all \(i\), or equivalently,
\(\cL\) intersects the disks \(B_{\bX_i}(t)\), \(i=1,2,3\).
Denote
$
 R_i = \sigma \twonorm{\bY_i}+t,
$
and observe that
$
 B_{\bX_i}(t) \subset B_{\bU_i}(R_i).
$
Hence, any line that intersects the disks
\(B_{\bX_1}(t),B_{\bX_2}(t),B_{\bX_3}(t)\)
also intersects the disks
\(B_1,B_2,B_3\), where
$
 B_i = B_{\bU_i}(R_i).
$
Therefore, with
\[
 \cI
 \weq \{\text{there exists a line that intersects the disks }B_1,B_2,B_3\},
\]
we have
$
 q \wle \P(\cI).
$
Let
$
 \cD
 \weq \{\text{the disks }B_1\text{ and }B_2\text{ are disjoint}\}.
$
Then
\begin{equation}
\label{eq:InterCommunity0_sb}
 q \wle \P(\cI,\cD)+\P(\cD^c).
\end{equation}

Let
\[
 R_{12}\;\coloneqq\; R_1+R_2,
 \qquad
 A \;\coloneqq\; \frac12(R_1+R_2)+R_3 \weq \frac{R_{12}}{2}+R_3,
\]
and, on $\cD$, define
\begin{equation}
 K_{12}
 \;\coloneqq\; \frac{R_{12}}{\sqrt{(U_{11}-U_{21})^2-R_{12}^2}}.
 \label{eq:K12}
\end{equation}
On $\cD^c$, we define $K_{12} \coloneqq \infty$.
Also define
\[
 \kappa_\alpha
 \;\coloneqq\;
 \begin{cases}
 (2|\cot\alpha|)^{-1}, & \alpha\neq \pi/2,\\
 \infty, & \alpha=\pi/2.
 \end{cases}
\]
We split
\begin{equation}
\label{eq:ID_split_angle}
 \P(\cI,\cD)
 \wle
 \P(\cI,\cD,K_{12}\le \kappa_\alpha)
 +
 \P(K_{12}>\kappa_\alpha,\cD).
\end{equation}

A geometric argument (Lemma~\ref{lem:U3_vertical_bound_angle}) implies that on
\(\cI\cap\cD\cap\{K_{12}\le \kappa_\alpha\}\),
\[
 |V_3|
 \le U_\alpha^*,
\]
where
\[
 U_\alpha^*
 \weq
 \frac{2}{\sin\alpha}
 \left[
 \frac12\left(|U_{11}|+|U_{21}|\right)K_{12}
 + A(1+K_{12})
 \right].
\]
Since \(|V_3|\) is independent of \(\bU_1,\bU_2,\bY_1,\bY_2,\bY_3\), by the small-ball bound
\eqref{eq:smallball_V}, for \(0\le u\le \ell\),
\[
 \P(|V_3|\le u)\wle C_0\left(\frac{u}{\ell}\right)^\rho.
\]
Since \(C_0\ge 1\), the same bound extends trivially to all \(u\ge 0\): if \(u>\ell\), then
\[
 \P(|V_3|\le u)=1\le C_0\left(\frac{u}{\ell}\right)^\rho.
\]
Therefore
\begin{align}
 \P(\cI,\cD,K_{12}\le \kappa_\alpha)
 &\weq \E \Bigl[
 1_\cD 1(K_{12}\le \kappa_\alpha)\,
 \P\left(\cI \mid \bU_1,\bU_2,\bY_1,\bY_2,\bY_3\right)
 \Bigr] \notag\\
 &\wle \E \Bigl[
 1_\cD \,
 \P\left(|V_3| \le U_\alpha^* \mid \bU_1,\bU_2,\bY_1,\bY_2,\bY_3\right)
 \Bigr] \notag\\
 &\wle C_0 \ell^{-\rho} \E\left[(U_\alpha^*)^\rho 1_\cD\right].
 \label{eq:InterCommunity1_angle}
\end{align}

Since \(|U_{11}|,|U_{21}|\le \ell/2\), we have
\[
 U_\alpha^* 1_\cD
 \wle
 \frac{1}{\sin\alpha}\,\ell K_{12} 1_\cD
 + \frac{2}{\sin\alpha}\,A(1+K_{12})1_\cD.
\]
Since \(0<\rho\le 1\), the map \(x\mapsto x^\rho\) is subadditive, so
\[
 (x+y)^\rho \le x^\rho+y^\rho,
 \qquad
 (1+x)^\rho \le 1+x^\rho
 \qquad (x,y\ge 0).
\]
Hence
\[
 (U_\alpha^*)^\rho 1_\cD
 \wle
 C_{\alpha,\rho}\Bigl(
 \ell^\rho K_{12}^\rho 1_\cD
 + A^\rho 1_\cD
 + A^\rho K_{12}^\rho 1_\cD
 \Bigr).
\]
Also,
\[
 A^\rho
 = \left(\frac{R_{12}}{2}+R_3\right)^\rho
 \wle 2^{-\rho}R_{12}^\rho + R_3^\rho.
\]
Therefore
\begin{align*}
 \E\left[(U_\alpha^*)^\rho 1_\cD\right]
 &\wle C_{\alpha,\rho}\Bigl(
 \ell^\rho \E[K_{12}^\rho 1_\cD]
 + \E[R_{12}^\rho]
 + \E[R_3^\rho] + \E[R_{12}^\rho K_{12}^\rho 1_\cD]
 + \E[R_3^\rho]\E[K_{12}^\rho 1_\cD]
 \Bigr).
\end{align*}
Since \(R_{12}\ge R_1\) and \(R_3\stackrel{d}{=}R_1\), we have \(\E[R_3^\rho]\le \E[R_{12}^\rho]\).
Thus, with
\begin{equation}
 M_\beta \weq \E[R_{12}^\beta], \qquad \beta>0,
 \label{eq:m_beta_def}
\end{equation}
and
\begin{equation}
 J_1 \weq \E[K_{12}^\rho 1_\cD],
 \qquad
 J_2 \weq \E[R_{12}^\rho K_{12}^\rho 1_\cD],
 \label{eq:def_J12}
\end{equation}
we obtain
\begin{equation}
\label{eq:InterCommunity2_angle}
 \P(\cI,\cD,K_{12}\le \kappa_\alpha)
 \wle C_{\alpha,\rho} C_0 \left(
   \frac{M_\rho}{\ell^\rho}
   + J_1
   + \frac{M_\rho}{\ell^\rho} J_1
   + \frac{J_2}{\ell^\rho}
 \right).
\end{equation}

For the complementary event, Markov's inequality gives
\begin{equation}
\label{eq:bad_slope_event}
 \P(K_{12}>\kappa_\alpha,\cD)
 \wle \kappa_\alpha^{-\rho}\E[K_{12}^\rho 1_\cD]
 \weq \kappa_\alpha^{-\rho}J_1,
\end{equation}
with the convention \(\infty^{-\rho}=0\). Combining
\eqref{eq:ID_split_angle}, \eqref{eq:InterCommunity2_angle}, and
\eqref{eq:bad_slope_event}, and recalling that \(C_0\ge 1\), we conclude that
\begin{equation}
\label{eq:InterCommunity3_angle}
 \P(\cI,\cD)
 \wle C_{\alpha,\rho} C_0 \left(
   \frac{M_\rho}{\ell^\rho}
   + J_1
   + \frac{M_\rho}{\ell^\rho} J_1
   + \frac{J_2}{\ell^\rho}
 \right).
\end{equation}

We now bound \(J_1\) and \(J_2\).
Observe that
\begin{equation}
 \cD = \{S_{12}>R_{12}\},
 \qquad
 K_{12}^\rho 1_\cD
 = \frac{R_{12}^\rho}{(S_{12}^2-R_{12}^2)^{\rho/2}}\,1(S_{12}>R_{12}).
\label{eq:K_{12}_expr}
\end{equation}
Let \(\mu\) denote the law of \(S_{12}\), and for \(0<r\le \ell\) define
\[
H(r)\weq \int_{(r,\ell]} \frac{r^\rho}{(s^2-r^2)^{\rho/2}}\,\mu(ds).
\]
Conditioning on \(R_{12}=r\), we obtain for \(0<r\le \ell\),
\[
\E[K_{12}^\rho 1_\cD \mid R_{12}=r]=H(r),
\qquad
\E[R_{12}^\rho K_{12}^\rho 1_\cD \mid R_{12}=r]=r^\rho H(r),
\]
and both conditional expectations vanish when \(r>\ell\), since \(S_{12}\in[0,\ell]\) almost surely.

Therefore, for every \(r>0\),
\begin{equation}
\E[K_{12}^\rho 1_\cD \mid R_{12}=r]
=
H(r)\,1(r\le \ell),
\qquad
\E[R_{12}^\rho K_{12}^\rho 1_\cD \mid R_{12}=r]
=
r^\rho H(r)\,1(r\le \ell).
\label{eq:H_appl}
\end{equation}

We claim that
\begin{equation}
\label{eq:H_bound_sb}
 H(r)
 \wle C_{\rho,C_2} \left(\frac{r}{\ell}\right)^\rho
 \log\left(\frac{4\ell}{r}\right),
 \qquad 0<r\le \ell.
\end{equation}
To prove this, split \(H(r)=H_{\rm near}(r)+H_{\rm far}(r)\), where
\begin{align}
 H_{\rm near}(r)
 &\;\coloneqq\; \int_{(r,2r]\cap(r,\ell]}
 \frac{r^\rho}{(s^2-r^2)^{\rho/2}}\,\mu(ds),\label{eq:def_Hnear}\\
 H_{\rm far}(r)
 &\;\coloneqq\; \int_{(2r,\ell]}
 \frac{r^\rho}{(s^2-r^2)^{\rho/2}}\,\mu(ds).\label{eq:def_Hfar}
\end{align}

For the near part, partition \((r,2r]\) into dyadic intervals: for \(j\ge 0\),
\[
 I_j \weq \left(r+2^{-j-1}r,\; r+2^{-j}r\right],
\]
where $I_1, I_2, \dots$ are disjoint and arranged from right to left. If \(s\in I_j\), then
\[
s-r \in \left(2^{-j-1}r,\;2^{-j}r\right]
\quad\text{and}\quad
s+r \in \left(2r+2^{-j-1}r,\;2r+2^{-j}r\right],
\]
so
\[
 s^2-r^2=(s+r)(s-r)\ge 2r(s-r)\ge 2r\cdot 2^{-j-1}r=2^{-j}r^2.
\]
Therefore,
\[
 \sup_{s\in I_j}\frac{r^\rho}{(s^2-r^2)^{\rho/2}} \le 2^{j\rho/2}.
\]
Together with \eqref{eq:def_Hnear}, this gives
\begin{equation}
H_{\rm near}(r)
= \sum_{j\ge 0}\int_{I_j}\frac{r^\rho}{(s^2-r^2)^{\rho/2}}\,\mu(ds)
\wle \sum_{j\ge 0}
\left(\sup_{s\in I_j}\frac{r^\rho}{(s^2-r^2)^{\rho/2}}\right)\mu(I_j)
\wle \sum_{j\ge 0} 2^{j\rho/2}\,\mu(I_j).
\label{eq:Hnear_1}
\end{equation}
Also, by \eqref{eq:smallball_S_interval} and \(|I_j|=2^{-j-1}r\),
\[
\mu(I_j)
\wle C_2\left(\frac{|I_j|}{\ell}\right)^\rho
\weq C_2\left(\frac{2^{-j-1}r}{\ell}\right)^\rho.
\]
Therefore, \eqref{eq:Hnear_1} gives
\begin{align}
H_{\rm near}(r)
&\wle \sum_{j\ge 0} 2^{j\rho/2}\,
C_2\left(\frac{2^{-j-1}r}{\ell}\right)^\rho \nonumber\\
&\weq C_2\left(\frac{r}{\ell}\right)^\rho
\sum_{j\ge 0} 2^{j\rho/2}(2^{-j-1})^\rho \nonumber\\
&\weq C_2\,2^{-\rho}\left(\frac{r}{\ell}\right)^\rho
\sum_{j\ge 0} 2^{-j\rho/2} \nonumber\\
&\wle C_{\rho,C_2} \left(\frac{r}{\ell}\right)^\rho,
\label{eq:Hnear_eval}
\end{align}
since \(\sum_{j\ge 0}2^{-j\rho/2}\) is a convergent geometric series with \(0<2^{-\rho/2}<1\) by \(\rho>0\).

For the far part, partition \((2r,\ell]\) into dyadic intervals: let
\[
 m_r \weq \left\lceil \log_2\left(\frac{\ell}{r}\right)\right\rceil,
 \qquad
 J_k \weq (2^k r,\,2^{k+1}r]\cap(2r,\ell], \qquad 1\le k\le m_r.
\]
If \(s\in J_k\), then \(s\ge 2^k r\ge 2r\), so \(r\le s/2\) and hence
\[
 s^2-r^2 \ge s^2-\frac{s^2}{4}=\frac34 s^2 \ge \frac34\,2^{2k}r^2.
\]
Therefore
\[
\sup_{s\in J_k}\frac{r^\rho}{(s^2-r^2)^{\rho/2}}
\wle \left(\frac43\right)^{\rho/2}2^{-k\rho}
\wle C_\rho 2^{-k\rho}.
\]
Together with \eqref{eq:def_Hfar}, this gives
\begin{equation}
H_{\rm far}(r)
= \sum_{k=1}^{m_r}\int_{J_k}\frac{r^\rho}{(s^2-r^2)^{\rho/2}}\,\mu(ds)
\wle \sum_{k=1}^{m_r}
\left(\sup_{s\in J_k}\frac{r^\rho}{(s^2-r^2)^{\rho/2}}\right)\mu(J_k)
\wle C_\rho \sum_{k=1}^{m_r} 2^{-k\rho}\,\mu(J_k).
\label{eq:Hfar_1}
\end{equation}
Also, by \eqref{eq:smallball_S_interval} and \(|J_k|\le 2^k r\),
\[
\mu(J_k)
\wle C_2\left(\frac{|J_k|}{\ell}\right)^\rho
\wle C_2\left(\frac{2^k r}{\ell}\right)^\rho.
\]
Therefore, \eqref{eq:Hfar_1} gives
\begin{align}
H_{\rm far}(r)
&\wle C_\rho \sum_{k=1}^{m_r} 2^{-k\rho}\,
C_2\left(\frac{2^k r}{\ell}\right)^\rho \nonumber\\
&\weq C_{\rho,C_2} \left(\frac{r}{\ell}\right)^\rho
\sum_{k=1}^{m_r} 1 \nonumber\\
&\weq C_{\rho,C_2}\, m_r \left(\frac{r}{\ell}\right)^\rho \nonumber\\
&\wle C_{\rho,C_2} \left(\frac{r}{\ell}\right)^\rho
\left(1+\log\frac{\ell}{r}\right),
\label{eq:Hfar_eval}
\end{align}
where the last inequality uses
\[
m_r=\left\lceil \log_2\left(\frac{\ell}{r}\right)\right\rceil
\wle 1+\log_2\left(\frac{\ell}{r}\right)
\wle C\left(1+\log\frac{\ell}{r}\right),
\]
and the factor \(C\) is absorbed into \(C_{\rho,C_2}\).

Since \(H(r)=H_{\rm near}(r)+H_{\rm far}(r)\), \eqref{eq:Hnear_eval} and \eqref{eq:Hfar_eval} imply
\[
 H(r)
 \wle C_{\rho,C_2} \left(\frac{r}{\ell}\right)^\rho
 \left(1+\log\frac{\ell}{r}\right).
\]
Since \(0<r\le \ell\), we have \(\log(\ell/r)\ge 0\), and therefore
\(
1+\log\frac{\ell}{r}
\wle \log 4+\log\frac{\ell}{r}
= \log\frac{4\ell}{r}.
\)
This proves \eqref{eq:H_bound_sb}.

Consequently, \eqref{eq:def_J12}, \eqref{eq:H_appl}, and the tower property of conditional expectation yield
\begin{align}
 J_1
 &= \E\left[\E\left[K_{12}^\rho 1_\cD \mid R_{12}\right]\right]
 \nonumber\\
 &= \E\left[H(R_{12})\,1(R_{12}\le \ell)\right]
 \nonumber\\
 &\wle C_{\rho,C_2} \ell^{-\rho}
 \E\left[ R_{12}^\rho \log\left(\frac{4\ell}{R_{12}}\right) 1(R_{12}\le \ell)\right],
 \label{eq:J1_sb}
\\
 J_2
 &= \E\left[\E\left[R_{12}^\rho K_{12}^\rho 1_\cD \mid R_{12}\right]\right]
 \nonumber\\
 &= \E\left[R_{12}^\rho H(R_{12})\,1(R_{12}\le \ell)\right]
 \nonumber\\
 &\wle C_{\rho,C_2} \ell^{-\rho}
 \E\left[ R_{12}^{2\rho} \log\left(\frac{4\ell}{R_{12}}\right) 1(R_{12}\le \ell)\right].
 \label{eq:J2_sb}
\end{align}

Next we estimate moments of
\(
 R_{12} = 2t + \sigma \twonorm{\bY_1} + \sigma \twonorm{\bY_2}.
\)
Since \(\bY_1 \eqd N(0,I_2)\), the random variable \(\twonorm{\bY_1}\) is Rayleigh with density
\(
f(r)= r e^{-r^2/2}1(r\ge 0).
\)
Hence, for every \(\beta>0\),
\[
\E \twonorm{\bY_1}^\beta
\weq \int_0^\infty r^\beta f(r)\,dr
\weq \int_0^\infty r^{\beta+1} e^{-r^2/2}\,dr
\weq 2^{\beta/2}\Gamma\left(1+\frac{\beta}{2}\right)
<\infty.
\]
Thus \(\twonorm{\bY_1}\) has finite moments of every order. The same holds for \(\twonorm{\bY_2}\).
We claim that
\begin{equation}
(a+b+c)^\beta \wle c_\beta(a^\beta+b^\beta+c^\beta),
\qquad a,b,c\ge 0,
\label{eq:sub_conv}
\end{equation}
where
\[
c_\beta \weq
\begin{cases}
1, & 0<\beta\le 1,\\
3^{\beta-1}, & \beta>1.
\end{cases}
\]
Indeed, if \(0<\beta\le 1\), then \(x\mapsto x^\beta\) is subadditive on \([0,\infty)\), so
\(
(a+b+c)^\beta \wle a^\beta+b^\beta+c^\beta.
\)
If \(\beta>1\), then \(x\mapsto x^\beta\) is convex, and Jensen's inequality gives
\[
\left(\frac{a+b+c}{3}\right)^\beta
\wle \frac{a^\beta+b^\beta+c^\beta}{3}.
\]
For \(\beta\in\{\rho,2\rho\}\), \eqref{eq:m_beta_def} and \eqref{eq:sub_conv} give
\begin{align}
M_\beta
&\weq \E\left(2t+\sigma\twonorm{\bY_1}+\sigma\twonorm{\bY_2}\right)^\beta \nonumber\\
&\wle c_\beta\left((2t)^\beta + \sigma^\beta \E\twonorm{\bY_1}^\beta
+ \sigma^\beta \E\twonorm{\bY_2}^\beta\right) \nonumber\\
&\weq c_\beta\left(2^\beta t^\beta
+ 2^{1+\beta/2}\Gamma\left(1+\frac{\beta}{2}\right)\sigma^\beta\right) \nonumber\\
&\wle C_\beta (t+\sigma)^\beta
\weq C_\beta \ell^\beta \delta^\beta,
\label{eq:Mbeta_sb}
\end{align}
where \(C_\beta\) depends only on \(\beta\).

Also, for any \(\beta>0\) and any \(0<r\le \ell\),
\[
 r^\beta \log\left(\frac{4\ell}{r}\right)
 = r^\beta \log\left(\frac{4}{\delta}\right)
 + r^\beta \log\left(\frac{\ell\delta}{r}\right).
\]
The second term is nonpositive when \(r\ge \ell\delta = t+\sigma\), while for \(0<r\le \ell\delta\),
the function \(u\mapsto u^\beta \log(1/u)\) is bounded by \((e\beta)^{-1}\) on \((0,1]\).
Hence
\[
 r^\beta \log\left(\frac{4\ell}{r}\right)
 \wle r^\beta \log\left(\frac{4}{\delta}\right)
 + \frac{(\ell\delta)^\beta}{e\beta}.
\]
Taking expectations and using \eqref{eq:Mbeta_sb}, we obtain
\begin{equation}
\label{eq:Rbeta_log_sb}
 \E\left[ R_{12}^\beta \log\left(\frac{4\ell}{R_{12}}\right) 1(R_{12}\le \ell)\right]
 \wle C_\beta \ell^\beta \delta^\beta \log(1/\delta),
 \qquad \beta\in\{\rho,2\rho\},
\end{equation}
where we also used \(\delta\le 1/(2e)\) to absorb \(\log(4/\delta)\) into \(C\log(1/\delta)\).

Combining \eqref{eq:J1_sb}, \eqref{eq:J2_sb}, and \eqref{eq:Rbeta_log_sb}, we get
\begin{equation}
\label{eq:Jbounds_sb}
 J_1 \wle C_{\rho,C_2}\,\delta^\rho \log(1/\delta),
 \qquad
 J_2 \wle C_{\rho,C_2}\,\ell^\rho \delta^{2\rho} \log(1/\delta).
\end{equation}
Substituting \eqref{eq:Mbeta_sb} and \eqref{eq:Jbounds_sb} into
\eqref{eq:InterCommunity3_angle}, we conclude that
\begin{equation}
\label{eq:InterCommunity3_sb}
 \P(\cI,\cD)
 \wle C_{\alpha,\rho,C_0,C_2}
 \Bigl(
   \delta^\rho
   + \delta^\rho \log(1/\delta)
   + \delta^{2\rho} \log(1/\delta)
 \Bigr).
\end{equation}

It remains to bound \(\P(\cD^c)\).
Since \(\cD^c = \{S_{12}\le R_{12}\}\) and \(S_{12}\) is independent of \(R_{12}\), the cumulative small-ball bound
\eqref{eq:smallball_S_cdf} yields
\[
 \P(\cD^c)
 = \E\left[\P(S_{12}\le R_{12} \mid R_{12})\right]
 \wle C_1 \ell^{-\rho}\E[R_{12}^\rho 1(R_{12}\le \ell)] + \P(R_{12}>\ell).
\]
Since \(R_{12}^\rho/\ell^\rho \ge 1\) on \(\{R_{12}>\ell\}\), we also have
\[
 \P(R_{12}>\ell) \wle \ell^{-\rho}\E[R_{12}^\rho 1(R_{12}>\ell)].
\]
Therefore
\begin{equation}
\label{eq:InterCommunity4_sb}
 \P(\cD^c)
 \wle C_1 \ell^{-\rho} \E[R_{12}^\rho]
 \wle C_{\rho,C_1}\,\delta^\rho.
\end{equation}

Finally, combining \eqref{eq:InterCommunity0_sb}, \eqref{eq:InterCommunity3_sb},
and \eqref{eq:InterCommunity4_sb}, and using
\[
 \delta^{2\rho}\log(1/\delta) \le \delta^\rho\log(1/\delta),
 \qquad
 \delta^\rho \le \delta^\rho\log(1/\delta)
 \qquad (\delta\le e^{-1}),
\]
we obtain
\(
 q \wle c \delta^\rho \log(1/\delta),
\)
for a constant \(C\) depending only on \(\alpha,\rho,C_0,C_1,C_2\).
\end{proof}

\begin{corollary}
\label{cor:coordinate_smallball_common}
Let \((\bZ,\bX)\eqd \GLMM_N(\alpha,\nu,\sigma)\) with
\(
\supp(\nu)\subset[-\ell/2,\ell/2],
\)
\(\ell>0\).
Assume that for some \(\rho\in(0,1]\) and \(A\ge 1\),
\[
 \nu(I)\wle A\left(\frac{|I|}{\ell}\right)^\rho
\]
for every interval \(I\subset[-\ell/2,\ell/2]\). If \((t+\sigma)/\ell\le 1/(2e)\), then the
inter-community hyperedge probability \(q\) of the TLS hypergraph \(H=H_t\) satisfies
\[
 q \wle C\left(\frac{t+\sigma}{\ell}\right)^\rho
 \log\left(\frac{\ell}{t+\sigma}\right),
\]
where \(C\) depends only on \(\alpha,\rho,\) and \(A\).
\end{corollary}
\begin{proof}
Let \(V_1,V_2,V_3\) denote the signed coordinates of the three latent locations measured from the
common midpoint along their respective line segments, so that \(V_1,V_2,V_3\) are independent with
common law \(\nu\). Define
\(
S_{12}\coloneqq |V_1-V_2|.
\)
We verify the assumptions of Lemma~\ref{lem:prob_between}.

First, for \(0\le u\le \ell\),
\[
\P(|V_3|\le u)
=
\nu([-u,u]\cap[-\ell/2,\ell/2])
\wle
A\left(\frac{2u}{\ell}\right)^\rho.
\]
Thus \eqref{eq:smallball_V} holds with \(C_0=2^\rho A\).

Next, conditioning on \(V_1\), we obtain
\[
\P(S_{12}\le u\mid V_1)
=
\nu([V_1-u,V_1+u]\cap[-\ell/2,\ell/2])
\wle
A\left(\frac{2u}{\ell}\right)^\rho,
\]
and hence
\[
\P(S_{12}\le u)\wle 2^\rho A\left(\frac{u}{\ell}\right)^\rho.
\]
So \eqref{eq:smallball_S_cdf} holds with \(C_1=2^\rho A\).

Similarly, for \(0\le a\le a+u\le \ell\),
\[
\{S_{12}\in[a,a+u]\}
\subset
\{V_2\in[V_1-a-u,V_1-a]\}\cup\{V_2\in[V_1+a,V_1+a+u]\},
\]
so conditioning on \(V_1\) gives
\[
\P(S_{12}\in[a,a+u]\mid V_1)\wle 2A\left(\frac{u}{\ell}\right)^\rho.
\]
Taking expectations yields \eqref{eq:smallball_S_interval} with \(C_2=2A\).

Therefore Lemma~\ref{lem:prob_between} applies and gives
\[
q \wle c \delta^\rho\log(1/\delta),
\qquad
\delta\coloneqq \frac{t+\sigma}{\ell},
\]
where \(C\) depends only on \(\alpha,\rho,\) and \(A\).
\end{proof}

We conclude the section by deriving a lower bound on the probability of accepting an inter-community triple: if the distance between a pair of observations is less than $t \sqrt{2}$, then the pair forms hyperedges with each remaining node.

\begin{lemma}
\label{lem:q_lower_bound}
Assume \(\sigma_N=o(1)\). Let \((\bZ,\bX)\eqd \GLMM_N(\alpha,\nu,\sigma_N)\), and fix
\(\ell>0\) such that \(\supp(\nu)\subset[-\ell/2,\ell/2]\). Assume that there exist
\(\rho\in(0,1]\), \(B>0\), and \(u_0>0\) such that
\[
 \nu([-u,u]) \wge B\left(\frac{u}{\ell}\right)^\rho,
 \qquad 0\le u\le u_0.
\]
Then, for every fixed \(\varepsilon>0\), there exists a constant \(c_\varepsilon>0\) such that
\[
 q(t)\ge c_\varepsilon
\]
for all \(t\ge \varepsilon\) and all sufficiently large \(N\).
\end{lemma}
\begin{proof}
Since \(q(t)\) is nondecreasing in \(t\), it is enough to lower bound \(q(\varepsilon)\).

By symmetry, we may consider \((Z_1,Z_2,Z_3)=(1,1,2)\). Choose
\(L_1=[-\ell/2,\ell/2]\times\{0\}\), and write
\(\bU_1=(V_1,0)\), \(\bU_2=(V_2,0)\), where \(V_1,V_2\) are independent with law \(\nu\).

Let \(\bM=\frac12(\bX_1+\bX_2)\), and let \(L\) be a line through \(\bM\) and \(\bX_3\). Then
\(d(\bX_3,L)=0\) and \(d(\bX_1,L)=d(\bX_2,L)\le \frac12\|\bX_1-\bX_2\|_2\). Hence
\[
\sigma_\tls(\bX_1,\bX_2,\bX_3)
\le \frac{1}{\sqrt2}\|\bX_1-\bX_2\|_2
\le \frac{1}{\sqrt2}\left(\|\bU_1-\bU_2\|_2+\sigma_N\|\bY_1-\bY_2\|_2\right).
\]
Therefore, by independence,
\begin{align}
q(\varepsilon)
&\weq \P\left(\sigma_\tls(\bX_1,\bX_2,\bX_3)\le \varepsilon \,\middle|\, Z_1=1,Z_2=1,Z_3=2\right) \nonumber\\
&\wge
\P\left(\frac{1}{\sqrt2}\|\bU_1-\bU_2\|_2\le \frac \varepsilon 2\right)
\P\left(\frac{\sigma_N}{\sqrt2}\|\bY_1-\bY_2\|_2\le \frac \varepsilon 2\right).
\label{eq:q_lower_bound_product_short}
\end{align}

Now let \(a_\varepsilon\coloneqq \min\{\varepsilon/(2\sqrt2),u_0\}\). On the event \(\{|V_1|\le a_\varepsilon,\ |V_2|\le a_\varepsilon\}\),
\[
\|\bU_1-\bU_2\|_2=|V_1-V_2|
\le |V_1|+|V_2|
\le 2a_\varepsilon
\le \frac{\varepsilon}{\sqrt2}.
\]
Thus
\[
\P\left(\frac{1}{\sqrt2}\|\bU_1-\bU_2\|_2\le \frac \varepsilon 2\right)
\wge
\P(|V_1|\le a_\varepsilon,\ |V_2|\le a_\varepsilon)
=
\nu([-a_\varepsilon,a_\varepsilon])^2
\wge
B^2\left(\frac{a_\varepsilon}{\ell}\right)^{2\rho}.
\]
Also, since \(\bY_1-\bY_2\eqd N(0,2I_2)\), the norm \(\|\bY_1-\bY_2\|_2\) is Rayleigh with scale
parameter \(\sqrt2\), and hence
\[
\P\left(\frac{\sigma_N}{\sqrt2}\|\bY_1-\bY_2\|_2\le \frac \varepsilon 2\right)
=
1-\exp\left(-\frac{\varepsilon^2}{8\sigma_N^2}\right).
\]
Substituting these bounds into \eqref{eq:q_lower_bound_product_short} gives
\[
q(\varepsilon) \wge
B^2\left(\frac{a_\varepsilon}{\ell}\right)^{2\rho}
\left(1-\exp\left(-\frac{\varepsilon^2}{8\sigma_N^2}\right)\right).
\]
Since \(\sigma_N=o(1)\), the last factor converges to \(1\). Therefore, for all sufficiently large \(N\),
\[
q(\varepsilon)\wge \frac12\,B^2\left(\frac{a_\varepsilon}{\ell}\right)^{2\rho}.
\]
Hence, if we define
\[
c_\varepsilon \coloneqq \frac12\,B^2
\left(
\min\left\{\frac{\varepsilon}{2\sqrt2\,\ell},\frac{u_0}{\ell}\right\}
\right)^{2\rho},
\]
then \(q(\varepsilon)\ge c_\varepsilon\) for all sufficiently large \(N\). Since \(q(t)\) is nondecreasing in \(t\), it follows that
$
q(t)\ge q(\varepsilon)\ge c_\varepsilon
$
for all 
$t\ge \varepsilon$
and all sufficiently large \(N\), as claimed.
\end{proof}

\subsection{Analysis of the spectral algorithm}

We prove that the spectral algorithm outlined in Algorithm \ref{alg:line_clustering_spectral} achieves exact recovery of the community labels under the given model. Specifically, we apply the Davis--Kahan theorem to bound the distance between the estimated and true subspaces spanned by the leading eigenvectors of the expected similarity matrix, which is required for establishing the consistency of the algorithm.

\subsubsection{Concentration of the similarity matrix}
The following lemma establishes the concentration of the similarity matrix under the Frobenius norm.

\begin{lemma}
\label{lem:similarity_concentration}
Let \((\bZ,\bX)\eqd\GLMM_N(\alpha,\nu,\sigma_N)\), and fix \(\ell>0\) such that
\(
\supp(\nu)\subset[-\ell/2,\ell/2].
\)
Let \(H\) be the TLS hypergraph with threshold \(t_N\), let \(\bW\) be the corresponding
similarity matrix, and write
\(
 \bW^* \coloneqq \E[\bW\given \bZ].
\)
Write \(p\) and \(q\) for the intra- and inter-community hyperedge probabilities of \(H\).
\begin{itemize}
\item[(i)]
If
\(
  1-p\ll N^{-1}
\)
and
\(
  q\ll N^{-1},
\)
then
\begin{equation}
\label{it:conc_strong}
  \E\left[\|\bW-\bW^*\|_F^2\right] \ll N^3.
\end{equation}

\item[(ii)]
If
\(
  1-p\ll 1
\)
and
\(
  q\ll 1,
\)
then
\begin{equation}
\label{it:conc_weak}
  \E\left[\|\bW-\bW^*\|_F^2\right] \ll N^4.
\end{equation}
\end{itemize}
\end{lemma}
\begin{proof}
Set
\[
  \vartheta_N \coloneqq (1-p)+q.
\]
Since \(\bW\) is symmetric with zero diagonal,
\begin{equation}    
  \E\left[\|\bW-\bW^*\|_F^2 \given \bZ\right]
  = \sum_{i,j}\E\left[(W_{ij}-W^*_{ij})^2 \given \bZ\right]
  = 2\sum_{i<j}\Var(W_{ij}\given \bZ).
\label{eq:conc_W_decom}
\end{equation}

Fix \(i<j\). Since
\[
W_{ij}=\sum_{k\notin\{i,j\}} A_{ijk},
\]
we have
\[
\Var(W_{ij}\given \bZ)
=
\sum_{k\notin\{i,j\}}\Var(A_{ijk}\given \bZ)
+
\sum_{\substack{k,\ell\notin\{i,j\}\\k\neq \ell}}
\cov(A_{ijk},A_{ij\ell}\given \bZ).
\]
There are \(N-2=O(N)\) variance terms and \(O(N^2)\) covariance terms.

If \(Z_i=Z_j\), then each \(A_{ijk}\) is Bernoulli with parameter either \(p\) or \(q\), according
to whether \(Z_k=Z_i\) or \(Z_k\neq Z_i\). Hence
\[
\Var(A_{ijk}\given \bZ)\le \max\{p(1-p),q(1-q)\}\le \vartheta_N.
\]
Similarly, if \(Z_i\neq Z_j\), then every triple \((i,j,k)\) is inter-community, so
\(A_{ijk}\) is Bernoulli with parameter \(q\), and again
\[
\Var(A_{ijk}\given \bZ)=q(1-q)\le \vartheta_N.
\]
By Cauchy--Schwarz, every covariance term satisfies
\[
|\cov(A_{ijk},A_{ij\ell}\given \bZ)|
\le
\sqrt{\Var(A_{ijk}\given \bZ)\Var(A_{ij\ell}\given \bZ)}
\le \vartheta_N.
\]
Therefore there exists a constant \(C<\infty\) such that
\[
  \sup_{1\le i<j\le N}\Var(W_{ij}\given \bZ)\le C N^2 \vartheta_N.
\]
Substituting this into \eqref{eq:conc_W_decom} yields
\[
  \E\left[\|\bW-\bW^*\|_F^2 \given \bZ\right]
  \le 2\binom{N}{2} C N^2 \vartheta_N
  \lesssim N^4 \vartheta_N.
\]
Taking expectations over \(\bZ\), we conclude that
\begin{equation}
\label{eq:frob_conc_general}
  \E\left[\|\bW-\bW^*\|_F^2\right]
  \lesssim N^4 \vartheta_N.
\end{equation}
In case (i), \(\vartheta_N\ll N^{-1}\), and
\eqref{eq:frob_conc_general} gives
\[
\E\left[\|\bW-\bW^*\|_F^2\right]\ll N^3.
\]
In case (ii), \(\vartheta_N\ll 1\), and
\eqref{eq:frob_conc_general} gives
\[
\E\left[\|\bW-\bW^*\|_F^2\right]\ll N^4.
\]
This proves \eqref{it:conc_strong} and \eqref{it:conc_weak}.
\end{proof}

\begin{corollary}
\label{cor:similarity_concentration_conditions}
Under the assumptions of Theorem~\ref{thm:almost_main}\,{(i)}, Lemma~\ref{lem:similarity_concentration}\,{(i)} applies. Likewise, under the assumptions of
Theorem~\ref{thm:almost_main}\,{(ii)}, Lemma~\ref{lem:similarity_concentration}\,{(ii)} applies.
\end{corollary}
\begin{proof}
In case (i), the condition
\[
  \sigma_N\sqrt{\log N} \ll t_N \ll \frac{\ell}{(N\log N)^{1/\rho}}
\]
implies \(1-p\ll N^{-1}\) by the Lemma~\ref{lem:_within_probability} and \(q\ll N^{-1}\) by
Corollary~\ref{cor:coordinate_smallball_common}. In case (ii), the condition
\[
  \sigma_N \ll t_N \ll \ell
\]
implies \(1-p\ll 1\) by Lemma~\ref{lem:_within_probability} and \(q\ll 1\) by
Corollary~\ref{cor:coordinate_smallball_common}.
\end{proof}

\subsubsection{Perturbation bound for leading eigenvectors}
\label{sec:perturbation_bound}

We restrict attention to (approximately) balanced labelings.
Let $Z_i \eqd \mathrm{Unif}\{1,2\}$ be i.i.d.\ and define
$N_1 = \sum_{i=1}^N \1\{Z_i=1\}$ and $N_2 = N-N_1$.
By Lemma~\ref{lem:Ni_concentration}, with the choice
\begin{equation}
  \eta_N \coloneqq \sqrt{\frac{2\log N}{N}},
  \label{eq:eta_N_choice}
\end{equation}
we have
\[
  \P\left( \left|N_1 - \tfrac{N}{2}\right| \le \eta_N N \right)
  \ge 1 - 2\exp\left(-\tfrac{2}{3}\,\eta_N^2 N\right)
  = 1 - 2 N^{-4/3}.
\]
Throughout this section, we condition on a fixed realization $\bZ=(Z_1,\dots,Z_N)\in\{1,2\}^N$
satisfying the (approximate balance) event
\begin{equation}
  G \;\coloneqq\; \left\{\left|N_1-\tfrac{N}{2}\right|\le \eta_N N\right\}.
  \label{eq:approx_balance}
\end{equation}
On $G$ we have
\begin{equation}
N_1=\frac{N}{2}\left(1 + O(\eta_N)\right),
\qquad
N_2=\frac{N}{2}\left(1 + O(\eta_N)\right),
\qquad
\sqrt{N_1N_2}=\frac{N}{2}\left(1 + O\left(\eta_N^2\right)\right).
\label{eq:balance_asymp_basic}
\end{equation}
Moreover, the estimates in this section are used only in regimes where \(1-p=o(1)\) and \(q=o(1)\); by
Lemma~\ref{lem:_within_probability} and Corollary~\ref{cor:coordinate_smallball_common}, this
holds under the assumptions of Theorem~\ref{thm:almost_main}. In particular, \(p+q=1+o(1)\).

Let
\(
C_1\coloneqq\{i:Z_i=1\},
\)
\(
C_2\coloneqq\{i:Z_i=2\},
\)
and define the orthonormal block-indicator vectors
\begin{equation}
\bu_1\coloneqq \frac{\1_{C_1}}{\sqrt{N_1}},
\qquad
\bu_2\coloneqq \frac{\1_{C_2}}{\sqrt{N_2}}.
\label{eq:u1u2_def}
\end{equation}

Let $\bW$ be the similarity matrix constructed in the algorithm. We use the convention $W_{ii}=0$ for all $i$, and define the conditional mean
\[
  \bW^* \;\coloneqq\; \E[\bW \given \bZ].
\]
For $i\neq j$, conditional on $\bZ$,
\[
  \E[W_{ij}\given \bZ]
  =
  \begin{cases}
    w_{\mathrm{in}}^{(1)}, & i,j\in C_1,\\[2pt]
    w_{\mathrm{in}}^{(2)}, & i,j\in C_2,\\[2pt]
    w_{\mathrm{out}}, & i\in C_1,\ j\in C_2 \ \text{or}\ i\in C_2,\ j\in C_1,
  \end{cases}
\]
where $w_{\mathrm{out}}=(N-2)q$ and
\begin{equation}
w_{\mathrm{in}}^{(1)}=(N_1-2)p+N_2 q,
\qquad
w_{\mathrm{in}}^{(2)}=(N_2-2)p+N_1 q.
\label{eq:win_two}
\end{equation}
Equivalently, after permuting indices so that $C_1$ comes first and $C_2$ second, $\bW^*$ has the block-constant form
\begin{equation}
\bW^*
=
\begin{bmatrix}
w_{\mathrm{in}}^{(1)}\left(\bJ_{N_1}-\bm{I}_{N_1}\right) & w_{\mathrm{out}}\bJ_{N_1\times N_2}\\[2pt]
w_{\mathrm{out}}\bJ_{N_2\times N_1} & w_{\mathrm{in}}^{(2)}\left(\bJ_{N_2}-\bm{I}_{N_2}\right)
\end{bmatrix},
\label{eq:Wstar_block_form}
\end{equation}
where $\bJ_{m\times \ell}$ and $\bJ_{m}$ denote the all-ones $m\times \ell$ and $m\times m$ matrices, and $\bm{I}_m$ is the $m\times m$ identity matrix.

Under \(1-p=o(1)\) and \(q=o(1)\), we have \(p+q=1+o(1)\), so on \(G\), \eqref{eq:win_two} and
\eqref{eq:balance_asymp_basic} give\begin{equation}
w_{\mathrm{in}}^{(1)} = (N-2)\frac{p+q}{2}\,(1+O(\eta_N)),
\quad
w_{\mathrm{in}}^{(2)} = (N-2)\frac{p+q}{2}\,(1+O(\eta_N)),
\quad
w_{\mathrm{out}}=(N-2)q,
\label{eq:win_wout_asymp}
\end{equation}
and hence $w_{\mathrm{in}}^{(1)}-w_{\mathrm{out}} = \frac{N-2}{2}(p-q)\,(1+O(\eta_N))$ and $w_{\mathrm{in}}^{(2)}-w_{\mathrm{out}} = \frac{N-2}{2}(p-q)\,(1+O(\eta_N))$.

Define
$\bQ\coloneqq [\bu_1\ \ \bu_2].$
From \eqref{eq:u1u2_def} and \eqref{eq:Wstar_block_form},
\[
\bW^*\bu_1
=
w_{\mathrm{in}}^{(1)}(N_1-1)\bu_1
+
w_{\mathrm{out}}\sqrt{N_1N_2}\,\bu_2,
\qquad
\bW^*\bu_2
=
w_{\mathrm{out}}\sqrt{N_1N_2}\,\bu_1
+
w_{\mathrm{in}}^{(2)}(N_2-1)\bu_2,
\]
so $\bW^*\bQ=\bQ\bB$ with
\begin{equation}
\bB
=
\begin{bmatrix}
a & c\\[2pt]
c & d
\end{bmatrix},
\qquad
a\coloneqq w_{\mathrm{in}}^{(1)}(N_1-1),\ \ 
d\coloneqq w_{\mathrm{in}}^{(2)}(N_2-1),\ \ 
c\coloneqq w_{\mathrm{out}}\sqrt{N_1N_2}.
\label{eq:B_def}
\end{equation}
Since $\bB$ is symmetric, it has an orthonormal eigenbasis $\bx_1,\bx_2\in\R^2$ with eigenvalues
$\lambda_1^*,\lambda_2^*$. In particular,
\begin{equation}
\lambda_1^*
=
\frac{a+d}{2}
+
\sqrt{\left(\frac{a-d}{2}\right)^2+c^2},
\qquad
\lambda_2^*
=
\frac{a+d}{2}
-
\sqrt{\left(\frac{a-d}{2}\right)^2+c^2}.
\label{eq:lambda12_B_explicit}
\end{equation}
On $G$, \eqref{eq:win_wout_asymp}, \eqref{eq:balance_asymp_basic} and \eqref{eq:B_def} give
\[
a=\frac{N}{2}w_{\mathrm{in}}^{(1)}(1+O(\eta_N)),
\qquad
d=\frac{N}{2}w_{\mathrm{in}}^{(2)}(1+O(\eta_N)),
\qquad
c=\frac{N}{2}w_{\mathrm{out}}(1+O(\eta_N)),
\]
so $a=d(1+O(\eta_N))$ and \eqref{eq:lambda12_B_explicit} yields
\[
\lambda_2^*=\frac{N}{2}\left(w_{\mathrm{in}}-w_{\mathrm{out}}\right)\,(1+O(\eta_N)),
\qquad
\text{where }w_{\mathrm{in}}\coloneqq\min\{w_{\mathrm{in}}^{(1)},\,w_{\mathrm{in}}^{(2)}\}.
\]

For any eigenpair $(\lambda,\bx)$ of $\bB$, we have
\begin{equation}
\bW^*(\bQ\bx)=\bQ(\bB\bx)=\lambda(\bQ\bx).
\label{eq:lift_eig}
\end{equation}
Moreover, if $N_1,N_2\ge 1$, then $\bQ=[\bu_1\ \bu_2]$ has full column rank, so $\bQ\bx\neq \0$. Thus $\lambda_1^*,\lambda_2^*$ are eigenvalues of $\bW^*$ and
$\mathrm{range}(\bQ)=\mathrm{span}\{\bu_1,\bu_2\}$ contains the corresponding eigenvectors. We next show that $\lambda_1^*, \lambda_2^*$ are the top two eigenvalues of $\bW^*$, so the top-$2$ eigenspace of $\bW^*$ equals
$\mathrm{span}\{\bu_1,\bu_2\}$. For $C_1$, take $\bv$ supported on $C_1$ and orthogonal to $\bu_1$. The constraint $\langle \bu_1,\bv\rangle=0$ is exactly $\sum_{j\in C_1} v_j=0$, so it removes one degree of freedom from the
$N_1$ free coordinates on $C_1$; hence this subspace has dimension $N_1-1$.
For such $\bv$ and $i\in C_1$,
\[
(\bW^*\bv)_i
=\sum_{\substack{j\in C_1\\ j\neq i}} w_{\mathrm{in}}^{(1)}v_j
= w_{\mathrm{in}}^{(1)}\left(\sum_{j\in C_1}v_j-v_i\right)
=-w_{\mathrm{in}}^{(1)}v_i,
\]
so $-w_{\mathrm{in}}^{(1)}$ has multiplicity $N_1-1$. Similarly, $-w_{\mathrm{in}}^{(2)}$ has multiplicity $N_2-1$. Since \(w_{\mathrm{in}}^{(1)},w_{\mathrm{in}}^{(2)}\ge 0\), the remaining eigenvalues of \(\bW^*\)
are \(-w_{\mathrm{in}}^{(1)}\) and \(-w_{\mathrm{in}}^{(2)}\), both nonpositive. Moreover, under
\(1-p=o(1)\) and \(q=o(1)\), \eqref{eq:win_wout_asymp} yields
\[
\lambda_2^*=\frac{N}{2}\left(w_{\mathrm{in}}-w_{\mathrm{out}}\right)(1+O(\eta_N))>0
\]
for all sufficiently large \(N\). Hence the eigenvalues in \eqref{eq:lambda12_B_explicit} are the
top two eigenvalues of \(\bW^*\).

On $G$, \eqref{eq:win_wout_asymp} gives $w_{\mathrm{in}}^{(1)}=w_{\mathrm{in}}^{(2)}(1+O(\eta_N))$. Hence, we have
\begin{equation}
\lambda_3^*
\;\coloneqq\;
\max\{-w_{\mathrm{in}}^{(1)},-w_{\mathrm{in}}^{(2)}\}
=
-w_{\mathrm{in}}(1+O(\eta_N)),
\qquad
\Delta \;\coloneqq\; \lambda_2^*-\lambda_3^*.
\label{eq:Delta_def}
\end{equation}
Consequently,
\begin{equation}
\Delta
=
\lambda_2^*-\lambda_3^*
=
\frac{N}{2}\left(w_{\mathrm{in}}-w_{\mathrm{out}}\right)\,(1+O(\eta_N))
=
\frac{N(N-2)}{4}(p-q)\,(1+O(\eta_N)).
\label{eq:spectral_gap}
\end{equation}

Let $\bU^*\in\R^{N\times 2}$ be a matrix with orthonormal columns spanning the top-$2$ eigenspace of $\bW^*$,
and let $\bU\in\R^{N\times 2}$ be the corresponding matrix for $\bW$.
Define the $\sin\Theta$ distance by
\[
  \| \sin \Theta(\bU, \bU^*) \|_F
  \;=\;
  \sqrt{2 - \| \bU^\top \bU^* \|_F^2}.
\]
The Frobenius Davis--Kahan theorem (see, e.g.\ \cite{pensky2024davis}) yields
\begin{equation}
  \| \sin \Theta(\bU, \bU^*) \|_F
  \;\leq\;
  2\,\Delta^{-1}\,\|\bW-\bW^*\|_F.
\label{eq:davis_kahan}
\end{equation}
Moreover,
\begin{equation}
  \inf_{\bO \in \cO_2} \| \bU - \bU^* \bO \|_F
  \;\leq\;
  \sqrt{2}\, \| \sin \Theta(\bU, \bU^*) \|_F,
\label{def:O^*}
\end{equation}
and the infimum is attained at some $\bO^*\in\cO_2$. Combining \eqref{eq:davis_kahan} and \eqref{def:O^*} gives
\begin{equation}
  \| \bU - \bU^* \bO^* \|_F
  \;\leq\;
  2\sqrt{2}\,\Delta^{-1}\,\|\bW-\bW^*\|_F,
\label{eq:perturbation_bound}
\end{equation}
which is the perturbation estimate used in the $K$-means step.

\subsubsection{\texorpdfstring{$K$}{K}-means clustering}
\label{subsec:kmeans-clustering}
Let $\mathbb{Z}_{N,K}$ denote the space of membership matrices, that is, $N \times K$ matrices with entries in $\{0, 1\}$ for which each row $i$ has only one non-zero element. According to \cite{kumar2004simple}, there exists a polynomial-time algorithm that, given $\bU \in \R^{N \times K}$ and any $\varepsilon > 0$, produces a pair $(\hat{\bZmat}, \hat{\bX}) \in \mathbb{Z}_{N,K} \times \mathbb{R}^{K \times K}$ satisfying
\begin{equation}
\| \hat{\bZmat}\hat{\bX} - \bU \|_F^2 \leq (1 + \varepsilon)\min_{\bZmat \in \mathbb{Z}_{N,K}, \bX \in \mathbb{R}^{K \times K}} \|\bZmat \bX - \bU \|_F^2.
\label{eq:k-means-objective}
\end{equation}
Each node $i$ is then assigned to the cluster $k$ for which $\hat{\bZmat}_{ik} = 1.$ The following lemma gives a deterministic bound for the permutation-invariant Hamming distance (see Equation \eqref{def:ham_star}) of a $K$-means solution.

\begin{lemma}[Adapted from Lemma 5.3 of \cite{lei2015consistency} by \cite{avrachenkov2022statistical}.]
\label{lem:$K$-means}
Fix $\varepsilon > 0$ and consider two matrices $\bU, \bV \in \mathbb{R}^{N \times K}$, where $\bV = \bZmat \bX$ for some $\bZmat \in \mathbb{Z}_{N,K}$ and $\bX \in \mathbb{R}^{K \times K}$. Let $(\hat{\bZmat}, \hat{\bX})$ be a $(1 + \varepsilon)$-approximate solution to the $K$-means objective \eqref{eq:k-means-objective}. Define $\bz$ and $\hat{\bz}$ as the corresponding membership vectors derived from $\bZmat$ and $\hat{\bZmat}$, respectively. Let $N_{\min}$ denote the smallest cluster size, and define the separation $\delta = \min_{k \neq \ell} \| \bX_{k\ast} - \bX_{\ell\ast} \|_2.$ If the following inequality holds:
\begin{equation}
4(2 + \varepsilon)\frac{\|\bU - \bV\|_F^2}{\delta^2} \leq N_{\min},
\label{ineq:$K$-means_condition}
\end{equation}
then the proportion of misclassified memberships, up to label permutations, is bounded by
\begin{equation} 
\Ham^*(\hat{\bz}, \bz) \leq 4(2 + \varepsilon)^2 \frac{\|\bU - \bV\|_F^2}{\delta^2}.
\label{eq:error_bound}
\end{equation}
\end{lemma}

\subsubsection{Proof of Theorem \ref{thm:almost_main}}
\label{sec:proof_2}

Recall the approximate balance event \(G\) from \eqref{eq:approx_balance}. On \(G\), the spectral gap \(\Delta\) defined in
\eqref{eq:Delta_def} satisfies \eqref{eq:spectral_gap}, and the subspace perturbation bound \eqref{eq:perturbation_bound} holds with
\(\bE\coloneqq \bW-\bW^*\):
\begin{equation}
\inf_{\bO\in\cO_2}\|\bU-\bU^*\bO\|_F^2
\;\le\;
8\,\Delta^{-2}\,\|\bE\|_F^2.
\label{eq:perturbation_bound_2}
\end{equation}
Let \(\bO^*\in\cO_2\) attain the infimum.

To apply Lemma~\ref{lem:$K$-means}, represent \(\bU^*\bO^*\) in the form \(\bZmat\bX\).
Let \(\bZmat\in\{0,1\}^{N\times 2}\) be the membership matrix associated with \((C_1,C_2)\):
\[
\Zmat_{i1}=\1\{i\in C_1\},
\qquad
\Zmat_{i2}=\1\{i\in C_2\}.
\]
Since the top-\(2\) eigenspace of \(\bW^*\) equals \(\mathrm{span}\{\bu_1,\bu_2\}\), there exists \(\widetilde{\bO}\in\cO_2\) such that
\[
\bU^*\bO^*=[\bu_1\ \bu_2]\widetilde{\bO}.
\]
Using \([\bu_1\ \bu_2]=\bZmat\,\mathrm{diag}(N_1^{-1/2},N_2^{-1/2})\), we obtain
\begin{equation}
\bU^*\bO^*=\bZmat\,\bX,
\qquad
\bX \;\coloneqq\; \mathrm{diag}(N_1^{-1/2},N_2^{-1/2})\,\widetilde{\bO}.
\label{eq:X_def}
\end{equation}
Define the row separation
\[
\delta^2 \;\coloneqq\; \|\bX_{1\ast}-\bX_{2\ast}\|_2^2.
\]
Since \(\widetilde{\bO}\in\cO_2\), its rows are orthonormal, so
\[
\delta^2
=\|N_1^{-1/2}\widetilde{\bO}_{1\ast}-N_2^{-1/2}\widetilde{\bO}_{2\ast}\|_2^2
=\frac{1}{N_1}+\frac{1}{N_2},
\]
and therefore
\begin{equation}
\frac{1}{\delta^2}
=\frac{1}{N_1^{-1}+N_2^{-1}}
=\frac{N_1N_2}{N}.
\label{eq:delta_inverse_exact}
\end{equation}
On \(G\), \eqref{eq:balance_asymp_basic} gives
\begin{equation}
\frac{1}{\delta^2}=\frac{N}{4}\,(1+O(\eta_N)),
\qquad
N_{\min}\coloneqq \min\{N_1,N_2\}=\frac{N}{2}\,(1+O(\eta_N)),
\label{eq:delta_and_Nmin_balanced}
\end{equation}
with \(\eta_N\) given by \eqref{eq:eta_N_choice}.
Let \((\hat{\bZmat},\hat{\bX})\) be a \((1+\varepsilon)\)-approximate \(K\)-means solution applied to \(\bU\)
as in \eqref{eq:k-means-objective}, and let \(\hat{\bz}\) be the induced membership vector.
Applying Lemma~\ref{lem:$K$-means} with \(\bU\) and \(\bZmat\bX=\bU^*\bO^*\) shows that if
\begin{equation}
4(2+\varepsilon)\frac{\|\bU-\bU^*\bO^*\|_F^2}{\delta^2}\le N_{\min},
\label{eq:kmeans_suff_cond_here}
\end{equation}
then
\begin{equation}
\Ham^*(\hat{\bz},\bz)\le 4(2+\varepsilon)^2\frac{\|\bU-\bU^*\bO^*\|_F^2}{\delta^2}.
\label{eq:kmeans_error_here}
\end{equation}

Using \eqref{eq:perturbation_bound_2}, a sufficient condition for \eqref{eq:kmeans_suff_cond_here} is
\begin{equation}
32(2+\varepsilon)\,\Delta^{-2}\,\frac{\|\bE\|_F^2}{\delta^2} \;\le\; N_{\min}.
\label{eq:suff_cond_E}
\end{equation}
On \(G\), \eqref{eq:delta_and_Nmin_balanced} yields
\(1/\delta^2=\frac{N}{4}\,(1+O(\eta_N))\) and \(N_{\min}=\frac{N}{2}\,(1+O(\eta_N))\), so \eqref{eq:suff_cond_E} is implied
(for all sufficiently large \(N\)) by
\[
\|\bE\|_F^2 \;\le\; \frac{\Delta^2}{16(2+\varepsilon)}\,(1+O(\eta_N)).
\]
Under \eqref{eq:spectral_gap}, this becomes
\begin{equation}
\|\bE\|_F^2 \;\le\; \frac{N^2(N-2)^2}{256(2+\varepsilon)}(p-q)^2\,(1+O(\eta_N)).
\label{ineq:sufficient_condition}
\end{equation}

On \(G\), if \eqref{ineq:sufficient_condition} holds, then \eqref{eq:kmeans_error_here} applies. Using \eqref{eq:perturbation_bound_2} at the minimizer
\(\bO^*\) gives
\[
\|\bU-\bU^*\bO^*\|_F^2 \le 8\,\Delta^{-2}\,\|\bE\|_F^2,
\]
so
\[
\Ham^*(\hat{\bz},\bz)
\le
4(2+\varepsilon)^2\frac{\|\bU-\bU^*\bO^*\|_F^2}{\delta^2}
\le
32(2+\varepsilon)^2\,\Delta^{-2}\,\frac{\|\bE\|_F^2}{\delta^2}.
\]
Substituting \(1/\delta^2=\frac{N}{4}(1+O(\eta_N))\) from \eqref{eq:delta_and_Nmin_balanced} yields
\[
\Ham^*(\hat{\bz},\bz)
\le
8(2+\varepsilon)^2\,\frac{N}{\Delta^2}\,\|\bE\|_F^2\,(1+O(\eta_N)).
\]
Finally, \eqref{eq:spectral_gap} gives \(\Delta^2=\frac{N^2(N-2)^2}{16}(p-q)^2\,(1+O(\eta_N))\), and therefore
\begin{equation}
\Ham^*(\hat{\bz},\bz)
\;\le\;
\frac{128(2+\varepsilon)^2}{N(N-2)^2(p-q)^2}\,\|\bE\|_F^2\,(1+O(\eta_N))
=
\frac{128(2+\varepsilon)^2}{N^3(p-q)^2}\,\|\bE\|_F^2\,(1+O(\eta_N)).
\label{ineq:hamming_bound}
\end{equation}

\paragraph{Proof of Theorem \ref{thm:almost_main}:(i).}
\label{sec:proof_2a}
Let \(C\coloneqq128(2+\varepsilon)^{2}\). Recall that on \(G\) and on the event \eqref{ineq:sufficient_condition} we have the bound
\eqref{ineq:hamming_bound}. Set
\[
T \;\coloneqq\; \frac{1}{256(2+\varepsilon)}\,N^{2}(N-2)^{2}(p-q)^{2}.
\]
For all sufficiently large \(N\), the event \(\{\|\bE\|_F^{2}\le T\}\cap G\) implies \eqref{ineq:sufficient_condition}. Therefore,
\[
\Ham^{*}(\hat{\bz},\bz)
\wle
\frac{C}{(p-q)^2}\cdot\frac{\|\bE\|_{F}^{2}}{N^{3}}\,(1+O(\eta_N))\,\1_{\{\|\bE\|_F^{2}\le T\}}\1_G
\;+\; N\,\1_{G^{c}}
\;+\; N\,\1_{\{\|\bE\|_F^{2}> T\}}.
\]
Taking expectations and using \(\E[\|\bE\|_F^{2}\1_{\{\|\bE\|_F^{2}\le T\}}\1_G]\le \E[\|\bE\|_F^{2}]\) gives
\[
\E\left[\Ham^{*}(\hat{\bz},\bz)\right]
\wle
\frac{C}{(p-q)^2}\cdot\frac{\E\left[\|\bE\|_{F}^{2}\right]}{N^{3}}\,(1+O(\eta_N))
\;+\;
N\,\P(G^{c})
\;+\;
N\,\P(\|\bE\|_F^{2}>T),
\]
where \(\eta_N=o(1)\) by \eqref{eq:eta_N_choice}.
By Corollary~\ref{cor:similarity_concentration_conditions}, Lemma~\ref{lem:similarity_concentration}:(i) gives
\(\E\|\bE\|_{F}^{2}\ll N^{3}\). Moreover, Lemma~\ref{lem:_within_probability} and
Corollary~\ref{cor:coordinate_smallball_common} yield \(1-p\ll N^{-1}\) and \(q\ll N^{-1}\), hence
\(p-q=1-o(1)\). Moreover, \(N\P(G^{c})=o(1)\) by Lemma~\ref{lem:Ni_concentration}.
Finally, Markov's inequality yields
\[
\P(\|\bE\|_F^{2}>T)\le \frac{\E\|\bE\|_F^{2}}{T}
\ll
\frac{N^{3}}{N^{2}(N-2)^{2}(p-q)^{2}}
=o(N^{-1}),
\]
so \(N\,\P(\|\bE\|_F^{2}>T)=o(1)\). Hence \(\E[\Ham^{*}(\hat{\bz},\bz)]=o(1)\),  which proves part (i).

\paragraph{Proof of Theorem \ref{thm:almost_main}:(ii).}
\label{sec:proof_2b}
Let \(C\coloneqq128(2+\varepsilon)^{2}\), \(\gamma > 0\), and define
\[
E_{\gamma}\coloneqq\{\|\bE\|_{F}\le \gamma N^{2}|p-q|\}\cap G.
\]
By Lemma~\ref{lem:_within_probability} and Corollary~\ref{cor:coordinate_smallball_common}, we have
\(1-p=o(1)\) and \(q=o(1)\), hence \(p-q=1-o(1)\). In addition,
Corollary~\ref{cor:similarity_concentration_conditions} shows that
Lemma~\ref{lem:similarity_concentration}:(ii) applies. Fix \(\delta>0\) and choose
\(\gamma=(\delta/(3C))^{1/2}\), decreasing it if needed so that \(\gamma^{2}<1/(256(2+\varepsilon))\).
Then \(E_{\gamma}\) implies \eqref{ineq:sufficient_condition} for all sufficiently large \(N\), and hence
\eqref{ineq:hamming_bound} holds on \(E_{\gamma}\). On \(E_{\gamma}\) we have \(\|\bE\|_F^2\le \gamma^2 N^4(p-q)^2\), so
\[
\frac{1}{N}\Ham^{*}(\hat{\bz},\bz)
\wle
\frac{C}{N^4(p-q)^2}\,\|\bE\|_F^2\,(1+O(\eta_N))
\wle
C\gamma^{2}(1+O(\eta_N))
=
\frac{\delta}{3}\,(1+O(\eta_N))
\le \frac{\delta}{2}
\]
for all sufficiently large \(N\) since \(\eta_N = o(1)\), while \(\Ham^{*}(\hat{\bz},\bz)\le N\) on \(E_{\gamma}^{c}\). Therefore,
\[
\E\left[\frac{1}{N}\Ham^{*}(\hat{\bz},\bz)\right]
\wle
\frac{\delta}{2}
+\P(G^{c})
+\P(\|\bE\|_{F}>\gamma N^{2}|p-q|).
\]
Here \(\P(G^{c})=o(1)\) by Lemma~\ref{lem:Ni_concentration}, and by Markov's inequality applied to \(\|\bE\|_F^{2}\),
\[
\P(\|\bE\|_{F}>\gamma N^{2}|p-q|)
=
\P(\|\bE\|_{F}^{2}>\gamma^{2}N^{4}(p-q)^{2})
\le
\frac{\E\|\bE\|_{F}^{2}}{\gamma^{2}N^{4}(p-q)^{2}}
=o(1),
\]
since \(\E\|\bE\|_{F}^{2}\ll N^{4}\) by Lemma~\ref{lem:similarity_concentration}:(ii) and
\(p-q=1-o(1)\).
Taking \(\limsup\) in the above display yields
\[
\limsup_{N\to\infty}\E\left[\frac{1}{N}\Ham^{*}(\hat{\bz},\bz)\right]\le \frac{\delta}{2}.
\]
Since \(\delta>0\) is arbitrary, it follows that
\[
\lim_{N\to\infty}\E\left[\frac{1}{N}\Ham^{*}(\hat{\bz},\bz)\right]=0,
\]
which proves part (ii), and hence the theorem. \qed

\subsection{Data-driven threshold selection}
\label{sec:data-driven}

In this section, we prove the consistency of Algorithm~\ref{alg:line_clustering_spectral_nonparameteric}. Throughout this section, we assume that $\nu$ satisfies both Assumptions~\ref{ass:nu_lower_smallball} and \ref{ass:nu_upper_smallball}. We sample
\(M=M_N\) triples \(\cT=(\tau_1,\dots,\tau_M)\) independently and uniformly at random (with replacement) from
\(\binom{[N]}{3}\), and define the empirical distribution function
\begin{equation}
F_M(t)\weq \frac1M\sum_{i=1}^{M}\bm1\left\{\sigmaTLS(\tau_i)\le t\right\},
\qquad t>0.
\label{def:FM}
\end{equation}
Let
\[
G_N(t)\coloneqq \E[F_M(t)] = \tfrac14\,p(t)+\tfrac34\,q(t).
\]
If \(s_{(1)}\le s_{(2)}\le\cdots\le s_{(M)}\) are the order statistics of
\(s_i\coloneqq \sigma_\tls(\bX_{i_1},\bX_{i_2},\bX_{i_3})\), where \(\tau_i=\{i_1,i_2,i_3\}\), then
Lemma~\ref{lem:TLS_distinct} implies that the sampled TLS residuals are almost surely distinct. Hence
\(F_M\) is a step function with jumps at \(s_{(1)},\dots,s_{(M)}\), and the minimizer of
\[
|F_M(t)-\tfrac14|
\]
is attained at the step closest to \(M/4\), namely at
\[
t_N^*\weq s_{(\newround{M/4})}.
\]

We use only \(M\) sampled triples, with \(1\ll M\ll N^{1/2}\), in order to estimate the threshold.
This keeps the threshold-selection stage asymptotically negligible, while still allowing \(F_M\) to
approximate the residual distribution uniformly. The ingredients are collected in
Appendix~\ref{ap:concentration_sub}: Lemma~\ref{lem:find_concentration} gives a uniform concentration bound for
\(F_M\), Lemma~\ref{lem:expectation_gap}:(i) provides a vanishing benchmark sequence \(s_N\) for which \(F_M(s_N)\)
is close to \(1/4\), Lemma~\ref{lem:expectation_gap}:(ii) rules out thresholds bounded away from \(0\), and
Lemma~\ref{lem:selection_rule} shows that if \(G_N(t_N)\) is close to \(1/4\) for a vanishing deterministic
threshold sequence \(t_N\), then both \(q(t_N)\) and \(1-p(t_N)\) are small.

\subsection{Proof of Theorem \ref{thm:main_nonparametric}}

Algorithm~\ref{alg:line_clustering_spectral_nonparameteric} samples \(M\ll N^{1/2}\) triples
\(\tau_1,\dots,\tau_M\eqd\Unif\!\left(\binom{[N]}{3}\right)\). Let
\[
\cM \coloneqq \bigcup_{m=1}^M \tau_m
\]
be the set of sampled nodes. Then \(|\cM|\le 3M\ll N^{1/2}\). The set \(\cM\) is used for threshold
selection, while \(\cM^c=[N]\setminus \cM\) is used for recovery. Since the observations are independent across
vertices, these two stages are independent.

Fix a deterministic sequence \(\delta_M\) such that \(M^{-1/2}\ll \delta_M\ll 1\), and let
\(s_N,\eta_N=o(1)\) be the sequences from Lemma~\ref{lem:expectation_gap}:(i). Define
\[
E_0\coloneqq \left\{\sup_{t>0}|F_M(t)-G_N(t)|\le \delta_M\right\},
\qquad
E_1\coloneqq \left\{|F_M(s_N)-\tfrac14|\le \eta_N\right\}.
\]
By Lemmas~\ref{lem:find_concentration} and \ref{lem:expectation_gap}:(i),
\(
\P(E_0\cap E_1)=1-o(1).
\)
On \(E_1\), the benchmark \(s_N\) is feasible, so
\begin{equation}
\label{eq:tn*_concentration_14}
|F_M(t_N^*)-\tfrac14|
\le
|F_M(s_N)-\tfrac14|
\le \eta_N.
\end{equation}

We first show that \(t_N^*\to 0\) in probability. Fix \(\varepsilon>0\), and let \(c_\varepsilon>0\) be the
constant from Lemma~\ref{lem:expectation_gap}:(ii). Define
\[
E_{2,\varepsilon}\coloneqq \left\{\inf_{t\ge \varepsilon}|F_M(t)-\tfrac14|\ge c_\varepsilon\right\}.
\]
Then \(\P(E_{2,\varepsilon})=1-o(1)\). Since \(\eta_N=o(1)\), we have \(\eta_N<c_\varepsilon\) for all sufficiently
large \(N\). Thus on \(E_1\cap E_{2,\varepsilon}\) the minimizer \(t_N^*\) cannot satisfy \(t_N^*\ge \varepsilon\), so
\[
\P(t_N^*\ge \varepsilon)\le \P(E_1^c)+\P(E_{2,\varepsilon}^c)=o(1).
\]
Hence \(t_N^*\to 0\) in probability. Therefore there exists a deterministic sequence \(r_N=o(1)\) such that
\[
\P(t_N^*\le r_N)=1-o(1).
\]

Choose a deterministic sequence \(\varepsilon_N=o(1)\) such that \(\eta_N+\delta_M\lesssim \varepsilon_N\) and
\begin{equation}
\label{eq:sigma_ll}
 \sigma_N \ll \ell\left(\frac{\varepsilon_N}{\log(1/\varepsilon_N)}\right)^{1/\rho}
 \log^{-1/2}(1/\varepsilon_N).
\end{equation}
Such a choice is possible because \(\eta_N+\delta_M=o(1)\), \(\sigma_N\ll 1\), and the right-hand side of
\eqref{eq:sigma_ll} tends to zero arbitrarily slowly as \(\varepsilon_N\downarrow 0\).

Let
\[
E_N\coloneqq E_0\cap E_1\cap\{t_N^*\le r_N\}.
\]
Then \(\P(E_N)=1-o(1)\). On \(E_N\), \eqref{eq:tn*_concentration_14} and the definition of \(E_0\) give
\[
|G_N(t_N^*)-\tfrac14|
\le |G_N(t_N^*)-F_M(t_N^*)|+|F_M(t_N^*)-\tfrac14|
\le \delta_M+\eta_N
\lesssim \varepsilon_N.
\]

Fix \(C>0\) such that \(\delta_M+\eta_N\le C\varepsilon_N\) for all sufficiently large \(N\), and define
\[
S_N\coloneqq \Bigl\{\,t\in(0,r_N]:\ |G_N(t)-\tfrac14|\le C\varepsilon_N\,\Bigr\},
\qquad
\gamma_N\coloneqq \sup_{t\in S_N}\bigl(q(t)\vee(1-p(t))\bigr),
\]
with the convention \(\sup\emptyset=0\). We claim that \(\gamma_N=o(1)\). Suppose this does not hold. Then, after passing to a
subsequence, there exists \(c>0\) such that \(\gamma_N\ge c\) for all \(N\). By definition of the supremum, we may choose a deterministic
sequence \(u_N\in S_N\) such that
\[
q(u_N)\vee(1-p(u_N))\ge \frac{c}{2}.
\]
Since \(u_N\le r_N\), we have \(u_N=o(1)\), and because \(u_N\in S_N\),
\[
|G_N(u_N)-\tfrac14|\le C\varepsilon_N.
\]
Lemma~\ref{lem:selection_rule} therefore applies to the deterministic sequence \(u_N\) and yields
\[
q(u_N)\lesssim \varepsilon_N,
\qquad
1-p(u_N)\lesssim \varepsilon_N,
\]
which contradicts \(q(u_N)\vee(1-p(u_N))\ge c/2\). Hence \(\gamma_N=o(1)\).

On \(E_N\), we have \(t_N^*\le r_N\) and \(|G_N(t_N^*)-\tfrac14|\le C\varepsilon_N\), so \(t_N^*\in S_N\). Therefore,
\[
q(t_N^*)\le \gamma_N,
\qquad
1-p(t_N^*)\le \gamma_N
\]
on \(E_N\). In particular, with probability \(1-o(1)\), the random threshold \(t_N^*\) satisfies
\[
q(t_N^*)=o(1),
\qquad
1-p(t_N^*)=o(1).
\]

Let \(\mathcal F_N\) denote the sigma algebra generated by the threshold-selection stage. Then \(t_N^*\) and \(E_N\) are \(\mathcal F_N\)-measurable, while the data on \(\cM^c\) are independent of \(\mathcal F_N\). Moreover, conditional on \(\mathcal F_N\), the restricted sample \((\bZ_{\cM^c},\bX_{\cM^c})\) is again a Gaussian line mixture model on \(|\cM^c|=N-o(N)\) vertices. On \(E_N\), we have
\[
(1-p(t_N^*))+q(t_N^*)\le 2\gamma_N=o(1),
\]
and hence \(p(t_N^*)-q(t_N^*)=1-o(1)\). Therefore the same estimates as in the proof of
Theorem~\ref{thm:almost_main}:(ii), with \(N\) replaced by \(|\cM^c|\) and applying Lemma~\ref{lem:similarity_concentration}:(ii), yield
\[
\E\!\left[\Ham^*(\hat{\bz}_{\cM^c},\bz_{\cM^c})\given \mathcal F_N\right]\1_{E_N}
=o(|\cM^c|)\,\1_{E_N}.
\]
Taking expectations gives
\[
\E\!\left[\Ham^*(\hat{\bz}_{\cM^c},\bz_{\cM^c})\,\1_{E_N}\right]=o(N).
\]

Finally, since \(\Ham^*(\hat{\bz},\bz)\le |\cM|+\Ham^*(\hat{\bz}_{\cM^c},\bz_{\cM^c})\), we have
\[
\frac1N\,\E\!\left[\Ham^*(\hat{\bz},\bz)\right]
\le
\frac{|\cM|}{N}
+\frac1N\,\E\!\left[\Ham^*(\hat{\bz}_{\cM^c},\bz_{\cM^c})\,\1_{E_N}\right]
+\P(E_N^c),
\]
and the right-hand side is \(o(1)\). This proves almost exact recovery.
\qed
\appendix

\section{Decision theory}
\label{ap:decision_theory}

Throughout this appendix, we work under the assumptions and notation of
Section~\ref{sec:lower_bound_for_bayes_risk}. In particular, $(Z,X)$, the
densities \(f_1,f_2\) with respect to \(\mu\), the posterior probabilities
\[
 \P(Z=k\given X=x)=\frac{f_k(x)}{f_1(x)+f_2(x)},\qquad k=1,2,
\]
the single-coordinate MAP rule \(\tilde t\), the coordinatewise MAP estimator
\(\tbt\) from \eqref{def:coordinatewise_map}, the joint law \(\P\) of
\((\bZ,\bX)\), the marginal law \(\P_2\) of \(\bX\), and the full-measure set
\(\cR\) from Definition~\ref{def:good_set_bayes} are all as defined there.

\subsection{Bayes optimality under permutation-invariant Hamming loss}

For \(\bw,\bz\in\cZ^N\), write
\[
 \Ham^*(\bw,\bz)
 \;\coloneqq\;
 h\!\left(\Ham(\bw,\bz)\right),
 \qquad
 h(s)\coloneqq \min\{s,N-s\},\qquad s=0,\dots,N.
\]
For \(\bw\in\cZ^N\) and \(\bx\in\cX^N\), define the conditional and average risks
\[
 r^*(\bw\given\bx)
 \;\coloneqq\;
 \E\!\left[\Ham^*(\bw,\bZ)\given\bX=\bx\right],
 \qquad
 R^*(\bt)
 \;\coloneqq\;
 \E\!\left[\Ham^*(\bt(\bX),\bZ)\right].
\]

\begin{lemma}
\label{lem:pointwise-bayes_r}
A measurable rule \(\hat\bt:\cX^N\to\cZ^N\) minimizes \(R^*(\bt)\) if and only if
\[
 r^*(\hat\bt(\bx)\given\bx)
 \wle
 r^*(\bw\given\bx)
 \qquad\text{for all }\bw\in\cZ^N,
\]
for \(\P_2\)-almost every \(\bx\in\cX^N\).
\end{lemma}

\begin{proof}
Since \(0\le \Ham^*(\bw,\bz)\le N/2\), all risks are finite. The claim is the
standard Bayes optimality criterion applied with loss \(\Ham^*\); see
\cite[Theorem~1.1, Section~4]{lehmann1998point}.
\end{proof}

\subsection{Coordinatewise improvement}

For \(i\in[N]\) and \(\bw=(w_1,\dots,w_N)\in\cZ^N\), define the coordinate flip
\[
 w_\ell^{(i)}
 \;\coloneqq\;
 \begin{cases}
 3-w_i, & \ell=i,\\
 w_\ell, & \ell\neq i.
 \end{cases}
\]
For fixed \(\bx\in\cX^N\) and \(\bw\in\cZ^N\), define
\[
 \delta_i(\bw\given \bx)
 \;\coloneqq\;
 r^*(\bw\given\bx)-r^*(\bw^{(i)}\given\bx),
\]
and
\[
 \delta_{i^*}(\bw\given \bx)
 \;\coloneqq\;
 \max_{i\in\{0,1,\dots,N\}}
 \left\{r^*(\bw\given\bx)-r^*(\bw^{(i)}\given\bx)\right\},
\]
with the convention \(\bw^{(0)}=\bw\).

\begin{lemma}\label{lem:local-exchange-final}
Fix \(\bx\in\cR\) and a decision vector \(\bw\in\cZ^N\) which is not equal to
\(\tbt(\bx)\) or \(3-\tbt(\bx)\). Then
\[
 \delta_{i^*}(\bw \given \bx)
 \;>\;0.
\]
\end{lemma}

\begin{proof}
Fix \(\bx=(x_1,\dots,x_N)\in\cR\) and \(\bw=(w_1,\dots,w_N)\in\cZ^N\) as in the
statement, and write \(\tilde t_\ell=\tilde t(x_\ell)\). Applying
\eqref{eq:posterior_def_generic} with \eqref{eq:bayes_rule_generic}, the
conditional error corresponding to \(\tbt\) becomes
\begin{equation}
\label{eq:Bayes_error_vector}
\varepsilon_\ell
\;\coloneqq\;
\P(Z_\ell\neq\tilde t_\ell\given\bX=\bx)
=
\frac{\min\{f_1(x_\ell),f_2(x_\ell)\}}{f_1(x_\ell)+f_2(x_\ell)},
\qquad \ell\in[N].
\end{equation}

Since \(\bw\) is not equal to \(\tbt(\bx)\) or \(3-\tbt(\bx)\), there exist
indices with \(w_i\neq\tilde t_i\) and indices with \(w_j=\tilde t_j\). We fix
\(i\neq j\in[N]\) such that
\begin{equation}
w_i\neq\tilde t_i,
\qquad
w_j=\tilde t_j.
\label{eq:fixed_ij_bayes}
\end{equation}

For a decision vector \(\bv=(v_1,\dots,v_N)\in\cZ^N\), define the
misclassification indicators
\[
 M_\ell(\bv)\coloneqq \1\{Z_\ell\neq v_\ell\},
 \qquad \ell\in[N],
\]
and the partial sum
\[
 S(\bv)\coloneqq \sum_{\ell\notin\{i,j\}} M_\ell(\bv).
\]
Then
\[
 \Ham(\bv,\bZ)=S(\bv)+M_i(\bv)+M_j(\bv),
 \qquad
 \Ham^*(\bv,\bZ)=h\!\left(S(\bv)+M_i(\bv)+M_j(\bv)\right),
\]
and
\[
 r^*(\bv\given \bx)
 = \E_{\bx}\!\left[h\!\left(S(\bv)+M_i(\bv)+M_j(\bv)\right)\right],
\]
where \(\P_{\bx}(\cdot)=\P(\cdot\given \bX=\bx)\) and \(\E_{\bx}\) is the
corresponding conditional expectation.

By independence of \((Z_i,X_i)\) across \(i\) and \eqref{eq:Bayes_error_vector},
for each \(i\) we have
\[
 \P_{\bx}(Z_i\neq v_i)
 =
 \begin{cases}
 \varepsilon_i, & v_i=\tilde t_i,\\[1pt]
 1-\varepsilon_i, & v_i\neq\tilde t_i,
 \end{cases}
\]
and the random variables \(M_1(\bv),\dots,M_N(\bv)\) are independent under
\(\P_{\bx}\).
In particular, for the \(i,j\) fixed in \eqref{eq:fixed_ij_bayes},
\[
 \P_{\bx}(M_i(\bw)=1)=1-\varepsilon_i,
 \qquad
 \P_{\bx}(M_j(\bw)=1)=\varepsilon_j,
\]
while for the flipped vectors
\[
 \P_{\bx}(M_i(\bw^{(i)})=1)=\varepsilon_i,\quad
 \P_{\bx}(M_j(\bw^{(i)})=1)=\varepsilon_j,
\]
\[
 \P_{\bx}(M_i(\bw^{(j)})=1)=1-\varepsilon_i,\quad
 \P_{\bx}(M_j(\bw^{(j)})=1)=1-\varepsilon_j.
\]
The coordinates \(\ell\notin\{i,j\}\) are unchanged, so
\(S(\bw)\stackrel{d}{=}S(\bw^{(i)})\stackrel{d}{=}S(\bw^{(j)})\) under
\(\P_{\bx}\); we denote these random variables simply by \(S\).

Define the forward difference
\[
 \Delta h(s)\coloneqq h(s+1)-h(s),
 \qquad s=0,1,\dots,N-1.
\]
Conditioning on \((S,M_j(\bw))\) and averaging over \(M_i(\bw)\) gives
\begin{align*}
 r^*(\bw\given \bx)
 &=\E_{\bx}\!\left[h(S+M_i(\bw)+M_j(\bw))\right]\\
 &=\E_{\bx}\!\left[h(S+M_j(\bw))\right]
 +\P_{\bx}\!\left(M_i(\bw)=1\right)\,
 \E_{\bx}\!\left[\Delta h(S+M_j(\bw))\right]\\
 &=\E_{\bx}\!\left[h(S+M_j(\bw))\right]
 +(1-\varepsilon_i)\,\E_{\bx}\!\left[\Delta h(S+M_j(\bw))\right],
\end{align*}
and similarly
\[
 r^*(\bw^{(i)}\given \bx)
 =\E_{\bx}\!\left[h(S+M_j(\bw))\right]
 +\varepsilon_i\,\E_{\bx}\!\left[\Delta h(S+M_j(\bw))\right].
\]
Thus
\begin{equation}\label{eq:exchange-i-short}
 r^*(\bw\given \bx)-r^*(\bw^{(i)}\given \bx)
 =(1-2\varepsilon_i)\,\E_{\bx}\!\left[\Delta h(S+M_j(\bw))\right].
\end{equation}
An analogous calculation, swapping the roles of \(i\) and \(j\), yields
\begin{equation}\label{eq:exchange-j-short}
 r^*(\bw\given \bx)-r^*(\bw^{(j)}\given \bx)
 =-(1-2\varepsilon_j)\,\E_{\bx}\!\left[\Delta h(S+M_i(\bw))\right].
\end{equation}

Next, we expand the expectations in
\eqref{eq:exchange-i-short}--\eqref{eq:exchange-j-short}.
Using independence,
\begin{align}
 \E_{\bx}\!\left[\Delta h(S+M_i(\bw))\right]
 &=\E_{\bx}[\Delta h(S)]
 +(1-\varepsilon_i)\,\E_{\bx}\!\left[\Delta h(S+1)-\Delta h(S)\right],
 \label{eq:affine-i-short}\\[2pt]
 \E_{\bx}\!\left[\Delta h(S+M_j(\bw))\right]
 &=\E_{\bx}[\Delta h(S)]
 +\varepsilon_j\,\E_{\bx}\!\left[\Delta h(S+1)-\Delta h(S)\right].
 \nonumber
\end{align}
Subtracting gives
\begin{equation}\label{eq:affine-diff-short}
 \E_{\bx}\!\left[\Delta h(S+M_j(\bw))\right]
 -\E_{\bx}\!\left[\Delta h(S+M_i(\bw))\right]
 =-(1-\varepsilon_i-\varepsilon_j)\,
 \E_{\bx}\!\left[\Delta h(S+1)-\Delta h(S)\right].
\end{equation}

Since \(h(s)=\min\{s,N-s\}\), its forward difference is
\[
 \Delta h(s)=
 \begin{cases}
 1,& s\le \lfloor N/2\rfloor-1,\\[1pt]
 0,& \text{$N$ odd and } s=\lfloor N/2\rfloor,\\[1pt]
 -1,& s\ge \lceil N/2\rceil,
 \end{cases}
\]
and is nonincreasing in \(s\). Hence
\begin{align}
 \E_{\bx}\!\left[\Delta h(S)\right]
 &= \sum_{s=0}^{N-2} \Delta h(s)\,\P_{\bx}(S=s) \nonumber\\
 &= \sum_{s\le \lfloor N/2\rfloor-1} \P_{\bx}(S=s)
 \;-\;\sum_{s\ge \lceil N/2\rceil} \P_{\bx}(S=s),
 \label{eq:Ex-Delta-hS-short}
\end{align}
and similarly
\begin{align}
 \E_{\bx}\!\left[\Delta h(S+1)\right]
 &= \sum_{s=0}^{N-2} \Delta h(s+1)\,\P_{\bx}(S=s) \nonumber\\
 &= \sum_{s\le \lfloor N/2\rfloor-2} \P_{\bx}(S=s)
 \;-\;\sum_{s\ge \lceil N/2\rceil-1} \P_{\bx}(S=s).
 \label{eq:Ex-Delta-hSplus1-short}
\end{align}
Subtracting \eqref{eq:Ex-Delta-hS-short} from
\eqref{eq:Ex-Delta-hSplus1-short} yields
\begin{equation}\label{eq:jump-mass-short}
 \E_{\bx}\!\left[\Delta h(S+1)-\Delta h(S)\right]
 = -\P_{\bx}\!\left(S=\lfloor N/2\rfloor-1\right)
 -\P_{\bx}\!\left(S=\lceil N/2\rceil-1\right),
\end{equation}
and combining \eqref{eq:affine-diff-short} and \eqref{eq:jump-mass-short} gives
\begin{align}
 &\E_{\bx}\!\left[\Delta h(S+M_j(\bw))\right]
 -\E_{\bx}\!\left[\Delta h(S+M_i(\bw))\right] \nonumber\\[2pt]
 &\qquad=
 (1-\varepsilon_i-\varepsilon_j)\,
 \left(
 \P_{\bx}\!\left(S=\lfloor N/2\rfloor-1\right)
 +\P_{\bx}\!\left(S=\lceil N/2\rceil-1\right)
 \right).
 \label{eq:affine-diff-final-short}
\end{align}

Recalling
\[
 \delta_i(\bw \given \bx)
 \coloneqq
 r^*(\bw\given \bx)-r^*(\bw^{(i)}\given \bx),\qquad
 \delta_j(\bw \given \bx)
 \coloneqq
 r^*(\bw\given \bx)-r^*(\bw^{(j)}\given \bx),
\]
let
\[
 \delta_{i,j}(\bw \given \bx)
 \coloneqq
 (1-2\varepsilon_j)\,\delta_i(\bw \given \bx)
 + (1-2\varepsilon_i)\,\delta_j(\bw \given \bx).
\]
Using \eqref{eq:exchange-i-short}, \eqref{eq:exchange-j-short}, and
\eqref{eq:affine-diff-final-short}, we obtain
\begin{align}
 \delta_{i,j}(\bw \given \bx)
 &=
 (1-2\varepsilon_i)(1-2\varepsilon_j)
 \left(
 \E_{\bx}[\Delta h(S+M_j(\bw))]
 -\E_{\bx}[\Delta h(S+M_i(\bw))]
 \right) \nonumber\\[2pt]
 &=
 (1-2\varepsilon_i)(1-2\varepsilon_j)(1-\varepsilon_i-\varepsilon_j)
 \left(
 \P_{\bx}\!\left(S=\lfloor N/2\rfloor-1\right)
 +\P_{\bx}\!\left(S=\lceil N/2\rceil-1\right)
 \right).
 \label{eq:DeltaR-final-short}
\end{align}

We now show that the right-hand side of \eqref{eq:DeltaR-final-short} is strictly
positive. On \(\cR\) we have \(\varepsilon_i,\varepsilon_j\in(0,1/2)\), so
\[
 1-2\varepsilon_i>0,\quad
 1-2\varepsilon_j>0,\quad
 1-\varepsilon_i-\varepsilon_j>0.
\]
Moreover, for each \(\ell\notin\{i,j\}\),
\[
 \P_{\bx}\!\left(M_\ell(\bw)=0\right)>0
 \quad\text{and}\quad
 \P_{\bx}\!\left(M_\ell(\bw)=1\right)>0,
\]
because \(\varepsilon_\ell\in(0,1/2)\) and \(w_\ell\) either coincides with
\(\tilde t_\ell\) or not. Therefore \(S(\bw)\) is a sum of independent Bernoulli
random variables, each taking the values \(0\) and \(1\) with positive probability.
For any \(m\in\{0,\dots,N-2\}\) one can choose a configuration
\((m_\ell)_{\ell\notin\{i,j\}}\in\{0,1\}^{N-2}\) with
\(\sum_{\ell\notin\{i,j\}} m_\ell=m\), so
\[
 \P_{\bx}\!\left(S(\bw)=m\right)
 \wge
 \prod_{\ell\notin\{i,j\}} \P_{\bx}\!\left(M_\ell(\bw)=m_\ell\right)
 \;>\;0.
\]
In particular,
\[
 \P_{\bx}\!\left(S(\bw)=\lfloor N/2\rfloor-1\right)
 +\P_{\bx}\!\left(S(\bw)=\lceil N/2\rceil-1\right)
 \;>\;0.
\]
Thus the product in \eqref{eq:DeltaR-final-short} is strictly positive, and
\[
 \delta_{i,j}(\bw \given \bx) > 0.
\]

Finally, since
\[
 \delta_{i,j}(\bw \given \bx)
 = (1-2\varepsilon_i)\,\delta_i(\bw \given \bx)
 +(1-2\varepsilon_j)\,\delta_j(\bw \given \bx) > 0
\]
with \(1-2\varepsilon_i>0\) and \(1-2\varepsilon_j>0\), at least one of
\(\delta_i(\bw \given \bx)\) or \(\delta_j(\bw \given \bx)\) must be strictly
positive. By definition of \(\delta_{i^*}(\bw \given \bx)\),
\[
 \delta_{i^*}(\bw \given \bx)
 \ge \max\{\delta_i(\bw \given \bx),\delta_j(\bw \given \bx)\}>0,
\]
proving the claim.
\end{proof}

\subsection{Proof of Proposition~\ref{prop:pi-hamming-cwmap}}
\label{ap:proof_of_prop:pi-hamming-cwmap}

For every \(\bw,\bz\in\cZ^N\),
\[
 \Ham^*(3-\bw,\bz)=\Ham^*(\bw,\bz),
\]
so for every \(\bx\),
\[
 r^*(\tbt(\bx)\given\bx)=r^*(3-\tbt(\bx)\given\bx).
\]
Now fix \(\bx\in\cR\). If \(\bw\notin\{\tbt(\bx),\,3-\tbt(\bx)\}\), then
Lemma~\ref{lem:local-exchange-final} gives
\[
 \delta_{i^*}(\bw\given\bx)>0.
\]
Hence there exists \(i\in[N]\) such that
\[
 r^*(\bw^{(i)}\given\bx)<r^*(\bw\given\bx),
\]
and therefore \(\bw\) cannot minimize \(r^*(\cdot\given\bx)\). It follows that every
pointwise minimizer of the conditional risk belongs to
\(\{\tbt(\bx),\,3-\tbt(\bx)\}\). Since the two elements of this set have equal conditional
risk, both are pointwise minimizers. Since \(\P_2(\cR)=1\),
Lemma~\ref{lem:pointwise-bayes_r} implies that \(\tbt\) is Bayes optimal under the loss
\(\Ham^*\). This proves the proposition.
\qed

\subsection{Proof of Lemma~\ref{lem:impossibility}}
\label{ap:gaussian}

By \eqref{eq:posterior_def_generic}, for any measurable
\(t\colon\R^2\to\{1,2\}\),
\[
 \P\left(Z\neq t(\bx)\given \bX=\bx\right)
 \weq
 1-\frac{f_{t(\bx)}(\bx)}{f_1(\bx)+f_2(\bx)}.
\]
Specializing to the single-coordinate MAP estimator \(\tilde t\) from
\eqref{eq:bayes_rule_generic}, we obtain
\[
 \P\left(Z\neq\tilde t(\bx)\given \bX=\bx\right)
 \weq
 \frac{\min\{f_1(\bx),f_2(\bx)\}}{f_1(\bx)+f_2(\bx)},
\]
and hence, with \(A_*\coloneqq\{x:f_1(\bx)>f_2(\bx)\}\),
\begin{equation}\label{eq:perr_again}
 \perr
 \;\coloneqq\;
 \P\left(Z_i\neq\tilde t(\bX_i)\right)
 \weq
 \frac12\int_{\R^2}\min\{f_1,f_2\}\,d\lambda
 \weq
 \tfrac12 F_2(A_*)+\tfrac12 F_1(A_*^c),
\end{equation}
where the unconditional law of \(\bX_i\) admits the Lebesgue density
\((f_1+f_2)/2\), and \(\lambda\) is Lebesgue measure on \(\R^2\).

Assume now that the latent coordinate on each line segment has a common law \(\nu\) on
\([-\ell/2,\ell/2]\), that \(\nu\) is symmetric about \(0\), and that
\[
 \nu([-u,u]) \wge B \left(\frac{u}{\ell}\right)^\rho,
 \qquad 0\le u\le u_0,
\]
for some \(\rho\in(0,1]\), \(B>0\), and \(u_0>0\).
Let \(U\eqd \nu\). If \(\nu=\delta_0\), then \(f_1=f_2\), so \(\perr=1/2\), and the desired lower
bound is immediate. Thus we may assume \(\nu\neq \delta_0\).

Without loss of generality, we may choose coordinates so that the line segments are aligned as in
Figure~\ref{fig:quadrants}:
\[
 L_1=\{(u\cos\beta,-u\sin\beta):u\in[-\ell/2,\ell/2]\},
 \quad
 L_2=\{(u\cos\beta,+u\sin\beta):u\in[-\ell/2,\ell/2]\},
\]
where \(\beta=\alpha/2\in(0,\pi/2)\). Let
\[
 \varphi_\sigma(x)\coloneqq (2\pi\sigma^2)^{-1/2}e^{-x^2/(2\sigma^2)}.
\]
\begin{figure}[ht]
\centering
\begin{tikzpicture}[scale=1.5]
 % Define the angle beta (in degrees)
 \def\beta{30}
 
 % Fill Quadrant I (x > 0, y > 0) with a red pattern of diagonal lines
 \fill[pattern=north west lines, pattern color=myred, fill opacity=0.3] (0,0) rectangle (1,1);
 
 % Fill Quadrant III (x < 0, y < 0) with the same red pattern
 \fill[pattern=north west lines, pattern color=myred, fill opacity=0.3] (-1,-1) rectangle (0,0);
 
 % Draw coordinate axes
 \draw[->] (-1,0) -- (1,0) node[right] {$x_1$};
 \draw[->] (0,-1) -- (0,1) node[above] {$x_2$};
 
 % Draw line segment L_1: (u*cos(beta), -u*sin(beta)) for u in [-1,1] in blue
 \draw[thick, color=myblue] ({-cos(\beta)}, {sin(\beta)}) -- ({cos(\beta)}, {-sin(\beta)})
 node[pos=0.8, below left] {$L_1$};

 % Draw line segment L_2: (u*cos(beta), u*sin(beta)) for u in [-1,1] in red
 \draw[thick, color=myred] ({-cos(\beta)}, {-sin(\beta)}) -- ({cos(\beta)}, {sin(\beta)})
 node[pos=0.8, above left] {$L_2$};

% --- insert after drawing L1 and L2 -----------------------------------------
% draw the two rays for the angle
\def\radius{0.3} 
%\draw[gray, thick,-] (0,0) -- (\beta:\radius);
%\draw[gray, thick,-] (0,0) -- (-\beta:\radius);
% draw the arc between them
\draw[gray, thick] (\beta:\radius) arc[start angle=\beta, end angle=-\beta, radius=\radius];
% label the angle
\node at ({\radius*1.4},0.0) {$\alpha$};
\end{tikzpicture}
 \caption{The decision region $A_*^c$ where the density $f_2$ dominates $f_1$ is colored in red.
 %\LLnote{Page 9. It would be good to indicate the angle $\alpha$ or $\beta$ in Fig 2.}
 }
\label{fig:quadrants}
\end{figure}

Then the conditional laws \(F_1\) and \(F_2\) of \(\bX_i\) given \(Z_i=1\) and \(Z_i=2\) have
Lebesgue densities
\[
 f_1(x_1,x_2)
 =
 \int_{-\ell/2}^{\ell/2}
 \varphi_\sigma(x_1-u\cos\beta)\,
 \varphi_\sigma(x_2+u\sin\beta)\,\nu(du),
\]
\[
 f_2(x_1,x_2)
 =
 \int_{-\ell/2}^{\ell/2}
 \varphi_\sigma(x_1-u\cos\beta)\,
 \varphi_\sigma(x_2-u\sin\beta)\,\nu(du).
\]
Expanding the squares gives
\[
 f_1(x_1,x_2)
 =
 c_\sigma(x_1,x_2)
 \int_{-\ell/2}^{\ell/2}
 \exp\!\left(
 -\frac{u^2}{2\sigma^2}
 + \frac{u(x_1\cos\beta-x_2\sin\beta)}{\sigma^2}
 \right)\nu(du),
\]
\[
 f_2(x_1,x_2)
 =
 c_\sigma(x_1,x_2)
 \int_{-\ell/2}^{\ell/2}
 \exp\!\left(
 -\frac{u^2}{2\sigma^2}
 + \frac{u(x_1\cos\beta+x_2\sin\beta)}{\sigma^2}
 \right)\nu(du),
\]
where
\[
 c_\sigma(x_1,x_2)\coloneqq
 \frac{1}{2\pi\sigma^2}\exp\!\left(-\frac{x_1^2+x_2^2}{2\sigma^2}\right).
\]
Since \(\nu\) is symmetric, we may rewrite these integrals in terms of the function
\[
 G_\sigma(t)
 \weq \int_{-\ell/2}^{\ell/2}
 \exp\left(-\frac{u^2}{2\sigma^2}\right)
 \cosh\left(\frac{ut}{\sigma^2}\right)\nu(du),
 \qquad t\in\R.
\]
Indeed, for every \(t\in\R\),
\[
\int_{-\ell/2}^{\ell/2}
\exp\left(-\frac{u^2}{2\sigma^2}\right)
\exp\left(\frac{ut}{\sigma^2}\right)\nu(du)
=
\int_{-\ell/2}^{\ell/2}
\exp\left(-\frac{u^2}{2\sigma^2}\right)
\exp\left(-\frac{ut}{\sigma^2}\right)\nu(du),
\]
because \(\nu(B)=\nu(-B)\) for every Borel set \(B\). Averaging the two sides and using
\(\cosh v=\tfrac12(e^v+e^{-v})\), we obtain
\[
\int_{-\ell/2}^{\ell/2}
\exp\left(-\frac{u^2}{2\sigma^2}\right)
\exp\left(\frac{ut}{\sigma^2}\right)\nu(du)
=
\int_{-\ell/2}^{\ell/2}
\exp\left(-\frac{u^2}{2\sigma^2}\right)
\cosh\left(\frac{ut}{\sigma^2}\right)\nu(du)
=
G_\sigma(t).
\]
It follows that
\begin{align*}
 f_1(x_1,x_2)
 &= c_\sigma(x_1,x_2)\,G_\sigma(x_1\cos\beta-x_2\sin\beta),\\
 f_2(x_1,x_2)
 &= c_\sigma(x_1,x_2)\,G_\sigma(x_1\cos\beta+x_2\sin\beta).
\end{align*}

We next show that \(G_\sigma\) is even and strictly increasing on \([0,\infty)\).
Evenness follows from the evenness of \(\cosh\): for every \(t\in\R\),
\[
\begin{aligned}
G_\sigma(-t)
&\weq \int_{-\ell/2}^{\ell/2}
\exp\left(-\frac{u^2}{2\sigma^2}\right)
\cosh\left(-\frac{ut}{\sigma^2}\right)\nu(du) \\
&\weq \int_{-\ell/2}^{\ell/2}
\exp\left(-\frac{u^2}{2\sigma^2}\right)
\cosh\left(\frac{ut}{\sigma^2}\right)\nu(du) \\
&\weq G_\sigma(t).
\end{aligned}
\]
To prove strict monotonicity, fix \(0\le s<t\). Then
\[
\begin{aligned}
G_\sigma(t)-G_\sigma(s)
&\weq \int_{-\ell/2}^{\ell/2}
\exp\left(-\frac{u^2}{2\sigma^2}\right)
\left[
\cosh\left(\frac{ut}{\sigma^2}\right)
-
\cosh\left(\frac{us}{\sigma^2}\right)
\right]\nu(du) \\
&\weq \int_{-\ell/2}^{\ell/2}
\exp\left(-\frac{u^2}{2\sigma^2}\right)
\left[
\cosh\left(\frac{|u|t}{\sigma^2}\right)
-
\cosh\left(\frac{|u|s}{\sigma^2}\right)
\right]\nu(du),
\end{aligned}
\]
where we used again that \(\cosh\) is even. Since \(\cosh\) is strictly increasing on
\([0,\infty)\), the integrand is nonnegative for every \(u\in\R\), and it is strictly positive
whenever \(u\neq 0\). Since \(\nu\neq\delta_0\), we have \(\nu(\{u\neq 0\})>0\). Therefore
\[
G_\sigma(t)-G_\sigma(s)>0,
\]
so \(G_\sigma\) is strictly increasing on \([0,\infty)\).

Therefore $f_1(x_1,x_2)\le f_2(x_1,x_2)$ is equivalent to
\[
|x_1\cos\beta-x_2\sin\beta|
\le
|x_1\cos\beta+x_2\sin\beta|.
\]
Squaring both sides gives
\[
 |x_1\cos\beta+x_2\sin\beta|^2
 -
 |x_1\cos\beta-x_2\sin\beta|^2
 =
 4x_1x_2\sin\beta\cos\beta.
\]
Since \(\sin\beta\cos\beta>0\), it follows that
\[
 A_*^c
 =
 \{x\in\R^2:f_1(x)\le f_2(x)\}
 =
 \{(x_1,x_2)\in\R^2:x_1x_2\ge 0\},
\]
that is, \(A_*^c\) is the union of quadrants I and III up to the boundary $\{x_1x_2=0\}$, exactly as in
Figure~\ref{fig:quadrants}.

Now let \(T(x_1,x_2)\coloneqq (x_1,-x_2)\). Under \(F_1\),
\[
 \bX_i \weq (U\cos\beta+\sigma Y_1,\,-U\sin\beta+\sigma Y_2),
\]
so
\[
 T(\bX_i)\weq (U\cos\beta+\sigma Y_1,\ U\sin\beta-\sigma Y_2).
\]
Since \(-Y_2\stackrel{d}{=}Y_2\), the law of \(T(\bX_i)\) is exactly \(F_2\). Moreover,
\(T\) maps \(A_*^c=\{x_1x_2\ge 0\}\) to \(A_*=\{x_1x_2<0\}\) up to the boundary
\(\{x_1x_2=0\}\), which has zero \(F_1\)- and \(F_2\)-measure. Hence
\(F_2(A_*)=F_1(A_*^c)\), and \eqref{eq:perr_again} becomes
\[
 \perr = F_1(A_*^c).
\]

Under \(F_1\), the random vector \(\bX_i\) is centrally symmetric, because
\[
 \left(U\cos\beta+\sigma Y_1,\,-U\sin\beta+\sigma Y_2\right)
 \stackrel{d}{=}
 -\left(U\cos\beta+\sigma Y_1,\,-U\sin\beta+\sigma Y_2\right)
\]
by symmetry of \(U,Y_1,Y_2\). Therefore quadrants I and III have equal \(F_1\)-mass, and since
\(A_*^c\) is their union,
\begin{align}
 \perr
 &= 2\,\P\!\left(U\cos\beta+\sigma Y_1\ge 0,\,-U\sin\beta+\sigma Y_2\ge 0\right)
 \nonumber\\
 &= 2\,\P\!\left(\sigma Y_1\ge -U\cos\beta,\ \sigma Y_2\ge U\sin\beta\right).
 \label{eq:perr_common_1}
\end{align}
Conditioning on \(U\) and using the independence of \(Y_1\) and \(Y_2\), we obtain
\begin{align}
 \perr
 &=2\,\E\left[
 \P\left(\sigma Y_1\ge -U\cos\beta,\ \sigma Y_2\ge U\sin\beta \mid U\right)
 \right] \nonumber\\
 &=2\,\E\left[
 \P\left(Y_1\ge -\frac{U\cos\beta}{\sigma}\;\middle|\; U\right)
 \P\left(Y_2\ge \frac{U\sin\beta}{\sigma}\;\middle|\; U\right)
 \right] \nonumber\\
 &=2\,\E\left[
 \Phibar\left(-\frac{U\cos\beta}{\sigma}\right)
 \Phibar\left(\frac{U\sin\beta}{\sigma}\right)
 \right].
 \label{eq:ErrExact_common}
\end{align}

For the lower bound, define
\[
 \eta_\alpha \coloneqq \frac{1}{2\max\{\cos(\alpha/2),\sin(\alpha/2)\}},
 \qquad
 r_\sigma \coloneqq \eta_\alpha\sigma\wedge u_0.
\]
Since \(\beta=\alpha/2\), the event \(\{|U|\le r_\sigma\}\) implies \(|U|\le \eta_\alpha\sigma\), and therefore
\[
 \left|\frac{U\cos\beta}{\sigma}\right|
 \le
 \eta_\alpha\cos\beta
 \le \frac12,
 \qquad
 \left|\frac{U\sin\beta}{\sigma}\right|
 \le
 \eta_\alpha\sin\beta
 \le \frac12.
\]
Hence, on the event \(\{|U|\le r_\sigma\}\),
\[
 -\frac{U\cos\beta}{\sigma}\le \frac12,
 \qquad
 \frac{U\sin\beta}{\sigma}\le \frac12.
\]
Since \(\Phibar\) is decreasing, it follows that
\[
 \Phibar\left(-\frac{U\cos\beta}{\sigma}\right)\ge \Phibar(1/2),
 \qquad
 \Phibar\left(\frac{U\sin\beta}{\sigma}\right)\ge \Phibar(1/2),
\]
and thus
\[
 \Phibar\left(-\frac{U\cos\beta}{\sigma}\right)
 \Phibar\left(\frac{U\sin\beta}{\sigma}\right)
 \ge \Phibar(1/2)^2\,\1\{|U|\le r_\sigma\}.
\]
Substituting this into \eqref{eq:ErrExact_common} gives
\[
 \perr
 \wge
 2\,\E\left[\Phibar(1/2)^2\,\1\{|U|\le r_\sigma\}\right]
 =
 2\,\Phibar(1/2)^2\,\P(|U|\le r_\sigma).
\]
Since \(r_\sigma\le u_0\), the lower small-ball assumption yields
\[
 \P(|U|\le r_\sigma)
 =
 \nu([-r_\sigma,r_\sigma])
 \wge
 B\left(\frac{r_\sigma}{\ell}\right)^\rho.
\]
Therefore
\[
 \perr \wge 2\,\Phibar(1/2)^2\,B\left(\frac{r_\sigma}{\ell}\right)^\rho.
\]
Finally,
\[
 \frac{r_\sigma}{\ell}
 =
 \min\left\{\eta_\alpha\frac{\sigma}{\ell},\frac{u_0}{\ell}\right\}
 \wge
 \min\left\{\eta_\alpha,\frac{u_0}{\ell}\right\}
 \left(\frac{\sigma}{\ell}\wedge 1\right),
\]
so
\[
 \perr
 \wge
 2\,\Phibar(1/2)^2\,B
 \min\left\{\eta_\alpha^\rho,\left(\frac{u_0}{\ell}\right)^\rho\right\}
 \left(\frac{\sigma}{\ell}\wedge 1\right)^\rho.
\]
This proves the lemma.
\qed

\section{Concentration inequalities}
\label{ap:conc_ineq}

\subsection{Standard concentration inequalities}

We collect here standard tail bounds used in our analysis.

\paragraph{Rayleigh distribution.}
Let $X$ be Rayleigh distributed with scale parameter $\sigma > 0$. Then, for any $t > 0$, its CDF is
\begin{equation}
\P\left( X \leq t \right) = 1 - \exp\left( -\frac{t^2}{2\sigma^2} \right).
\label{eq:rayleigh}
\end{equation}

\paragraph{Chi-squared distribution.}
A standard Cram\'er--Chernoff technique utilizing the moment generating function
(e.g.\ \cite{Ghosh_2021})
implies that for a chi-squared random variable $X$ with $k$ degrees of freedom,
\begin{equation}
 \label{eq:ChiSquare}
 \P(X \ge k s)
 \wle \exp\left( -\frac{k}{2} \left( s - 1 - \log s \right) \right)
 \weq (es)^{k/2} e^{-ks/2}
 \qquad \text{for all $s \ge 1$}.
\end{equation}

\paragraph{Binomial distribution.}
Let $X = \sum_{i=1}^n X_i$ be the sum of independent Bernoulli random variables $X_i \in \{0,1\}$ with $\E[X_i] = p_i$, and let $\mu = \E[X] = \sum_{i=1}^n p_i$. Then, for any $0 \leq \delta \leq 1$,
\begin{equation}
\P\left( |X - \mu| \geq \delta \mu \right) \leq 2 \exp\left( -\frac{\delta^2 \mu}{3} \right).
\label{ineq:binomial}
\end{equation}
For a detailed discussion, see \cite[Section 5]{molloy2002graph}.

\subsection{Concentration of community sizes and triple statistics}
\label{ap:concentration_counts}
We begin by presenting two lemmas that establish concentration bounds for the community sizes and the counts of intra- and inter-community triples.

\begin{lemma}
 \label{lem:Ni_concentration}
 Let $Z_i \eqd \mathrm{Unif}\{1,2\}$ be i.i.d., and define
 \(
 N_1 \coloneqq \sum_{i=1}^N \1\{Z_i=1\}.
 \)
 For any sequence $N^{-1/2} \ll \delta_N \leq 1$,
 \begin{equation}
 \P\left(\left|N_1-\tfrac{N}{2}\right| \ge \delta_N N \right)
 \wle
 2 \exp\left(-\tfrac{2\delta_N^{2} N}{3}\right)
 = o(1).
 \label{eq:N1_tail}
 \end{equation}
\end{lemma}
\begin{proof}
By the Chernoff bound \eqref{ineq:binomial}, we have \eqref{eq:N1_tail}.
\end{proof}

\subsection{Subsampling of triples}
\label{ap:concentration_sub}
The results in this section concern the subsampling of triples and rely on the notation introduced in Section~\ref{sec:data-driven}. 

\begin{lemma}
\label{lem:TLS_distinct}
Fix \(N \ge 3\) and let \((\bZ,\bX)\) be sampled from \(\GLMM_N(\alpha,\nu,\sigma_N)\) so that the joint distribution of \((\bX_1,\dots,\bX_N)\) admits a density with respect to Lebesgue measure on \(\R^{2N}\). For each triple \(\tau=\{i_1,i_2,i_3\}\subset[N]\), let
\[
  s(\tau) \coloneqq \sigma_\tls(\bX_{i_1},\bX_{i_2},\bX_{i_3})
\]
denote its TLS residual. Then, with probability one, all TLS residuals corresponding to distinct triples are different, i.e.
\[
  \P\left(\exists\,\tau\neq\tau' \text{ with } s(\tau)=s(\tau')\right)=0.
\]
\end{lemma}

\begin{proof}
Fix two distinct triples \(\tau\neq\tau'\). Choose an index \(a\in\tau'\setminus\tau\), write
$
\bx\coloneqq \bX_a\in\R^2,
$
and collect all remaining coordinates into \(\by\in\R^{2N-2}\). Then \(s(\tau)\) depends only on \(\by\), whereas \(s(\tau')\) depends on \((\bx,\by)\). For a triple \(\tau=\{u_1,u_2,u_3\}\), let
\[
\bar{\bX}_{\tau}\coloneqq \frac13(\bX_{u_1}+\bX_{u_2}+\bX_{u_3}),
\qquad
\bS(\tau)\coloneqq \sum_{r=1}^3 (\bX_{u_r}-\bar{\bX}_{\tau})(\bX_{u_r}-\bar{\bX}_{\tau})^\top
\]
be its centered scatter matrix. By \cite{VANHUFFEL1993377},
$
s(\tau)^2=\lambda_{\min}(\bS(\tau)).
$
Since \(\bS(\tau)\) is \(2\times 2\) and symmetric,
\[
\lambda_{\min}(\bS(\tau))
=
\frac{\tr \bS(\tau)}{2}
-\frac12\sqrt{(\tr \bS(\tau))^2-4\det \bS(\tau)}.
\]
Hence, for fixed \(\by\), the map \(\bx\mapsto s(\tau')^2\) is real analytic away from the discriminant set
\[
E_{\by}\coloneqq \{\bx\in\R^2:D_{\tau'}(\bx,\by)=0\},
\qquad
D_{\tau'}(\bx,\by)\coloneqq (\tr \bS(\tau'))^2-4\det \bS(\tau').
\]

Now fix \(\by\) such that the two points indexed by \(\tau'\setminus\{a\}\) are distinct. Denote these two fixed points by \(\bY_1,\bY_2\), and let \(L\) be the line through them. Then \(D_{\tau'}(\cdot,\by)\) is not the zero polynomial: indeed, if \(\bx\in L\) and \(\bx\neq \bY_1,\bY_2\), then the three points of \(\tau'\) are collinear and not all identical, so \(\bS(\tau')\) has rank one. Thus its eigenvalues are \(\lambda_1>0\) and \(\lambda_2=0\), and
\[
D_{\tau'}(\bx,\by)
=
(\tr \bS(\tau'))^2-4\det \bS(\tau')
=
(\lambda_1-\lambda_2)^2
=
\lambda_1^2
>0.
\]
Therefore \(E_{\by}\) has Lebesgue measure zero in \(\R^2\); see, for example, \cite{mityagin2015zero}.

Since \(s(\tau)^2\) is constant as a function of \(\bx\), define on \(\R^2\setminus E_{\by}\)
\[
h_{\by}(\bx)\coloneqq s(\tau')^2-s(\tau)^2.
\]
This is real analytic, and it is nonconstant: if \(\bx\in L\setminus E_{\by}\), then \(s(\tau')=0\), whereas if \(\bx\notin L\), then the three points of \(\tau'\) are not collinear, so \(s(\tau')>0\). It follows that \(h_{\by}\) is nonconstant.

Since \(h_{\by}\) is a nonconstant real-analytic function on the open set \(\R^2\setminus E_{\by}\), its zero set there has Lebesgue measure zero. Adding back the null set \(E_{\by}\), we conclude that
\[
\{\bx\in\R^2:s(\tau')^2=s(\tau)^2\}
\]
has Lebesgue measure zero for every \(\by\) such that the two points indexed by \(\tau'\setminus\{a\}\) are distinct. The set of \(\by\) for which those two points coincide has Lebesgue measure zero in \(\R^{2N-2}\). Hence, by Fubini's theorem, the set
\[
\{(\bX_1,\dots,\bX_N)\in\R^{2N}: s(\tau')=s(\tau)\}
\]
has Lebesgue measure zero in \(\R^{2N}\). Since \((\bX_1,\dots,\bX_N)\) admits a density with respect to Lebesgue measure, it follows that
\(
\P\left(s(\tau)=s(\tau')\right)=0.
\)
There are only finitely many pairs of distinct triples, so a union bound concludes the proof.
\end{proof}
The following lemma ensures that the triples $\tau_i$ are disjoint with high probability.
\begin{lemma}
\label{lem:disjointness}
 Let $M \ll N^{1/2}$. Suppose we sample $M$ triples $\cT = (\tau_1,\dots,\tau_M)$ independently and uniformly at random from the set $\binom{[N]}{3}$. Then with high probability, all triples in $\cT$ are disjoint.
 \end{lemma}
 
 \begin{proof}
 For $1 \le i < j \le M$, define the event
 $
 E_{ij} \coloneqq \left\{\tau_i \cap \tau_j \neq \emptyset\right\},
 $
 which indicates that the $i$th and $j$th triples overlap in at least one element. We wish to bound the probability that there exists at least one pair of sampled triples that intersect.
 By the union bound,
 \begin{equation}
 \P\left(\bigcup_{1 \le i < j \le M} E_{ij}\right) \\\
 \le \sum_{1 \le i < j \le M} \P(E_{ij})
 = \binom{M}{2} \P(E_{12}),
 \label{ineq:disjoint_triples}
 \end{equation}
 where $E_{12}$ is a generic event that the first two sampled triples overlap. A straightforward counting argument shows that
 $$
 \P(E_{12}) = \frac{|\{(\tau_1,\tau_2) : \tau_1,\tau_2 \in \binom{[N]}{3},\; \tau_1\cap \tau_2 \neq \emptyset\}|}{\binom{N}{3}^2}
 =1 - \frac{\binom{N-3}{3}}{\binom{N}{3}} = O\left(\frac{1}{N}\right),
 $$
 since $\tau_2$ must share at least one of the three elements of $\tau_1$.
 Letting $M \ll N^{1/2}$ ensures $M^2 / N = o(1)$, and the right-hand side of \eqref{ineq:disjoint_triples} vanishes.
 \end{proof}

The following result shows that the empirical distribution of the sampled TLS residuals
concentrates uniformly around its mean within any deterministic margin \(\delta_M\) satisfying
\(M^{-1/2}\ll \delta_M\le 1\). The proof uses the Dvoretzky--Kiefer--Wolfowitz inequality, which we record first.

\begin{lemma}[Dvoretzky--Kiefer--Wolfowitz inequality {(see, e.g., \cite[Theorem 11.5]{kosorok2008introduction})}]\label{lem:DKW}
Let \(X_1,\dots,X_n\) be i.i.d.\ real-valued random variables with distribution function \(F\), and let
\[
F_n(t)\coloneqq \frac1n\sum_{i=1}^n \1\{X_i\le t\},
\qquad t\in\R,
\]
be the empirical distribution function. Then, for every \(x>0\),
\[
\P\left(\sup_{t\in\R}\sqrt{n}\,|F_n(t)-F(t)|>x\right)
\le 2e^{-2x^2}.
\]
\end{lemma}

\begin{lemma}
\label{lem:find_concentration}
Let \(1 \ll M \ll N^{1/2}\) and let \(M^{-1/2} \ll \delta_M \le 1\). Then
\[
\sup_{t>0}\left|F_M(t)-\E[F_M(t)]\right|\le \delta_M
\]
with high probability, where
\(
\E[F_M(t)] = \tfrac14\,p(t) + \tfrac34\,q(t).
\)
\end{lemma}

\begin{proof}
Draw \(M\) triples \(\tau_1,\dots,\tau_M\) independently and uniformly from \(\binom{[N]}{3}\), and letting $S_i \coloneqq \sigma_\tls(\tau_i),$ define
\[
A_i(t)\coloneqq \1\{S_i\le t\},
\qquad
T_i\coloneqq \1\{\tau_i \text{ is within one community}\},
\qquad
\pi_S\coloneqq \frac1M\sum_{i=1}^M T_i.
\]
Then
\(
F_M(t)=\frac1M\sum_{i=1}^M A_i(t).
\)

Let \(E_2\coloneqq\{\tau_1,\dots,\tau_M \text{ are pairwise disjoint}\}\). By Lemma~\ref{lem:disjointness},
\(
\P(E_2^c)=o(1).
\)
Conditional on \((\tau_1,\dots,\tau_M)\) and on \(E_2\), the variables \(T_1,\dots,T_M\) are i.i.d.\
Bernoulli\((1/4)\), since the triples are disjoint and the labels are i.i.d.\ \(\mathrm{Unif}\{1,2\}\). Hence Hoeffding's inequality gives
\[
\P\left(\left|\pi_S-\frac14\right|>\frac{\delta_M}{4}\,\middle|\,\tau_1,\dots,\tau_M,E_2\right)
\wle
2\exp\left(-\frac{M\delta_M^2}{8}\right)
=o(1),
\]
because \(M\delta_M^2\to\infty\). Thus the event
\[
E_1\coloneqq \left\{\left|\pi_S-\frac14\right|\le \frac{\delta_M}{4}\right\}
\]
holds with high probability on \(E_2\).

Set
\[
\cT_{\mathrm{in}}\coloneqq \{i:T_i=1\},
\qquad
\cT_{\mathrm{out}}\coloneqq \{i:T_i=0\}.
\]
On \(E_1\), we have \(|\cT_{\mathrm{in}}|=(\frac14+O(\delta_M))M\) and
\(|\cT_{\mathrm{out}}|=(\frac34+O(\delta_M))M\), so in particular
\(
|\cT_{\mathrm{in}}|,\,|\cT_{\mathrm{out}}|\gtrsim M
\)
for all sufficiently large \(N\). Define the empirical distribution functions
\[
F_{\mathrm{in}}(t)\coloneqq \frac{1}{|\cT_{\mathrm{in}}|}\sum_{i\in\cT_{\mathrm{in}}} A_i(t),
\qquad
F_{\mathrm{out}}(t)\coloneqq \frac{1}{|\cT_{\mathrm{out}}|}\sum_{i\in\cT_{\mathrm{out}}} A_i(t).
\]
Then
\(
F_M(t)=\pi_S F_{\mathrm{in}}(t)+(1-\pi_S)F_{\mathrm{out}}(t).
\)

Let \(\mathcal G_N\coloneqq \sigma(\tau_1,\dots,\tau_M,T_1,\dots,T_M)\). On \(E_2\), conditional on
\(\mathcal G_N\), the residuals \(\{S_i\}_{i\in\cT_{\mathrm{in}}}\) are i.i.d.\ with distribution function
\(p(t)\), and \(\{S_i\}_{i\in\cT_{\mathrm{out}}}\) are i.i.d.\ with distribution function \(q(t)\).
Hence, conditional on \(\mathcal G_N\), \(E_1\), and \(E_2\), Lemma~\ref{lem:DKW} yields
\[
\P\left(\sup_{t>0}|F_{\mathrm{in}}(t)-p(t)|>\frac{\delta_M}{4}\,\middle|\,\mathcal G_N,E_1,E_2\right)=o(1),
\]
and similarly
\[
\P\left(\sup_{t>0}|F_{\mathrm{out}}(t)-q(t)|>\frac{\delta_M}{4}\,\middle|\,\mathcal G_N,E_1,E_2\right)=o(1),
\]
since \(|\cT_{\mathrm{in}}|,|\cT_{\mathrm{out}}|\gtrsim M\) on \(E_1\) and \(M\delta_M^2\to\infty\).

By the law of total expectation,
\(
\E[F_M(t)]=\tfrac14\,p(t)+\tfrac34\,q(t).
\)
Moreover,
\[
F_M(t)-\E[F_M(t)]
=
\pi_S\bigl(F_{\mathrm{in}}(t)-p(t)\bigr)
+
(1-\pi_S)\bigl(F_{\mathrm{out}}(t)-q(t)\bigr)
+
\left(\pi_S-\frac14\right)\bigl(p(t)-q(t)\bigr).
\]
Since \(|p(t)-q(t)|\le 1\), it follows that on the event
\[
E_1\cap E_2\cap
\left\{\sup_{t>0}|F_{\mathrm{in}}(t)-p(t)|\le \frac{\delta_M}{4}\right\}
\cap
\left\{\sup_{t>0}|F_{\mathrm{out}}(t)-q(t)|\le \frac{\delta_M}{4}\right\},
\]
we have, uniformly in \(t>0\),
\[
|F_M(t)-\E[F_M(t)]|
\le
\pi_S\frac{\delta_M}{4}
+
(1-\pi_S)\frac{\delta_M}{4}
+
\left|\pi_S-\frac14\right|
\le \frac{\delta_M}{2}
\le \delta_M.
\]
Therefore, by the union bound,
\begin{align*}
\P\left(\sup_{t>0}|F_M(t)-\E[F_M(t)]|>\delta_M\right)
&\wle \P(E_2^c)+\P(E_1^c,E_2) \\
&\quad + \P\left(\sup_{t>0}|F_{\mathrm{in}}(t)-p(t)|>\frac{\delta_M}{4},\,E_1,E_2\right) \\
&\quad + \P\left(\sup_{t>0}|F_{\mathrm{out}}(t)-q(t)|>\frac{\delta_M}{4},\,E_1,E_2\right),
\end{align*}
and each term on the right-hand side is \(o(1)\). This proves the claim.
\end{proof}

The next lemma provides a condition on the threshold sequence that ensures the fraction of observed triples concentrates around the expected number of intra-community triples, $1/4$, and demonstrates that such a vanishing sequence exists. 
\begin{lemma}
\label{lem:expectation_gap}
Assume \(\sigma_N=o(1)\) and \(1\ll M\ll N^{1/2}\). Then:
\begin{enumerate}[label=(\roman*)]
\item There exist deterministic sequences \(s_N,\eta_N>0\) with \(s_N=o(1)\) and \(\eta_N=o(1)\) such that
\(
\left|F_M(s_N)-\tfrac14\right|\le \eta_N
\)
with high probability.

\item For every fixed \(\varepsilon>0\), there exists a constant \(c_\varepsilon>0\) such that, with high probability,
\[
\inf_{t\ge \varepsilon}\left|F_M(t)-\tfrac14\right|\ge c_\varepsilon.
\]
\end{enumerate}
\end{lemma}

\begin{proof}
Fix a deterministic sequence \(\delta_M=\delta_{M_N}\) such that
\[
M^{-1/2}\ll \delta_M\ll 1,
\]
for example \(\delta_M=M^{-1/3}\). Since \(M=M_N\to\infty\), we have \(\delta_M=o(1)\).

We first prove (i). Let \(s_N\coloneqq \sqrt{\sigma_N\ell}\). Then \(s_N=o(1)\) and
\[
\frac{\sigma_N}{s_N}=\sqrt{\frac{\sigma_N}{\ell}}=o(1),
\qquad
\frac{s_N}{\ell}=\sqrt{\frac{\sigma_N}{\ell}}=o(1),
\]
so \(\sigma_N\ll s_N\ll \ell\). In particular, \((s_N+\sigma_N)/\ell=o(1)\), hence \((s_N+\sigma_N)/\ell\le 1/(2e)\) for all sufficiently large \(N\). Therefore Lemma~\ref{lem:_within_probability} and Corollary~\ref{cor:coordinate_smallball_common} imply \(1-p(s_N)=o(1)\) and \(q(s_N)=o(1)\). Thus
\[
\left|\E[F_M(s_N)]-\tfrac14\right|
=
\left|\tfrac34\,q(s_N)-\tfrac14\bigl(1-p(s_N)\bigr)\right|
\le \tfrac34\,q(s_N)+\tfrac14\bigl(1-p(s_N)\bigr)
=o(1).
\]
Define \(\eta_N \coloneqq \delta_M+\left|\E[F_M(s_N)]-\tfrac14\right|\). Then \(\eta_N=o(1)\). By Lemma~\ref{lem:find_concentration},
\[
\sup_{t>0}|F_M(t)-\E[F_M(t)]|\le \delta_M
\]
with high probability. In particular,
\[
\left|F_M(s_N)-\tfrac14\right|
\le
|F_M(s_N)-\E[F_M(s_N)]|
+
\left|\E[F_M(s_N)]-\tfrac14\right|
\le \eta_N
\]
with high probability. This proves (i).

We next prove (ii). Fix \(\varepsilon>0\). Since \(q(t)\) is nondecreasing in \(t\), Lemma~\ref{lem:q_lower_bound} gives \(q(t)\ge q(\varepsilon)\) for all \(t\ge \varepsilon\). Since \(\sigma_N=o(1)\), there exists a constant \(c_\varepsilon>0\) such that \(q(t)\ge c_\varepsilon\) for all \(t\ge \varepsilon\) and all sufficiently large \(N\). Similarly, \(p(t)\) is nondecreasing in \(t\), so \(1-p(t)\le 1-p(\varepsilon)\) for all \(t\ge \varepsilon\). Since \(\varepsilon/\sigma_N\to\infty\), Lemma~\ref{lem:_within_probability} yields \(1-p(\varepsilon)=o(1)\). Hence, after decreasing \(c_\varepsilon\) if necessary, we have \(1-p(t)\le c_\varepsilon\) for all \(t\ge \varepsilon\) and all sufficiently large \(N\). Therefore, for all \(t\ge \varepsilon\),
\[
\left|\E[F_M(t)]-\tfrac14\right|
=
\left|\tfrac34\,q(t)-\tfrac14\bigl(1-p(t)\bigr)\right|
\ge \tfrac34\,c_\varepsilon-\tfrac14\,c_\varepsilon
=
\frac{c_\varepsilon}{2}
\]
for all sufficiently large \(N\).

Since \(\delta_M=o(1)\), we also have \(\delta_M\le c_\varepsilon/4\) for all sufficiently large \(N\). On the high-probability event from Lemma~\ref{lem:find_concentration}, \(\sup_{t>0}|F_M(t)-\E[F_M(t)]|\le \delta_M\), and hence, uniformly for all \(t\ge \varepsilon\),
\[
\left|F_M(t)-\tfrac14\right|
\ge
\left|\E[F_M(t)]-\tfrac14\right|
-
|F_M(t)-\E[F_M(t)]|
\ge \frac{c_\varepsilon}{2}-\delta_M
\ge \frac{c_\varepsilon}{4}.
\]
After replacing \(c_\varepsilon\) by \(c_\varepsilon/4\), this proves (ii).
\end{proof}

Finally, the following lemma gives a deterministic bound on the hyperedge probabilities when the difference 
$
\E[F_M(t_N)]-\tfrac14
$
is bounded.
\begin{lemma}
\label{lem:selection_rule}
Define $G_N(t)\coloneqq \tfrac14\,p(t)+\tfrac34\,q(t).
$ Let \(\sigma_N,t_N,\varepsilon_N=o(1)\) such that
\begin{equation}
\left|G_N(t_N) - \tfrac{1}{4}\right| \lesssim \varepsilon_N,
\label{eq:lin_comb}
\end{equation}
and
\begin{equation}
\sigma_N \ll \ell\left(\frac{\varepsilon_N}{\log(1/\varepsilon_N)}\right)^{1/\rho}\log^{-1/2}(1/\varepsilon_N).
\label{eq:sigma_ass_eps}
\end{equation}
Then
\[
q(t_N) \lesssim \varepsilon_N
\qquad\text{and}\qquad
1-p(t_N)\lesssim \varepsilon_N.
\]
\end{lemma}

\begin{proof}
Suppose, to the contrary, that the conclusion fails. Then at least one of the two bounds
\[
q(t_N)\lesssim \varepsilon_N,
\qquad
1-p(t_N)\lesssim \varepsilon_N
\]
must fail. We claim that under \eqref{eq:lin_comb}, each of these bounds implies the other.
Indeed, if \(q(t_N)=O(\varepsilon_N)\), then
\[
\frac14\left(1-p(t_N)\right)
=
\frac34\,q(t_N)+O(\varepsilon_N)
=
O(\varepsilon_N),
\]
so \(1-p(t_N)=O(\varepsilon_N)\). Conversely, if \(1-p(t_N)=O(\varepsilon_N)\), then
\[
\frac34\,q(t_N)
=
\frac14\left(1-p(t_N)\right)+O(\varepsilon_N)
=
O(\varepsilon_N),
\]
so \(q(t_N)=O(\varepsilon_N)\).
Hence, we may suppose
\begin{equation}
q(t_N)\gg \varepsilon_N
\qquad\text{and}\qquad
1-p(t_N)\gg \varepsilon_N.
\label{eq:contrary}
\end{equation}

Let
\(
s_N \coloneqq t_N+\sigma_N.
\)
Since \(t_N,\sigma_N=o(1)\), we have \(s_N=o(1)\). In particular, because \(\ell>0\) is fixed,
\(
\frac{s_N}{\ell}\le \frac{1}{2e}
\)
for all sufficiently large \(N\). Hence Corollary~\ref{cor:coordinate_smallball_common} applies and gives
\[
q(t_N)\lesssim \left(\frac{s_N}{\ell}\right)^\rho \log\left(\frac{\ell}{s_N}\right).
\]
Combining this with \eqref{eq:contrary} yields
\begin{equation}
\left(\frac{s_N}{\ell}\right)^\rho \log\left(\frac{\ell}{s_N}\right)\gg \varepsilon_N.
\label{eq:sn_pre_lower}
\end{equation}

Define
\(
g(u)\coloneqq u^\rho \log(1/u),
\)
for $0<u<e^{-1/\rho}$.
Since
\(
g'(u)=u^{\rho-1}\left(\rho\log(1/u)-1\right),
\)
the function \(g\) is strictly increasing on \((0,e^{-1/\rho})\). Also, if we set
\[
u_N\coloneqq \left(\frac{\varepsilon_N}{\log(1/\varepsilon_N)}\right)^{1/\rho},
\]
then \(u_N\in(0,e^{-1/\rho})\) for all sufficiently large \(N\). Moreover,
\[
\log\left(\frac{1}{u_N}\right)
=
\frac{1}{\rho}\log\left(\frac{\log(1/\varepsilon_N)}{\varepsilon_N}\right)
=
\frac{1}{\rho}\Bigl(\log\log(1/\varepsilon_N)+\log(1/\varepsilon_N)\Bigr).
\]
Therefore
\begin{align*}
g(u_N)
&=
u_N^\rho \log\left(\frac{1}{u_N}\right)\\
&=
\frac{\varepsilon_N}{\log(1/\varepsilon_N)}
\cdot
\frac{1}{\rho}\Bigl(\log\log(1/\varepsilon_N)+\log(1/\varepsilon_N)\Bigr)\\
&=
\frac{\varepsilon_N}{\rho}
\left(
1+\frac{\log\log(1/\varepsilon_N)}{\log(1/\varepsilon_N)}
\right)
\asymp \varepsilon_N.
\end{align*}
Since \eqref{eq:sn_pre_lower} states that \(g(s_N/\ell)\gg \varepsilon_N\), we have \(g(s_N/\ell)\gg \varepsilon_N \asymp g(u_N)\) with $s_N, \varepsilon_N \ll 1$, so monotonicity of \(g\) on \((0,e^{-1/\rho})\) implies that
\begin{equation}
s_N \gtrsim \ell u_N = \ell\left(\frac{\varepsilon_N}{\log(1/\varepsilon_N)}\right)^{1/\rho}.
\label{eq:sN_bound}
\end{equation}

By the assumed \(\sigma_N\)-bound \eqref{eq:sigma_ass_eps}, \eqref{eq:sN_bound} yields
\[
\frac{s_N}{\sigma_N}
\gg
\log^{1/2}(1/\varepsilon_N)\to\infty.
\]
Hence \(s_N\gg \sigma_N\), and together with \eqref{eq:sN_bound},
\begin{equation}
t_N=s_N-\sigma_N \asymp s_N \gtrsim \ell\left(\frac{\varepsilon_N}{\log(1/\varepsilon_N)}\right)^{1/\rho}.
\label{eq:tn_sn_equiv}
\end{equation}
In particular, \(t_N>\sqrt{3}\,\sigma_N\) for all sufficiently large \(N\), so Lemma~\ref{lem:_within_probability} applies.

Now \eqref{eq:contrary} and Lemma~\ref{lem:_within_probability} give
\[
\exp\left(
-\frac{3}{2}\left(\frac{t_N^2}{3\sigma_N^2}-1-\log\frac{t_N^2}{3\sigma_N^2}\right)
\right)
\wge 1-p(t_N)\gg \varepsilon_N.
\]
Define
\[
x_N\coloneqq \frac{t_N^2}{3\sigma_N^2},
\qquad
f(x)\coloneqq x-1-\log x.
\]
Then
\[
\exp\left(-\frac32 f(x_N)\right)\gg \varepsilon_N,
\]
and
\(
f(x_N)\lesssim \log(1/\varepsilon_N).
\)
Since \(\log x\le x/2\), \(x>0\), we have
\(
f(x)=x-1-\log x \ge (x/2)-1.
\)
Therefore
\[
x_N \lesssim 2\log(1/\varepsilon_N)+2,
\]
and since $\varepsilon_N \ll 1$, we have $\log(1/\varepsilon_N) \to \infty$, and the definition of $x_N$ gives
\begin{equation}
t_N^2 \lesssim \sigma_N^2 \log(1/\varepsilon_N).
\label{eq:tn_ub}
\end{equation}

Combining \eqref{eq:tn_sn_equiv} and \eqref{eq:tn_ub}, we obtain
\[
\ell\left(\frac{\varepsilon_N}{\log(1/\varepsilon_N)}\right)^{1/\rho}
\lesssim t_N
\lesssim \sigma_N \log^{1/2}(1/\varepsilon_N),
\]
which forces
\[
\sigma_N \gtrsim \ell\left(\frac{\varepsilon_N}{\log(1/\varepsilon_N)}\right)^{1/\rho}\log^{-1/2}(1/\varepsilon_N).
\]
This contradicts the assumption \eqref{eq:sigma_ass_eps}. Therefore \eqref{eq:contrary} is impossible, and
\[
q(t_N)\lesssim \varepsilon_N
\qquad\text{and}\qquad
1-p(t_N)\lesssim \varepsilon_N.
\qedhere
\]
\end{proof}

\section{Geometric analysis}
\label{ap:geom_tools}

In this appendix, we state and prove a geometric result used in the proof of
Lemma~\ref{lem:prob_between}. Consider non-overlapping circles \(C_1\) and \(C_2\) with centers on
the \(x\)-axis, and a third circle \(C_3\) whose center lies on the line
\[
\{(y\cot\alpha,y):y\in\R\},
\]
where \(\alpha\in(0,\pi)\) is fixed. The following result gives an upper bound on the vertical
coordinate of the center of \(C_3\), assuming that a single line intersects all three circles (Figure~\ref{fig:tangents}).

\begin{figure}[htbp!]
\centering
\begin{tikzpicture}[scale=3.5]

% Axes
\draw[->, white] (-1.5, 0) -- (.6, 0);% node[below] {$x$};
\draw[->, white] (0, -.4) -- (0, .4);% node[left] {$y$};

% Variables
\def\rone{0.2}
\def\rtwo{0.1}
\def\rthree{0.185}
\def\mone{-1.106}
\def\mtwo{-0.272}
\def\k{0.3855}  % k value
\def\h{0.1208}  % h value
\def\p{-0.55}   % p value
\def\q{0.562}   % q value
\def\b{0.4103}  % b value

% Circle 3
\draw[thick, black] (0, 0.4103-0.05) circle[radius=\rthree];
\node at (0, 0.4103-0.05) {\footnotesize $C_3$};

% Circle 1
\draw[thick, black] (\mone, 0) circle[radius=\rone];
\node at (\mone, 0) {\footnotesize $C_1$};

% Circle 2
\draw[thick, black] (\mtwo, 0) circle[radius=\rtwo];
\node at (\mtwo, 0) {\footnotesize $C_2$};

\draw[thick,dashed,domain=-1.5:.55] plot(\x,{ \k*(\x-\p) });
\draw[thick,dashed,domain=-1.5:.55] plot(\x,{- \k*(\x-\p) });
\draw[thick,dotted,domain=-1.5:.55] plot(\x,{ \h*(\x-\q) });
\draw[thick,dotted,domain=-1.5:.55] plot(\x,{- \h*(\x-\q) });
% Envelope (pointwise maximum)
%\draw[line width=1.2pt,myred,domain=-1.5:.55] 
%plot(\x,{
%max(
%max(\k*(\x-\p),-\k*(\x-\p)),
%max(\h*(\x-\q),-\h*(\x-\q))
%)});

% Vertical line segment with orthogonal bars
\def\bv{0.4103}  % b value
%\draw[line width=1.2pt, myred, -|] (0, -.7) -- (0, \bv);
\node[left] at (-0.012, \bv) {};
%\node[left] at (-0.01, -\bv) {\footnotesize $-b$};
\end{tikzpicture}
\caption{ 
Transverse (dashed) and direct (dotted) common tangents
of circles $C_1$ and~$C_2$.
Circle $C_3$ lies low enough so that a line intersects all three circles.
}
\label{fig:tangents}
\end{figure}

\begin{lemma}
\label{lem:U3_vertical_bound_angle}
Fix \(\alpha\in(0,\pi)\). Consider circles
\(C_i\) with centers \(P_i=(x_i,y_i)\) and radii \(r_i>0\), \(i=1,2,3\).
We assume that \(y_1=y_2=0\), that
$
P_3=(v_3\cos\alpha,v_3\sin\alpha),
$
and that
$
|x_1-x_2|>r_1+r_2
$
(circles \(C_1\) and \(C_2\) do not overlap).
Assume that there exists a line that intersects the circles \(C_1,C_2,C_3\).
Then, whenever
\[
K_{12}|\cot\alpha|<1,
\]
we have
\begin{equation}
 |v_3|
 \wle
 \frac{
 \frac{|x_1|+|x_2|}{2}\,K_{12}
 +
 \left(\frac{r_1+r_2}{2}+r_3\right)(1+K_{12})
 }{
 \sin\alpha\,\left(1-K_{12}|\cot\alpha|\right)
 },
 \label{eq:v3_angle_bound}
\end{equation}
where
\begin{equation}
 K_{12}
 \weq
 \frac{r_1+r_2}{\sqrt{(x_1-x_2)^2-(r_1+r_2)^2}}.
 \label{eq:CircleIntersectionK12}
\end{equation}
In particular, if
\[
K_{12}\le \kappa_\alpha,
\qquad\text{where}\qquad
\kappa_\alpha
\weq
\begin{cases}
(2|\cot\alpha|)^{-1}, & \alpha\neq \pi/2,\\
\infty, & \alpha=\pi/2,
\end{cases}
\]
then
\begin{equation}
 |v_3|
 \wle
 \frac{2}{\sin\alpha}
 \left[
 \frac{|x_1|+|x_2|}{2}\,K_{12}
 +
 \left(\frac{r_1+r_2}{2}+r_3\right)(1+K_{12})
 \right].
 \label{eq:v3_angle_bound_simple}
\end{equation}
\end{lemma}
\begin{proof}
Assume that a line \(\cL\) intersects all the circles \(C_1,C_2,C_3\).
Since \(C_1\) and \(C_2\) are disjoint and have centers on the \(x\)-axis,
\(\cL\) cannot be vertical.
Then we may parameterize the line as
\[
\cL=\{(x,y):y=\gamma x+\beta\}
\]
for some numbers \(\gamma,\beta\in\R\).
The distance from \(P_i\) to \(\cL\) can be written as
\[
 d(P_i,\cL)
 \weq \frac{|y_i-\gamma x_i-\beta|}{\sqrt{1+\gamma^2}}.
\]

Since \(\cL\) intersects \(C_1\) and \(C_2\), we see that
\begin{equation}
 \label{eq:CircleIntersectionC1C2_angle}
 \frac{|\gamma x_1+\beta|}{\sqrt{1+\gamma^2}} \le r_1
 \quad \text{and} \quad
 \frac{|\gamma x_2+\beta|}{\sqrt{1+\gamma^2}} \le r_2.
\end{equation}
By the triangle inequality,
\[
 |\beta|
 \wle |-\gamma x_i|+|\gamma x_i+\beta|
 \wle |\gamma x_i|+r_i\sqrt{1+\gamma^2}
\]
for \(i=1,2\). By averaging these inequalities, we obtain
\begin{equation}
\label{eq:beta_angle_bound}
 |\beta|
 \wle \frac{|x_1|+|x_2|}{2}\,|\gamma|
 + \frac{r_1+r_2}{2}\sqrt{1+\gamma^2}.
\end{equation}

Since \(\cL\) also intersects \(C_3\), whose center is
\(
P_3=(v_3\cos\alpha,v_3\sin\alpha),
\)
we see that
\[
 \frac{|v_3\sin\alpha-\gamma v_3\cos\alpha-\beta|}{\sqrt{1+\gamma^2}} \le r_3.
\]
Hence
\[
 |v_3(\sin\alpha-\gamma\cos\alpha)|
 \wle |\beta|+r_3\sqrt{1+\gamma^2}.
\]
By combining this with \eqref{eq:beta_angle_bound}, we obtain
\[
 |v_3|\,|\sin\alpha-\gamma\cos\alpha|
 \wle \frac{|x_1|+|x_2|}{2}\,|\gamma|
 + \left(\frac{r_1+r_2}{2}+r_3\right)\sqrt{1+\gamma^2}.
\]
Applying \(\sqrt{1+\gamma^2}\le 1+|\gamma|\), we get
\begin{equation}
\label{eq:v3_gamma_bound}
 |v_3|\,|\sin\alpha-\gamma\cos\alpha|
 \wle \frac{|x_1|+|x_2|}{2}\,|\gamma|
 + \left(\frac{r_1+r_2}{2}+r_3\right)(1+|\gamma|).
\end{equation}

Inequality \eqref{eq:CircleIntersectionC1C2_angle} also implies that
\(
 |\gamma x_1+\beta| \le r_1\sqrt{1+\gamma^2},
\)
and
\(
 |\gamma x_2+\beta| \le r_2\sqrt{1+\gamma^2},
\)
so that
\[
 |\gamma(x_1-x_2)|
 = \left|(\gamma x_1+\beta)-(\gamma x_2+\beta)\right|
 \le |\gamma x_1+\beta|+|\gamma x_2+\beta|
 \le (r_1+r_2)\sqrt{1+\gamma^2}.
\]
By squaring both sides of the above inequality and rearranging, we see that
\[
 \gamma^2\left((x_1-x_2)^2-(r_1+r_2)^2\right)\le (r_1+r_2)^2.
\]
Since \(C_1\) and \(C_2\) are disjoint, we have
\(
 |x_1-x_2|>r_1+r_2,
\)
so the factor
\(
 (x_1-x_2)^2-(r_1+r_2)^2
\)
is strictly positive. Dividing by this factor and taking square roots gives
\begin{equation}
\label{eq:gamma_bound_by_K12}
 |\gamma|
 \wle
 \frac{r_1+r_2}{\sqrt{(x_1-x_2)^2-(r_1+r_2)^2}}
 \weq K_{12}.
\end{equation}

Finally, since \(\alpha\in(0,\pi)\), we have \(\sin\alpha>0\), and therefore
\[
 |\sin\alpha-\gamma\cos\alpha|
 \ge \sin\alpha-|\gamma||\cos\alpha|
 = \sin\alpha\left(1-|\gamma||\cot\alpha|\right)
 \ge \sin\alpha\left(1-K_{12}|\cot\alpha|\right),
\]
where the last inequality uses \eqref{eq:gamma_bound_by_K12}.
If \(K_{12}|\cot\alpha|<1\), then the right-hand side is strictly positive, and combining this with
\eqref{eq:v3_gamma_bound} and \eqref{eq:gamma_bound_by_K12} yields
\[
 |v_3|
 \wle
 \frac{
 \frac{|x_1|+|x_2|}{2}\,K_{12}
 +
 \left(\frac{r_1+r_2}{2}+r_3\right)(1+K_{12})
 }{
 \sin\alpha\,\left(1-K_{12}|\cot\alpha|\right)
 }.
\]
This proves \eqref{eq:v3_angle_bound}.

For the final claim, assume that \(K_{12}\le \kappa_\alpha\). If \(\alpha=\pi/2\), then
\(
K_{12}|\cot\alpha|=0.
\)
If \(\alpha\neq \pi/2\), then by the definition of \(\kappa_\alpha\),
\(
K_{12}|\cot\alpha|\le \frac12.
\)
Hence in all cases,
\(
1-K_{12}|\cot\alpha|\ge \frac12.
\)
Substituting this into \eqref{eq:v3_angle_bound} yields
\[
 |v_3|
 \wle
 \frac{2}{\sin\alpha}
 \left[
 \frac{|x_1|+|x_2|}{2}\,K_{12}
 +
 \left(\frac{r_1+r_2}{2}+r_3\right)(1+K_{12})
 \right],
\]
which is \eqref{eq:v3_angle_bound_simple}. This concludes the proof.
\end{proof}

\addcontentsline{toc}{section}{References}
\bibliographystyle{splncs04}
\bibliography{my}
%\bibliography{my,refs_line_clus}

\end{document}